\documentclass[12pt]{amsart}
\usepackage{amsmath,amssymb,amsthm,tikz}
\usepackage[letterpaper,margin=1in]{geometry}
\usepackage{comment}
\usepackage[centertableaux]{ytableau}

\usepackage[colorlinks=true, pdfstartview=FitV, linkcolor=blue, citecolor=blue, urlcolor=blue]{hyperref}

\newcommand{\arxiv}[1]{\href{https://arxiv.org/abs/#1}{\texttt{arXiv:#1}}}

\newtheorem{theorem}{\textbf{Theorem}}[section]
\newtheorem{proposition}[theorem]{\textbf{Proposition}}
\newtheorem{provisional}[theorem]{\textbf{Provisional Assumption}}
\newtheorem{corollary}[theorem]{\textbf{Corollary}}

\newtheorem{lemma}[theorem]{\textbf{Lemma}}

\theoremstyle{definition}
\newtheorem{remark}[theorem]{Remark}
\newtheorem{definition}[theorem]{Definition}
\newtheorem{problem}[theorem]{Problem}
\theoremstyle{remark}
\newtheorem{example}[theorem]{Example}

\newcommand{\ket}[1]{\lvert #1 \rangle}
\newcommand{\bra}[1]{\langle #1 \rvert}
\newcommand{\braket}[2]{\langle #1 | #2 \rangle}
\newcommand{\abs}[1]{\lvert #1 \rvert}
\newcommand{\normord}[1]{ {: \mathrel{#1} :} }  
\newcommand{\bnodelim}{\raisebox{1pt}{$\genfrac{}{}{0pt}{2}{\ast}{\ast}$}}
\newcommand{\bnormord}[1]{ {\bnodelim\!\! \mathrel{#1} \!\!\bnodelim} }  
\newcommand{\dv}{|\!|}
\newcommand{\FPS}[1]{[\![#1]\!]}  
\newcommand{\FLS}[1]{(\!(#1)\!)}  
\newcommand{\RTM}{\mathbf{T}}  
\newcommand{\one}{\mathbf{1}}
\newcommand{\iso}{\cong}

\newcommand{\fermionfock}{\mathfrak{F}}
\newcommand{\bosonfock}{\mathfrak{B}}

\newcommand{\ii}{\mathbf{i}}
\newcommand{\pp}{\mathbf{p}}
\newcommand{\bt}{\mathbf{t}}
\newcommand{\uu}{\mathbf{u}}
\newcommand{\xx}{\mathbf{x}}
\newcommand{\yy}{\mathbf{y}}
\newcommand{\zz}{\mathbf{z}}
\newcommand{\DD}{\mathbf{D}}

\newcommand{\bal}{\boldsymbol{\alpha}}
\newcommand{\bbe}{\boldsymbol{\beta}}
\newcommand{\bbb}{\mathsf{b}}
\newcommand{\svar}{\mathsf{s}}

\newcommand{\gl}{\mathfrak{gl}}

\newcommand{\mcC}{\mathcal{C}}

\newcommand{\mcH}{\mathcal{H}}

\newcommand{\mcP}{\mathcal{P}}
\newcommand{\mcW}{\mathcal{W}}

\newcommand{\ZZ}{\mathbb{Z}}

\newcommand{\CC}{\mathbb{C}}
\newcommand{\End}{\operatorname{End}}
\newcommand{\res}{\operatorname{res}}
\newcommand{\htf}{\operatorname{ht}}  

\DeclareMathOperator{\Cl}{Cl}  
\DeclareMathOperator{\wt}{wt}  

\renewcommand{\dotsc}{\cdots}  

\definecolor{darkred}{rgb}{0.7,0,0} 
\newcommand{\defn}[1]{{\color{darkred}\emph{#1}}} 

\usepackage[colorinlistoftodos]{todonotes}

\newcommand{\travis}[1]{\todo[size=\tiny,color=blue!30]{#1 \\ \hfill --- Travis}}

\begin{document}

\title{Factorial Fock free fermions}

\author{Daniel Bump}
\address[D.~Bump]{Department of Mathematics, Stanford University, Stanford, CA 94305-2125}
\email{bump@math.stanford.edu}
\urladdr{https://math.stanford.edu/~bump/}

\author{Andrew Hardt}
\address[A.~Hardt]{Department of Mathematics, University of Illinois Urbana-Champaign, Urbana, IL 61801}
\email{ahardt@illinois.edu}
\urladdr{https://andyhardt.github.io/}

\author{Travis Scrimshaw}
\address[T.~Scrimshaw]{Department of Mathematics, Hokkaido University, 5 Ch\=ome Kita 8 J\=onishi, Kita Ward, Sapporo, Hokkaid\=o 060-0808}
\email{tcscrims@gmail.com}
\urladdr{https://tscrim.github.io/}


\maketitle

\begin{abstract}
We use a double shifted power analog of free fermion fields to introduce current operators, Hamiltonians, and vertex operators which are deformed by two families of parameters and satisfy analogous formulas to the classical case.
We show that the deformed half vertex operators correspond to the row transfer matrices of a solvable six vertex model recently given by Naprienko, which under a specialization yields the factorial Schur functions (up to a reindexing of parameters).
As a consequence, we show that under the boson-fermion correspondence using our deformed half vertex operators, the natural basis (under this specialization) maps to the double factorial Schur functions.
Furthermore, the image of the natural basis vectors are tau function solutions to the 2D Toda lattice.
\end{abstract}

\setcounter{tocdepth}{1}
\tableofcontents

\section{Introduction}

The Schur functions $s_{\lambda}$ are an important class of functions that appear in a broad range of mathematics, including representation theory, combinatorics, probability theory, and mathematical physics.
Each of these areas often have their own construction, such as the characters of the general linear Lie algebra $\gl_n (\CC)$ or as generating functions of semistandard tableaux, typically with multiple proofs showing these descriptions are equivalent.
Some influential references and extensions include~\cite{BBF11,HoweSchurLectures,Lenart00,Macdonald92,MacdonaldBook,Okounkov01,OR03,SaganSymReprBook,ECII}, with this list being far from complete.

Following~\cite{KacInfinite,KacRaina}, let us describe the construction of $s_{\lambda}$ using the \textit{boson-fermion correspondence} (with countably infinite bosons and fermions), an isomorphism between two Hilbert spaces associated with seemingly different physical systems that originated in the work of Skyrme~\cite{Skyrme61} (some other historical references include~\cite{Coleman75,FrenkelBFC,FK80,Mandelstam75,SW70}).
The \textit{fermionic Fock space} $\fermionfock$ originated with Dirac's electron sea~\cite{DiracElectrons}.
It models a system of (fermionic) particles that can occupy different positions, though no two can occupy the same position (the Pauli exclusion principle).
Interchanging two particles changes the sign of the representing vector, so $\fermionfock$ can be understood as something like an exterior algebra.
More precisely, if $v_i$ represents a particle at position $i$, then eventually one always sees an occupied (resp.\ unoccupied) position when moving in the negative (resp.\ positive direction).
It follows that a basis consists of semi-infinite monomials
\[
\ket{\lambda}_m = v_{m + \lambda_1} \wedge v_{m - 1 + \lambda_2} \wedge \cdots,
\]
where $\lambda = (\lambda_1, \lambda_2, \cdots)$ is a (integer) partition (with finite sum as usual) and $m \in \ZZ$ is called the \textit{charge}.
Let $\fermionfock_m$ be the subspace of charge $m$.
The other Fock space is the \textit{bosonic Fock space} $\bosonfock = \CC [\svar, \svar^{-1}] \otimes \Lambda$, where $\Lambda$ is the ring of symmetric functions (often considered as a countably generated polynomial ring in the powersum symmetric functions) and $\svar$ is an indeterminant.
The boson-fermion correspondence is an isomorphism of $\fermionfock$ with $\bosonfock$ as infinite dimensional Heisenberg Lie algebra representations in which $\fermionfock_m$ corresponds to $\svar^m \otimes \bosonfock$ with the natural operations on one space to be constructed from the other.
A consequence is that under this isomorphism $\ket{\lambda}_m$ is mapped to $\svar^m \otimes s_{\lambda}$.
Additional references for the boson-fermion correspondence regarding the relationship with vertex operators and vertex algebras include~\cite[Thm.~5.3.2]{FBZ04} and~\cite[Thm.~5.2]{KacVertex}.

Another construction of $s_{\lambda}$ comes from the free fermionic six vertex model, which are given by six local configurations that preserve the particles in the system which have weights that satisfy a particular quadratic relation.
The study of these models originated in the work on Linus Pauling in 1935 to account for the residual entropy in ice, with many generalizations appearing in the literature.
Here, we focus on solvable lattice models that we call \textit{Tokuyama models}, where the partition function of a model is $s_{\lambda}$ optionally multiplied by a deformation of the Weyl denominator, \textit{e.g.},~\cite{hkice,HamelKingBijective,ZinnJustinTiling}.
We consider such a model on the grid with columns indexed by $\ZZ$ and a finite number of rows.
We may regard the column data as elements of $\fermionfock$, and so each layer is a graphical representation of a \textit{row transfer matrix} that maps $\fermionfock_m \to \fermionfock_m$ and has a spectral parameter that is one of the parameters of the Schur function.
It may be seen (\textit{e.g.},~\cite{BBBGVertex,BrubakerSchultzHamiltonians,HardtHamiltonians,KorffHecke,ZinnJustinTiling}) that sometimes the row transfer matrix may be expressed as a \textit{half vertex operator} constructed with the current operators, which encode the motion of particles.
In particular, this fact is closely connected with the boson-fermion correspondence, where this half vertex operator is used to describe the isomorphism.

Biedenharn and Louck introduced the factorial Schur functions~\cite{BL89}, with the name coming from the use of the falling factorial.
However, here we will use the interpretation developed by Macdonald~{\cite[6th Variation]{Macdonald92}}, where they are considered as an inhomogeneous deformation of the Schur functions by a set of parameters $\bal \:= (\alpha_i)_{i \in \ZZ}$.
By specializing $a_i = 1 - i$, we recover the shifted Schur functions introduced by Okounkov and Olshanski~\cite{OO97} that were used to give a basis for $Z(U(\gl_n))$, the center of universal enveloping algebra of $\gl_n$, with applications to the asymptotic characters of $GL(\infty)$. 
The factorial Schur functions also have a geometric interpretation as representing Schubert classes in the torus equivariant cohomology of the Grassmannian (see, \textit{e.g.},~\cite{KnutsonTaoPuzzles}), where the $\bal$ parameters correspond to the torus parameters.
(This was essentially known to Lascoux and Sch{\"u}tzenberger; see remarks in~\cite{Mihalcea08,MolevSagan}.)
The factorial Schur functions have a free fermionic six vertex model description due to Zinn-Justin~\cite{ZJ09} and Bump--McNamara--Nakasuji~\cite{BMN14} as a deformation of the classical free fermionic five and six vertex models for Schur functions and natural correspond to the tableau description.
There are also generalizations~\cite{MolevFactorialSupersymmetric,Molev09} with lattice model interpretations~\cite{NaprienkoFFS}.
So a natural question is to describe current operators on $\fermionfock$ that generalize the boson-fermion correspondence to the case of factorial Schur functions, and also to describe lattice models whose row transfer matrices correspond to the half vertex operators that appear in the theory.
However, since the $\bal$ parameters are assigned to the columns of the lattice model, the techniques currently known (such as those in~\cite{HardtHamiltonians}) cannot be applied.

We will show that it is possible to include the column parameters $\bal$; in addition, we will incorporate a second set of parameters $\bbe$.
Obtaining this generalization of the boson-fermion correspondence requires an idea commonly used in studying factorial analogs (see, \textit{e.g.},~\cite{JingRozhkovskayaJacobiTrudi,JingRozhkovskayaShifted,MiyauraMukaihiraFactorial}), which is to replace normal powers with shifted powers in formulas.
Starting with the fermion fields, we use one type of shifted powers for $\bal$ and another for $\bbe$, following the construction in~\cite{MiyauraMukaihiraGeneralized,MiyauraMukaihiraFactorial,NaprienkoFFS}.
Under a reasonably mild analytic assumption, we show that the vacuum expectation for a pair creation and annihilation fermion fields satisfies the same evaluation as in the classical case (Proposition~\ref{prop:vacuum_deformed_fields}):
\begin{equation}
  \label{eq:intro_key} \bra{\varnothing} \psi(z|\bal; \bbe) \psi^{\ast}(w|\bal;\bbe) \ket{\varnothing}= \frac{z}{z - w},
  \quad\qquad
  \bra{\varnothing} \psi^{\ast}(w|\bal; \bbe) \psi(z|\boldsymbol{\alpha}; \bbe) \ket{\varnothing} = \frac{z}{w -  z}.
\end{equation}
It turns out that~\eqref{eq:intro_key} is the key fact since the rest of the boson-fermion correspondence then be constructed.
In particular, the deformed fermion fields therefore satisfy the same relations as the undeformed ones.
Therefore we have an automorphism (with some formality around convergence) on the Clifford algebra.
Indeed, we can construct a deformed version of the vertex operators (Theorem~\ref{thm:fermion_vertex_op}) and all of the requisite pieces uniquely and explicitly such as current operators (Proposition~\ref{prop:explicit_current}) and the shift operator (Proposition~\ref{prop:deformed_shift_identities}).

It is surprising that our deformation of the fermion fields still satisfies~\eqref{eq:intro_key} despite being a non-homogeneous deformation.
This is the first time such a deformation has been done as far as the authors are aware.

In our boson-fermion correspondence, we use our deformed shift operator and deformed half vertex operator to show that the natural basis of $\fermionfock$ goes to the two (family of) parameter deformation of the Schur functions, which we call the \emph{double factorial Schur functions}, that were constructed from a generating function in~\cite{MiyauraMukaihiraGeneralized,MiyauraMukaihiraFactorial}.
As a consequence of our constructions and with Wick's theorem, we prove numerous formulas by mimicking their classical vertex operator proofs.
Such results include Jacobi--Trudi formulas (Theorem~\ref{thm:jacobi_trudi}), Murnaghan--Nakayama rule (Theorem~\ref{thm:factorialmn}), Giambelli formula (Theorem~\ref{thm:giambelli}), skew Cauchy identity (Theorem~\ref{thm:skew_cauchy}) and different dualities (Corollary~\ref{cor:involution_skew} and Theorem~\ref{thm:dual_schur_identity}).
Furthermore, we are able to show our functions are solutions to the 2D Toda lattice (Section~\ref{sec:KP_Toda}) and its specializations (\textit{e.g.}, the Kadomtsev--Petviashvili (KP) hierarchy; Corollary~\ref{cor:KP_tau_solution}).

In Naprienko~\cite{NaprienkoFFS}, solvable lattice models were given whose partition functions include the double factorial Schur functions that we study.
In Theorem~\ref{thm:texph}, we connect our fermionic construction with the construction in~\cite{NaprienkoFFS} by showing that the deformed half vertex operators that appear in our description correspond precisely to the row transfer matrices in these lattice models.
Our proof relies on understanding how a single particle moves by our explicit formulas and then applying the so-called free fermion Lindstr{\"o}m--Gessel--Viennot (LGV) Lemma~{\cite[Proposition~A.2]{NaprienkoFFS}}.
An important consequence of this is that when we specialize $p_i$ to a powersum symmetric function, we can give a precise combinatorial description of the functions that is not apparent otherwise.
As previously mentioned, there have been several prior instances where free fermionic six vertex models have been connected to half vertex operators; \textit{e.g.},~\cite{BBBGVertex,BrubakerSchultzHamiltonians,HardtHamiltonians,KorffHecke,ZinnJustinTiling}.
All of this prior work has considered lattice models with row parameters only; ours is the first instance to describe a row transfer matrix involving column parameters using half vertex operators.
The current paper can be viewed as a culmination of these efforts since our results encompass the fully-general free fermionic six vertex model.
One can similarly study solvable lattice models which are not six vertex models but which do correspond to half vertex operators~\cite{BBBGVertex}, and it may be that our methods could help introduce column parameters in this setting.
On the other hand, we could similarly deform other applications of (half) vertex operators to other generalizations of Schur functions, such as for Hall--Littlewood polynomials~\cite{JingVertex}, symplectic/orthogonal characters~\cite{Baker96,JLW24}, and torus equivariant K-theory of the (orthogonal/symplectic) Grassmannian~\cite{GJ24II,GJ24,Iwao22,Iwao23,Iwao23II,IMS24}.

Finally, when $\bbe = 0$, we can remove the analytic assumption, which is studied in our companion paper~\cite{BHS0}.

This paper is organized as follows. We start by setting some basic information in Section~\ref{sec:shifted_powers} with shifted powers.
In Section~\ref{sec:boson_fermion}, we recall the classical boson-fermion correspondence.
In Section~\ref{sec:deformed}, we define our deformed fermion fields and describe them using vertex operators.
In Section~\ref{sec:deformed_correspondence}, we give our deformed version of the boson-fermion correspondence and describe a number of consequences.
In Section~\ref{sec:lattice_models}, we show that our deformed half vertex operator corresponds to the row transfer matrix of the solvable lattice model from~\cite{NaprienkoFFS}.
In Section~\ref{sec:deformed_powersums}, we give another version of the boson-fermion correspondence that deforms the powersum symmetric functions.

\subsection*{Acknowledgements}

The authors are very grateful to Slava Naprienko, who was invaluable in developing the results of this paper.
Additionally, the authors thank Ben Brubaker, Christian Korff, David Ridout, and Natasha Rozhkovskaya for useful conversations.
This work benefited from computations using \textsc{SageMath}~\cite{sage}.


T.S.~was partially supported by Grant-in-Aid for JSPS Fellows 21F51028 and for Scientific Research for Early-Career Scientists 23K12983.
A.H.~was partially supported by NSF RTG grant DMS-1937241.

\section{Shifted powers}
\label{sec:shifted_powers}

We begin by setting some general notation for the paper.
We use $\ii = \sqrt{-1}$, which will only be used to normalize the measure for contour integral formulas, as we frequently use $i \in \ZZ$.
Consider two sequences of parameters $\bal = (\cdots, \alpha_{-1}, \alpha_0, \alpha_1, \cdots)$ and $\bbe = (\cdots, \beta_{-1}, \beta_0, \beta_1, \cdots)$ in $\CC$.
Let $\sigma_{\bal}$ be the operator that acts on $\bal$ by $\alpha_i \mapsto \alpha_{i+1}$.
We define $\sigma_{\bbe}$ similarly acting on $\bbe$.
When there is no danger of confusion we sometimes simply write $\sigma$ for either $\sigma_{\bal}$ or $\sigma_{\bbe}$.
Similarly, we define an operator $\iota_{\bal}$ by $\alpha_i \mapsto \alpha_{1-i}$ and likewise for $\iota_{\bbe}$.
Let $z, w$ be indeterminates, which we will often consider to be algebraically independent transcendental elements in $\CC$.

We define two related, but slightly different, \defn{shifted powers} by
\begin{align*}
(z; \bal)^k & := \begin{cases} (1 - z\alpha_1) (1 - z\alpha_2) \cdots (1- z\alpha_k)  & \text{if } k > 0, \\
1 & \text{if } k = 0, \\
	(1 - z\alpha_0)^{-1} (1 - z\alpha_{-1})^{-1} \cdots (1 - z\alpha_{k+1})^{-1} & \text{if } k < 0.
\end{cases}
\allowdisplaybreaks \\
(z|\bbe)^k & := \begin{cases}
(z - \beta_1) (z - \beta_2) \cdots (z - \beta_k) & \text{if } k > 0, \\
1 & \text{if } k = 0, \\
(z - \beta_0)^{-1} (z - \beta_{-1})^{-1} \cdots (z - \beta_{k+1})^{-1} & \text{if } k < 0,
\end{cases}
\end{align*}
These were extensively used in~\cite{MiyauraMukaihiraGeneralized}, and we note some simple, yet useful, relations from the definition~\cite[Eq.~(2.4), (2.5), (2.6)]{MiyauraMukaihiraGeneralized}
\begin{subequations}
\label{eq:basic_spowers_rels}
\begin{align}
\label{eq:shifting_powers}
\sigma^m (z;\bal)^k & = (z;\sigma^m\bal)^k = \frac{(z;\bal)^{k+m}}{(z;\bal)^m},
&
\sigma^m (z|\bbe)^k & = (z|\sigma^m\bbe)^k = \frac{(z|\bbe)^{k+m}}{(z|\bbe)^m},
\allowdisplaybreaks\\
\label{eq:shifted_inversion}
\frac{1}{(z;\bal)^k} & = (z;\iota \bal)^{-k} = (z;\sigma^k\bal)^{-k},
&
\frac{1}{(z|\bbe)^k} & = (z|\iota \bbe)^{-k} = (z|\sigma^k\bbe)^{-k},
\allowdisplaybreaks\\
(z;\bal)^k & = z^k (z^{-1}|\bal)^k,
&
(z|\bbe)^k & = z^k (z^{-1}; \bbe)^k.
\end{align}
\end{subequations}
In particular, note that
\[
(z;\bal)^k (z;\sigma^k\bal)^{-k} = \frac{(z;\bal)^k}{(z;\bal)^k} = (z|\bbe)^k (z|\sigma^k\bbe)^{-k} = \frac{(z|\bbe)^k}{(z|\bbe)^k} = 1.
\]
For brevity, we write $(z;\bal) = (z;\bal)^1$ and $(z|\bbe) = (z|\bbe)^1$.
In this paper, $(z;\bal)^k$ and $(z|\bbe)^k$ will \emph{never} mean the $k$-fold product of $(z;\bal)$ and $(z|\bbe)$, respectively, so there will be no danger of confusion.

Let us show that the shifted powers satisfy a version of the finite geometric progression formula. 

\begin{lemma}\label{lem:finitegeometric}
	We have
	\[
		(1-wz)\sum_{k=0}^{n}(1-\alpha_{k+1}\beta_{k+1})\frac{(z|\bal)^k}{(z;\bbe)^{k+1}}\frac{(w|\bbe)^k}{(w;\bal)^{k+1}} = 1 - \frac{(z|\bal)^{n+1}}{(z;\bbe)^{n+1}}\frac{(w|\bbe)^{n+1}}{(w;\bal)^{n+1}}.
	\]
\end{lemma}
\begin{proof}
	Observe that
	\[
		(1-wz)(1-\alpha_{k+1}\beta_{k+1}) = (1-\beta_{k+1}z)(1-\alpha_{k+1}w) - (z-\alpha_{k+1})(w-\beta_{k+1}).
	\]
	Therefore, we have
	\[
		(1-wz)(1-\alpha_{k+1}\beta_{k+1})\frac{(z|\bal)^k}{(z;\bbe)^{k+1}}\frac{(w|\bbe)^k}{(w;\bal)^{k+1}} = \frac{(z|\bal)^k}{(z;\bbe)^{k}}\frac{(w|\bbe)^k}{(w;\bal)^{k}} - \frac{(z|\bal)^{k+1}}{(z;\bbe)^{k+1}}\frac{(w|\bbe)^{k+1}}{(w;\bal)^{k+1}}.
	\]
	Consequently, the finite sum becomes a telescoping sum and only the first and the last terms survive: 
	\[
		(1-zw)\sum_{k=0}^{n}(1-\alpha_{k+1}\beta_{k+1})\frac{(z|\bal)^k}{(z;\bbe)^{k+1}}\frac{(w|\bbe)^k}{(w;\bal)^{k+1}} = 1 - \frac{(z|\bal)^{n+1}}{(z;\bbe)^{n+1}}\frac{(w|\bbe)^{n+1}}{(w;\bal)^{n+1}}.\qedhere 
	\]
\end{proof}

We will also use the variation of the result above.
\begin{corollary}\label{cor:finitegeometric}
	We have
	\[
		(1-wz^{-1})\sum_{k=0}^{n}(1-\alpha_{k+1}\beta_{k+1})\frac{z(z;\bal)^k}{(z|\bbe)^{k+1}}\frac{(w|\bbe)^k}{(w;\bal)^{k+1}} = 1 - \frac{(z;\bal)^{n+1}}{(z | \bbe)^{n+1}}\frac{(w|\bbe)^{n+1}}{(w;\bal)^{n+1}}.
	\]
\end{corollary}

\begin{proof}
	Make a change of variables $z \mapsto z^{-1}$ and simplify the expression afterwards.
\end{proof}

We will often want to equate coefficients defined in terms of the shifted powers $z\frac{(z|\bbe)^{i-1}}{(z;\bal)^i}$ or $\frac{(w;\bal)^{i-1}}{(w|\bbe)^i}$.
To do so, we need the following contour integral formula, which was given essentially as a formal definition in~\cite[Eq.~(2.12)]{MiyauraMukaihiraGeneralized} due to the fact that no explanation was given on how to rewrite the integrand as a formal distribution (\textit{cf}.~Remark~\ref{rem:formal_contour_integral} below).

\begin{proposition}
\label{prop:orthonormality}
Let $\eta$ be a counterclockwise circle centered at $0$ of radius $\abs{\beta_i} < r < \abs{\alpha_j^{-1}}$ for all $i$ and $j$.
Then
\[
\oint_{\eta} \frac{(z|\bbe)^{k-1}}{(z;\bal)^k} \frac{(z; \bal)^{n-1}}{(z|\bbe)^n} \frac{dz}{2\pi\ii}
= \oint_{\eta} \frac{(z; \sigma_{\bal}^k \bal)^{n-k-1}}{(z|\sigma_{\bbe}^{k-1}\bbe)^{n-k+1}} \frac{dz}{2\pi\ii}
= \frac{\delta_{nk}}{1 - \alpha_n \beta_n}.
\]
\end{proposition}

\begin{proof}
If $n = k$, then by the residue theorem, we have
\[
\oint_{\eta} \frac{1}{(z - \beta_n)(1 - \alpha_n z)} \frac{dz}{2\pi\ii} 
= \lim_{z\to\beta_n} \frac{z - \beta_n}{(z - \beta_n)(1 - \alpha_n z)} = \frac{1}{1 - \alpha_n \beta_n}
\]
since we have a simple pole at $z = \beta_n$.
If $n < k$, then we have no poles inside $\eta$, and so the result is $0$ by the residue theorem.
If $n > k$, then we have no zeros outside $\eta$, and so the result is $0$.
Indeed, this is applying the residue theorem with the substitution $z \mapsto z^{-1}$ and noting that $\lim_{\abs{z}\to\infty} \abs{f(z)} = \infty$, where $f(z) = \frac{(z; \sigma_{\bbe}^k \bbe)^{n-k-1}}{(z|\sigma_{\bal}^{k-1}\bal)^{n-k+1}}$.
\end{proof}

Using Proposition~\ref{prop:orthonormality}, we can define an inner product with the desired shifted powers are orthogonal, and hence linearly independent.

\begin{remark}
\label{rem:formal_contour_integral}
To justify~\cite[Eq.~(2.12)]{MiyauraMukaihiraGeneralized}, we will rewrite factors of the form
\[
\frac{1}{z - \beta_n} = \frac{z^{-1}}{1 - \beta_n z^{-1}} = \sum_{m=0}^{\infty} \beta_n^m z^{-m-1} \in \ZZ[\beta_n]\FPS{z^{-1}}.
\]
Hence, we can rewrite the integrand as a formal distribution in $\ZZ\FPS{\bal,\bbe}\FPS{z^{\pm1}}$.
Next, we consider the contour integral formally as $\oint f(z) \frac{dz}{2\pi\ii} := f_{-1}$ for any formal distribution $f(z) = \sum_{i\in\ZZ} f_i z^i$.
Now, if we take the $n = k$ case, then we have
\[
\oint \frac{z^{-1}}{(1 - \beta_n z^{-1})(1 - \alpha_n z)} \frac{dz}{2\pi\ii} = \sum_{m=0}^{\infty} (\alpha_n \beta_n)^m = \frac{1}{1 - \alpha_n\beta_n},
\]
where the first equality is easy to see from the Cauchy product formula and the last equality is standard in $\ZZ\FPS{\alpha_n,\beta_n}$.
For $n < k$, we have a product of factors $(z - \beta_i), (1 - \alpha_j z)^{-1} \in \ZZ[\bal,\bbe]\FPS{z}$, and hence, the formal contour integral is $0$.
When $n > k$, it is straightforward to see the degree of the expansion is strictly less than $-1$, yielding the formal contour integral being $0$.
(Alternatively, in this case we can take the integrand as a formal Laurent series in $\ZZ[\bal,\bbe]\FLS{z^{-1}}$ and the valuation (with respect to $z^{-1}$) is strictly greater than $1$.)
\end{remark}

\section{The classical boson-fermion correspondence}
\label{sec:boson_fermion}

For this section, we largely follow~\cite{AlexandrovZabrodin,MJD00} except we have interchanged the roles of holes and particles and some of our terminology.
We refer the reader to these sources for more details, and we note two other references are~\cite{KacInfinite,KacRaina}.

Consider the vector space $V = \bigoplus_{i \in \ZZ} \CC v_i$.
We will use the Clifford algebra $\mcC := \Cl(V \oplus V^*)$ using the canonical pairing between $V$ and $V^*$, which can be described explicitly as the $\CC$-algebra generated by $\{\psi_i, \psi_i^* \mid i \in \ZZ\}$ (with finite sums) subject to the relations
\[
[\psi_i, \psi_j]_+ = [\psi_i^*, \psi_j^*]_+ = 0,
\qquad\qquad
[\psi_i, \psi_j^*]_+ = \delta_{ij},
\]
where $[x,y]_+ = xy + yx$.
Subsequently, there is an anti-involution that interchanges $\psi_i \leftrightarrow \psi_i^*$, which by slight abuse of notation is denoted $\ast$.
Taking an appropriate limit $\CC^n \xrightarrow{n\to\infty} V$, we can construct a $\mcC$-representation $\fermionfock$ generated by a single vector $\ket{\varnothing}$ subject to the relations
\[
\psi_i \ket{\varnothing} = 0 \quad \text{if } i \leqslant 0,
\qquad\qquad
\psi_j^* \ket{\varnothing} = 0 \quad \text{if } j > 0.
\]
(This agrees with~\cite{KacInfinite} but not~\cite{AlexandrovZabrodin}.)
We call this representation $\fermionfock$ the \defn{fermionic Fock space}, and it is irreducible and faithful. It is essentially
the only irreducible representation of~$\mcC$.

There is a natural basis for $\fermionfock$ from the defining relations.
However, we want to consider another basis by realizing $\fermionfock$ as the analogous limit of spinor representations.
We will use language motivated by Dirac's theory of the electron (see~{\cite[Sec.~4.2]{KacRaina}}) in connection with the Fock space.
We define the \defn{charge} (or level) $m$ component $\fermionfock_m$ consisting of \defn{semi-infinite monomials} of the form
\[
\ket{\eta}_m = v_{i_m} \wedge v_{i_{m - 1}} \wedge \cdots,
\]
where $\eta = i_m > i_{m - 1} > \cdots$ and $i_k = k$ for $k \ll 0$.
It is understood that the $\wedge$ operation is anticommutative, so $v_{i_m} \wedge v_{i_{m - 1}} \wedge \cdots$ is defined if a finite number of the terms in the sequence $i_m, i_{m - 1}, \cdots$ are out of order (which is a requirement of our definition).
We will call $\sum_{k\leqslant m}(i_k-k)$ the \defn{energy} of the state $\eta$.
It is easy to see that the sum has only finitely many nonzero terms, and that the energy is a nonnegative integer.
Thus given such a monomial, we will say that the site $j \in \ZZ$ is \defn{occupied} if $j = i_r$ for some $r$, otherwise we will say it is \defn{unoccupied}.
In this analogy, the occupied site $j$ represents a particle that has energy $j - r$.
If $\eta$ has energy $E$, then there is a partition $\lambda = (\lambda_1 \geqslant \lambda_2 \geqslant \cdots \geqslant 0)$ of $E$ such that $i_k = k + \lambda_{m-k+1}$ for $k\leqslant m$, and we will generally use the notation $\ket{\lambda}_m$ for $\ket{\eta}_m$.
In particular, $\ket{\varnothing}_m$ the vector $v_m \wedge v_{m - 1} \wedge \cdots$, which is the \defn{vacuum} of charge $m$.
We remark that $\{ \ket{\lambda}_m \}_{\lambda \in \mcP}$, where $\mcP$ is the set of all partitions, is a basis for $\fermionfock_m$.
Additionally, an arbitrary basis vector $\ket{\lambda}_m$ can be expressed as
\begin{equation}
\label{eq:ket-from-vacuum}
\ket{\lambda}_m = 
\begin{cases}
\psi_{m+\lambda_1}\psi_{m+\lambda_2-1}\cdots \psi_{m+\lambda_\ell-\ell+1} \psi_{m-\ell}\psi_{m-\ell-1}\cdots \psi_2\psi_1 \ket{\varnothing}, & \text{if } m-\ell\ge 0, \\ 
\psi_{m+\lambda_1}\psi_{m+\lambda_2-1}\cdots \psi_{m+\lambda_\ell-\ell+1} \psi_{m-\ell+1}^*\psi_{m-\ell+2}^*\cdots \psi_{-1}^*\psi_0^* \ket{\varnothing}, & \text{if } m-\ell<0,
\end{cases}
\end{equation}
where $\ell=\ell(\lambda)$.

We have $\fermionfock \iso \bigoplus_{m\in\ZZ} \fermionfock_m$ with $\ket{\varnothing} \leftrightarrow \ket{\varnothing}_0$, and we will henceforth identify the two representations.
Next we describe the $\mcC$-action explicitly on the semi-infinite monomials.
The generators act on $\ket{\eta}_m = v_{i_m} \wedge v_{i_{m-1}} \wedge \cdots$ by
\begin{align*}
\psi_j \ket{\eta}_m & = v_j \wedge v_{i_m} \wedge v_{i_{m-1}} \wedge \cdots,
\\
\psi_j^* \ket{\eta}_m & = \begin{cases} (-1)^{k-1} v_{i_m} \wedge \cdots \wedge \widehat{v}_{i_{m-k}} \wedge \cdots & \text{if $i_{m-k} = j$ for some $k$}, \\ 0 & \text{otherwise}, \end{cases}
\end{align*}
where $\widehat{v}_{i_{m-k}}$ denotes the vector does not appear in the wedge product.
As such, we refer to the generators $\psi_i$ (resp.\ $\psi_i^*$) as \defn{creation operators} (resp.\ \defn{annihilation operators}) as they try to create (resp.\ annihilate) a particle at position $i$ (this differs from~\cite{AlexandrovZabrodin,MJD00}).
This gives a filtered action of $\mcC$ on $\fermionfock$ with $\deg \psi_i = -\deg \psi_i^* = 1$.
For brevity, when working with charge $m = 0$, we will often simply write $\ket{\lambda} = \ket{\lambda}_0$.

Let ${}_m \bra{\lambda}$ denote the dual vector to $\ket{\lambda}_m$ under the canonical pairing, which satisfies
\[
{}_m \braket{\mu}{\lambda}_n = \delta_{mn} \delta_{\mu\lambda},
\qquad\qquad
({}_m \bra{\mu} X) \ket{\lambda}_n = {}_m \bra{\mu} (X \ket{\lambda}_n)
\]
for any element $X \in \mcC$.
Hence, we can write ${}_m \bra{\mu} X \ket{\lambda}_n$ unambiguously.
Under this pairing, this dual representation is defined by the relations
\[
\bra{\varnothing} \psi_i = 0 \quad \text{if } i > 0,
\qquad\qquad
\bra{\varnothing} \psi_j^* = 0 \quad \text{if } j \leqslant 0.
\]
We call ${}_m \bra{\varnothing}$ the \defn{dual vacuum} of charge $m$.
By slight abuse, we will consider the (restricted) dual space $\fermionfock^* := \bigoplus_{m \in \ZZ} \fermionfock^*_m$, where  $\fermionfock^*_m := \bigoplus_{\lambda\in\mcP} \CC \cdot {}_m \bra{\lambda}$.
Note this is smaller than the typical dual space, but it is useful and sufficient for our purposes.

An \defn{operator} is a map $\Upsilon \colon \fermionfock \to \widehat{\fermionfock}$, where $\widehat{\fermionfock}$ is the formal completion of $\fermionfock$, that intertwines with the $\mcC$-action, that is $\Upsilon(\eta \ket{\xi}) = \widetilde{\Upsilon}(\eta) \Upsilon(\ket{\xi})$ for all $\eta \in \mcC$ and $\ket{\xi} \in \fermionfock$ with $\widetilde{\Upsilon}$ being the corresponding map on the representation.
For simplicity, as a matter of notation we will consider $\Upsilon$ like an element of $\mcC$; that is we write $\Upsilon\ket{\xi} = \Upsilon(\ket{\xi})$ and $\Upsilon \eta = \widetilde{\Upsilon}(\eta) \Upsilon$.

\begin{remark}
\label{rem:operator_well_defined}
We can write an operator $\Upsilon$ in terms of matrix coefficients ${}_{\ell} \bra{\mu} \Upsilon \ket{\lambda}_m$, which is well-defined.
Moreover, all matrix coefficients uniquely determines $\widetilde{\Upsilon}$ since $\fermionfock$ is faithful (see also~\cite[Prop.~2.2]{AlexandrovZabrodin}).
In particular, $\Upsilon$ can be defined as an infinite linear combination of basis elements of $\mcC$, but it is sufficient to verify that every matrix coefficient is well-defined.
\end{remark}

We have an operator $\Sigma$ (which is unique) that satisfies $\Sigma \psi_i \Sigma^{-1} = \psi_{i+1}$, $\Sigma \psi^*_i \Sigma^{-1} = \psi^*_{i+1}$ and $\Sigma \ket{\varnothing}_m = \ket{\varnothing}_{m+1}$. This is called the \defn{shift operator}.
We can explicitly describe the action of $\Sigma$ on $\fermionfock$ using the spinor module presentation by $\Sigma \ket{\varnothing} = \psi_1 \ket{\varnothing}$ and $\Sigma^{-1} \ket{\varnothing} = \psi_0^* \ket{\varnothing}$, which is consistent since
\[
\Sigma^{-1} \Sigma \ket{\varnothing} = \Sigma^{-1} \psi_1 \ket{\varnothing} = \psi_0 \Sigma^{-1} \ket{\varnothing} = \psi_0 \psi_0^* \ket{\varnothing} = (1 - \psi_0^* \psi_0) \ket{\varnothing} = \ket{\varnothing}.
\]
Note that $\Sigma^* = \Sigma^{-1}$. 
We note that $\Sigma\ket{\lambda}_m = \ket{\lambda}_{m+1}$, and by the canonical pairing, ${}_m \bra{\lambda} \Sigma = {}_{m-1} \bra{\lambda}$ .

Let us set some additional notation.
If $n \geqslant 0$ we will denote by
\[
	\ket{n}_m = \ket{(n)}_m = v _{n+m} \wedge v_{m - 1} \wedge v_{m - 2} \wedge \cdots = v_{n+m} \wedge \ket{\varnothing}_{m - 1},
\]
Note that $\ket{0}_m = \ket{\varnothing}_m$.
We define the \defn{fermion fields} as the formal sums
\[
\psi(z) = \sum_{i \in \ZZ} \psi_i z^i,
\qquad\qquad
\psi^*(w) = \sum_{j \in \ZZ} \psi_j^* w^{-j},
\]
recalling that $z$ is a formal indeterminate.
Note
\begin{equation}
\label{eq:duality_fermion_fields}
\bigl(\psi(z) \bigr)^* = \psi^*(z^{-1})
\qquad \text{ and } \qquad
\bigl(\psi^*(w) \bigr)^* = \psi(w^{-1}).
\end{equation}

For an operator $X$, we define a \defn{normal ordering} $\normord{X}$ by all $\psi_i$ and $\psi_j^*$ are considered to be anticommuting inside and moved to the right in $X$ whenever $i \leqslant 0$ and $j > 0$ and extended by linearity (see, \textit{e.g.},~\cite{MJD00} for an axiomatic description).
Using this, we define the \defn{current operators} as $J_k := \sum_{i\in\ZZ} \normord{\psi_i \psi^*_{i+k}} = \sum_{i\in\ZZ} \normord{\psi_{i-k} \psi^*_{i}}$ and also can be constructed by
\begin{equation}
\label{eq:current_residue}
J_k = \res_{z=0} \bigl( z^{k-1} \normord{\psi(z) \psi^*(z)} \bigr) = \oint_{\gamma} z^k \normord{\psi(z) \psi^*(z)} \frac{dz}{2\pi\ii z} = \sum_{i,j} \oint_{\gamma} z^{i-j+k} \normord{\psi_i \psi_j^*} \frac{dz}{2\pi\ii z},
\end{equation}
where $\gamma$ is considered as a counterclockwise oriented circle around $\infty$.
Here, we treat the contour integral as a formal operator, but it will match any expectation values that we will want to compute.
We note that $J_k \ket{\varnothing}_m = 0$ and ${}_m \bra{\varnothing} J_{-k}$ for all $k > 0$ and $m \in \ZZ$, and for $k \neq 0$, we can consider $J_k$ as changing the position of a particle by $-k$ in all possible ways.
Furthermore, $J_k^* = J_{-k}$.

The current operator $J_0$ is special, and this is the only current operator where we need to use the normal ordering.
Indeed, from the normal ordering, the operator $J_0$ operator is well-defined acting on $\fermionfock$ as for all $j$ we have
\[
\normord{\psi_i \psi_j^*} = -\normord{\psi_j^* \psi_i} = \begin{cases}
-\psi_j^* \psi_i & \text{if } i \leqslant 0, \\
\psi_i \psi_j^* & \text{if } i > 0,
\end{cases}
\]
and $J_0 \ket{\eta} = m \ket{\eta}$ for any $\ket{\eta} \in \fermionfock_m$; in particular $J_0 \ket{\lambda}_m = m \ket{\lambda}_m$ for all partitions $\lambda$.
We also remark that ${}_{m} \bra{\lambda} J_0 = m \cdot {}_{m} \bra{\lambda}$ for all partitions $\lambda$.
Thus Equation~\eqref{eq:current_residue} allows us to also define the \defn{boson field} as
\[
J(z) := \sum_{k\in\ZZ} J_k z^{-k} = \normord{\psi(z) \psi^*(z)}.
\]
To further emphasize the link between our normal ordering and the vacuum $\ket{\varnothing}$, we note
\begin{subequations}
\label{eq:fermion_normal_ordering}
\begin{align}
\label{eq:fermion_field_normord}
\normord{\psi(z) \psi^*(w)} & = \psi(z) \psi^*(w) - \bra{\varnothing} \psi(z) \psi^*(w) \ket{\varnothing},
\\
\normord{\psi^*(w) \psi(z)} & = \psi^*(w) \psi(z) - \bra{\varnothing} \psi^*(w) \psi(z) \ket{\varnothing},
\end{align}
\end{subequations}
as well as the identities
\begin{align*}
\bra{\varnothing} \psi(z) \psi^*(w) \ket{\varnothing} 
& = \frac{z}{z - w},
&
\normord{\psi(z) \psi^*(w)} & = \psi(z) \psi^*(w) + \frac{z}{z-w},
\\
\bra{\varnothing} \psi^*(w) \psi(z) \ket{\varnothing} 
& = \frac{z}{w - z},
&
\normord{\psi^*(w) \psi(z)} & = -\normord{\psi(z) \psi^*(w)},
\\
[J_k, \psi(z)] & = z^k \psi(z),
&
[J_k, \psi^*(z)] & = -z^k \psi^*(z).
\end{align*}
We require $\abs{w} < \abs{z}$ (resp.\ $\abs{z} < \abs{w}$) for the top (resp.\ middle) left equalities to hold as functions as otherwise the left hand side would not converge.
Consequently, we have
\begin{equation}
\label{eq:classical_delta_comm}
\begin{aligned}
[\psi(z), \psi^*(w)]_+ & = \normord{\psi(z) \psi^*(w)} + \frac{z}{z-w} + \normord{\psi^*(w) \psi(z)} + \frac{z}{w-z}
\\ & = \frac{z}{z-w} + \frac{z}{w-z} = \delta(w/z),
\end{aligned}
\end{equation}
where $\delta(\zeta) = \sum_{k \in \ZZ} (\zeta)^k$ is the \defn{formal $\delta$ distribution} (or the delta function).
We remark that the formal $\delta$ distribution satisfies $f(z)\delta(w/z) = f(w)\delta(w/z)$ for any function $f$.
Note that if we na\"ively simplified the last step, we would obtain $0$, but we have mutually exclusive conditions on the absolute values of $z$ and $w$.

Next, the current operators generate the (infinite dimensional) \defn{Heisenberg Lie algebra} $\mcH$ as they satisfy the relations
\[
[J_k, J_{\ell}] = k \delta_{k,-\ell} \cdot \one,
\]
where $\one$ is the identity element in the Clifford algebra.
The universal enveloping algebra of the Heisenberg Lie algebra is the (differential) Weyl algebra $\mcW$ for the polynomial ring $\bosonfock = \CC[\svar^{\pm}; \pp]$, where $\pp = (p_1, p_2, \cdots)$.
We consider the polynomial ring as $\ZZ$-graded $\bosonfock = \bigoplus_{m\in\ZZ} \bosonfock_m$ defined by $\bosonfock_m := q^m \CC[\pp]$.
This isomorphism $\Phi$ is realized by, for $k > 0$, mapping
\[
J_{-k} \mapsto p_k,
\qquad\qquad
J_k \mapsto k \frac{\partial}{\partial p_k},
\qquad\qquad
J_0 \mapsto \svar \frac{\partial}{\partial \svar},
\qquad\qquad
\Sigma \mapsto \svar.
\]
This is also another minor difference with~\cite{AlexandrovZabrodin,MJD00} as they map $J_{-k} \mapsto \frac{1}{k} p_k$ and $J_k \mapsto \frac{\partial}{\partial p_k}$, which will also impact some of our formulas below.

Thus, we can now state the \defn{boson-fermion correspondence}, which says that $\fermionfock$ is isomorphic to the natural $\mcW$-representation on the polynomial ring $\bosonfock$, which is known as the \defn{bosonic Fock space}, under the map
\[
\ket{U} \mapsto \sum_{k \in \ZZ} {}_k \bra{\varnothing} e^{H_+(\pp)} \ket{U} \cdot \svar^k,
\qquad\qquad
\text{ where }
H_{\pm}(\pp) := \sum_{k=1}^{\infty} \frac{p_k}{k} J_{\pm k}
\]
are referred to as the \defn{Hamiltonians}.
To see the inverse map and the $\mcC$-action on $\bosonfock$, we first note that we can write
\begin{equation}
\label{eq:field_vertex_op}
\psi(z) = e^{H_-(z)} z^{J_0} \Sigma e^{-H_+(z^{-1})},
\qquad\qquad
\psi^*(z) = e^{-H_-(z)} \Sigma^{-1} z^{-J_0} e^{H_+(z^{-1})},
\end{equation}
where we use the shorthand $H_{\pm}(z) := H_{\pm}(z, z^2, z^3, \cdots)$.
The fermion fields transform under the boson-fermion correspondence to an (operator) action on $\bosonfock$ given by the \defn{vertex operators}
\begin{align*}
X(z) & = e^{\overline{H}_-(z)} z^Q \svar e^{-\overline{H}_+(z^{-1})} = \exp\left( \sum_{k=1}^{\infty} \frac{p_k}{k} z^k \right) z^Q \svar \exp\left( \sum_{k=1}^{\infty} \frac{\partial}{\partial p_k} z^{-k} \right),
\\
X^*(w) & = e^{-\overline{H}_-(w)} \svar^{-1} w^{-Q} e^{\overline{H}_+(w^{-1})} = \exp\left( \sum_{k=1}^{\infty} -\frac{p_k}{k} w^k \right) \svar^{-1} w^{-Q} \exp\left( \sum_{k=1}^{\infty} \frac{\partial}{\partial p_k} w^{-k} \right),
\end{align*}
where $Q = \svar \frac{\partial}{\partial \svar}$ and $\overline{H}_{\pm}$ is the image under $\Phi$ of $H_{\pm}$.

If we expand the vertex operators into Fourier modes $X(z) = \sum_{i \in \ZZ} X_i z^i$ and $X^*(w) = \sum_{i \in \ZZ} X_i^* w^i$,
then the action of $\psi_i$ and $\psi_i^*$ corresponds to $X_i$ and $X_i^*$, respectively.
We can see this correspondence as
\begin{equation}
\label{eq:xi_function_def}
\xi(\pp; \pp') := [H_+(\pp), H_-(\pp')] = \sum_{k=1}^{\infty} \frac{1}{k} p_k p'_k,
\end{equation}
where $\pp' = (p_1', p_2', \cdots)$, and the Baker--Campbell--Hausdorff (BCH) formula implies
\begin{equation}
\label{eq:half_vertex_commutator}
e^{H_+(\pp)} e^{H_-(\pp')} = e^{\xi(\pp; \pp')} e^{H_-(\pp')} e^{H_+(\pp)},
\end{equation}
and similarly for the images under $\Phi$.
In particular, if we take
\[
p_k = p_k(\xx / \yy) = \sum_{i=1}^{\infty} x_i^k - (-y_i)^k
\]
to be the \defn{powersum supersymmetric function}\footnote{If we use standard plethystic notation then $p_k(\xx / \yy) = p_k[\xx - (-1) \yy]$; see, \textit{e.g.},~\cite{LR11}.} and $p'_k(z) = z^k$, then we have
\begin{equation}
\label{eq:exp_xi_specialization}
e^{\xi(\pp; z^1, z^2, \cdots)} = \prod_{i=1}^{\infty}  \frac{1 + y_i z}{1 - x_i z}.
\end{equation}

There is another way of interpreting~\eqref{eq:field_vertex_op} coming from the Weyl algebra $\mcW$.
To do so, we need to describe the boson normal ordering $\bnormord{X}$, which is given by moving all of the creation boson operators $J_k$ in $X$ to the left and annihilation boson operators $J_{-k}$ to the right.
Note that since the boson operators (nearly) commute with each other (as opposed to nearly skew commute like the fermions), we simply treat boson elements as commutative within the boson normal ordering sign.
Furthermore, it is different than the normal ordering $\normord{F}$ that we used for the fermions (\textit{i.e.}, the Clifford algebra elements).
Let $K := \ln \Sigma$, and this is canonically conjugate to $J_0$, which means it satisfies $[J_i, K] = \delta_{i0}$.
Hence, we can define the \defn{chiral boson field}
\[
\phi(z) := H_-(z) + K + J_0 \ln z - H_+(z),
\]
and therefore pulling back the vertex operators yield
\begin{equation}
\label{eq:chiral_boson_identities}
\psi(z) = \bnormord{e^{\phi(z)}},
\qquad\qquad
\psi^*(z) = \bnormord{e^{-\phi(z)}},
\qquad\qquad
z \frac{\partial}{\partial z} \phi(z) = J(z).
\end{equation}

\section{The deformed fermions and operators}
\label{sec:deformed}

In this section, we define $(\bal,\bbe)$-deformed versions of the fermion fields, current operators, and shift operator.
Our proofs will largely follow the computations in~\cite{AlexandrovZabrodin,MJD00}, where our main goal is to construct vertex operator descriptions of our deformed fermion fields.
We will refer to the $\bal = \bbe = 0$ specialization (that is, we set $\alpha_j = \beta_j = 0$ for all $j \in \ZZ$) as the classical case as all our formulas will reduce to those in Section~\ref{sec:boson_fermion}.

\subsection{Deformed fermion fields}

We begin by defining the \defn{deformed fermion fields} simply by using the shifted powers in place of the usual powers in the fermion fields:
\[
\psi(z|\bal; \bbe) = \sum_{i \in \ZZ} z \frac{(z|\bbe)^{i-1}}{(z;\bal)^i} \psi_i
\qquad\quad
\psi^*(w|\bal; \bbe) = \sum_{j \in \ZZ} (1-\alpha_j\beta_j) \frac{(w;\bal)^{j-1}}{(w|\bbe)^j} \psi_j^*.
\]
One might be curious about the extra factor of $z$ in $\psi(z|\bal)$; this is just to make our formulas match the classical formulas in Section~\ref{sec:boson_fermion}.
This could also be considered as coming from our choice of the origin for $\ket{\varnothing}_0$.

We first want to show that the deformed fermion fields satisfy the \emph{same} vacuum expectation that the usual fermion fields.

\begin{proposition}
\label{prop:vacuum_deformed_fields}
Assume that
\[
\frac{(z;\bal)^{n+1} (w|\bbe)^{n+1}}{(z|\bbe)^{n+1} (w;\bal)^{n+1}} \xrightarrow{n\to\infty} 0,
\qquad\qquad
\text{resp.}\ \ \frac{(z|\bbe)^{n+1} (w;\bal)^{n+1}}{(z;\bal)^{n+1} (w|\bbe)^{n+1}} \xrightarrow{n\to\infty} 0,
\]
then we have
\begin{align*}
\bra{\varnothing} \psi(z|\bal; \bbe) \psi^*(w|\bal; \bbe) \ket{\varnothing} & = \frac{z}{z - w},
\\
\bra{\varnothing} \psi^*(w|\bal; \bbe) \psi(z|\bal; \bbe) \ket{\varnothing} & = \frac{z}{w - z},
\end{align*}
respectively, as functions $\CC^2 \to \CC$.
\end{proposition}

\begin{proof}
We prove the first formula.
Using the fact that $\bra{\varnothing} \psi_i \psi_j^* \ket{\varnothing} = \delta_{ij}$ for $i,j \leqslant 0$ and $\psi_j^* \ket{\varnothing} = 0$ for $j > 0$, we have
\begin{align*}
	\bra{\varnothing} \psi(z|\bal; \bbe) \psi^*(w|\bal; \bbe) \ket{\varnothing} &= \sum_{k=0}^{\infty}z \frac{(z|\bbe)^{-k-1}}{(z;\bal)^{-k}}(1-\alpha_{-k}\beta_{-k}) \frac{(w;\bal)^{{-k}-1}}{(w|\bbe)^{-k}}\\
	&= \sum_{k=0}^{\infty}\frac{z(z; \tau^{-k} \bal)^k}{(z | \tau^{-k-1} \bbe)^{k+1}}(1-\alpha_{-k}\beta_{-k})\frac{(w | \tau^{-k}\bbe)^k}{(w; \tau^{-k-1}\bal)^{k+1}}\\
	&= \frac{1}{1-w/z}\lim_{n \to +\infty}\left[1 - \frac{(z;\bal)^{n+1}}{(z | \bbe)^{n+1}}\frac{(w|\bbe)^{n+1}}{(w;\bal)^{n+1}}\right]\\
	&= \frac{z}{z-w}.
\end{align*}
Note that in the calculations, we used Corollary~\ref{cor:finitegeometric}.
The proof of the second formula is mutatis mutandis. 
\end{proof}

\begin{remark}
From this point forward, unless otherwise noted, we will not use our analytic assumptions in any of our proofs.
In other words, all of our results follow if we assume the validity of the vacuum expectation formulas in Proposition~\ref{prop:vacuum_deformed_fields}.
We will generally omit this assumption from our theorems.
\end{remark}

We can define deformed creation and annihilation operators by using the expansions
\[
\psi(z|\bal; \bbe) = \sum_{i \in \ZZ} \psi_i^{(\bal;\bbe)} z^i,
\qquad\qquad
\psi^*(w|\bal; \bbe) = \sum_{j \in \ZZ} \psi_j^{*(\bal;\bbe)} w^{-j}.
\]
We claim that the algebra generated by $\{\psi_i^{(\bal;\bbe)}, \psi_i^{*(\bal;\bbe)} \mid i \in \ZZ\}$ is isomorphic to $\mcC$.
To do so, we need to show the canonical anticommutation relations, which in terms of the deformed fermion fields are
\begin{subequations}
\label{eq:deformed_CAR}
\begin{gather}
[\psi(z|\bal;\bbe), \psi(w|\bal;\bbe)]_+ = [\psi^*(z|\bal;\bbe), \psi^*(w|\bal;\bbe)]_+ = 0, \label{eq:deformed_skew}
\\
[\psi(z|\bal;\bbe), \psi^*(w|\bal;\bbe)]_+ = \delta(w/z). \label{eq:deformed_delta_comm}
\end{gather}
\end{subequations}
If we rewrite~\eqref{eq:deformed_CAR} in terms of the Fourier mode expansions, then we have
\[
[\psi_i^{(\bal;\bbe)}, \psi_j^{(\bal;\bbe)}]_+= [\psi_i^{*(\bal;\bbe)}, \psi_j^{*(\bal;\bbe)}]_+ = 0,
\qquad\qquad
[\psi_i^{(\bal;\bbe)}, \psi_j^{*(\bal;\bbe)}]_+ = \delta_{ij},
\]
for all $i, j \in \ZZ$.
The proof of~\eqref{eq:deformed_skew} is immediate from the definition of $\psi(z|\bal;\bbe)$ (resp.\ $\psi^*(w|\bal;\bbe)$) involving only skew-commuting generators of $\mcC$.
To show~\eqref{eq:deformed_delta_comm}, we use~\eqref{eq:fermion_normal_ordering} and $\normord{\psi(z) \psi^*(w)} = -\normord{\psi^*(w) \psi(z)}$ to see
\begin{equation}
\label{eq:deformed_CAR_proof_step}
[\psi(z|\bal;\bbe), \psi^*(w|\bal;\bbe)]_+ = \bra{\varnothing} \psi(z) \psi^*(w) \ket{\varnothing} + \bra{\varnothing} \psi^*(w) \psi(z) \ket{\varnothing}
\end{equation}
As in the computation of~\eqref{eq:classical_delta_comm}, if we na\"ively apply Proposition~\ref{prop:vacuum_deformed_fields} in~\eqref{eq:deformed_CAR_proof_step}, then we would get $0$.
Yet we cannot do this since the analytic conditions are opposing (specifically, both cannot hold at the same time).
However, if we expand the right hand sides in $\CC\FPS{w/z}$ and $\CC\FPS{z/w}$, respectively, and consider everything as formal distributions in $\CC\FPS{w^{\pm 1}, z^{\pm 1}}$, then we clearly obtain the $\delta$ distribution, yielding~\eqref{eq:deformed_delta_comm}.

Our next goal is to show that the analogous definition of the current operators from~\eqref{eq:current_residue} also produce a representation of the Heisenberg Lie algebra $\mcH$.
Hence, we define the \defn{deformed boson field} by
\[
	J(z|\bal;\bbe) := \normord{\psi(z|\bal;\bbe)\psi^*(z|\bal;\bbe)},
\]
which we then subsequently use to define the \defn{deformed current operators} as the Fourier expansion
\begin{equation}
\label{eq:current_series}
J(z|\bal;\bbe) = \sum_{k \in \ZZ} J_k^{(\bal;\bbe)} z^{-k}.
\end{equation}

\begin{proposition}
\label{prop:current_comm}
We have
\begin{align*}
[J(w|\bal;\bbe), \psi(z|\bal;\bbe)] & = \delta(w/z) \psi(z|\bal;\bbe),
\\
[J(z|\bal;\bbe), \psi^*(w|\bal;\bbe)] & = -\delta(z/w) \psi^*(w|\bal;\bbe).
\end{align*}
Moreover, we have
\[
[J^{(\bal;\bbe)}_k, \psi(z|\bal;\bbe)] = z^k \psi(z|\bal;\bbe),
\qquad\qquad
[J^{(\bal;\bbe)}_k, \psi^*(w|\bal;\bbe)] = -w^k \psi^*(w|\bal;\bbe).
\]
\end{proposition}

\begin{proof}
We only provide a proof of the first formula as the proof for the second is similar.
To prove the first formula, we first compute
\begin{align*}
[\normord{\psi(w|\bal;\bbe) \psi^*(\zeta|\bal;\bbe)}, \psi(z|\bal;\bbe)] & = \left[ \psi(w|\bal;\bbe) \psi^*(x|\bal;\bbe) + \frac{w}{w-\zeta}, \psi(z|\bal;\bbe) \right]
\\ & = \psi(w|\bal;\bbe) \psi^*(\zeta|\bal;\bbe) \psi(z|\bal;\bbe)
\\ & \hspace{20pt} - \psi(z|\bal;\bbe) \psi(w|\bal;\bbe) \psi^*(\zeta|\bal;\bbe)
\\ & = \delta(\zeta/z) \psi(z|\bal;\bbe)
\end{align*}
by using Equation~\eqref{eq:fermion_field_normord} with Proposition~\ref{prop:vacuum_deformed_fields} and~\eqref{eq:deformed_CAR}.
The first formula follows by taking the limit $\zeta \to w$.
\end{proof}

%
%


\begin{theorem}
\label{thm:heisenberg_relations}
We have
\[
	[J(w|\bal;\bbe), J(z|\bal;\bbe)] = z \partial_z \delta(w/z).
\]
Moreover, the Lie algebra generated by $\{J_k^{(\bal;\bbe)} \mid k \in \ZZ\}$ is isomorphic to $\mcH$ with the deformed current operators satisfying the Heisenberg relations
\[
[J_k^{(\bal;\bbe)}, J_{\ell}^{(\bal;\bbe)}] = k \delta_{k,-\ell}
\]
for all $k, \ell \in \ZZ$.
\end{theorem}

\begin{proof}
We begin by computing
\begin{align*}
	[J(w|\bal;\bbe), \normord{\psi(z|\bal;\bbe)\psi^*(\zeta|\bal;\bbe)}] & = \left[ J(w|\bal;\bbe), \psi(z|\bal;\bbe)\psi^*(\zeta|\bal;\bbe) + \frac{\zeta}{z-\zeta} \right]
	\\ & = [J(w|\bal;\bbe), \psi(z|\bal;\bbe)\psi^*(\zeta|\bal;\bbe)]
	\\ & = \bigl( \delta(w/z) - \delta(w/\zeta) \bigr) \psi(z|\bal;\bbe) \psi^*(\zeta|\bal;\bbe)
	\\ & = \bigl( \delta(w/z) - \delta(w/\zeta) \bigr) \normord{\psi(z|\bal;\bbe) \psi^*(\zeta|\bal;\bbe)}
	\\ & \hspace{20pt} + \frac{z \bigl( \delta(w/z) - \delta(w/\zeta) \bigr)}{z - \zeta},
\end{align*}
where we have used the identity $[A, BC] = [A,B]C + B[A,C]$, Equation~\eqref{eq:fermion_field_normord} with Proposition~\ref{prop:vacuum_deformed_fields}, and Proposition~\ref{prop:current_comm}.
Now when we take the limit $\zeta \to z$, the first resulting term becomes
\[
\bigl( \delta(w/z) - \delta(w/\zeta) \bigr) \normord{\psi(z|\bal;\bbe) \psi^*(\zeta|\bal;\bbe)} \xrightarrow{\zeta \to z} \bigl( \delta(w/z) - \delta(w/z) \bigr) J(z|\bal;\bbe) = 0.
\]
Thus, to finish the proof we claim that
\[
\lim_{\zeta \to z} \frac{z \bigl( \delta(w/z) - \delta(w/\zeta) \bigr)}{z - \zeta}
= z \lim_{\zeta \to z} \frac{\delta(w/z) - \delta(w/\zeta)}{z - \zeta}
= z \partial_z \delta(w/z).
\]
The first equality follows from the fact we consider $z$ to be a constant (\textit{i.e.}, we consider $\zeta$ as the variable).
The second is the definition of the derivative.
\end{proof}

In order to describe the action of the deformed current operators $J_k^{(\bal;\bbe)}$ on the vacuum vectors, we need a more explicit description.
First, we are required to set some additional notation.
The \defn{elementary symmetric functions} and \defn{homogeneous symmetric functions} are
\begin{subequations}
\label{eq:eh_defn}
\begin{align}
e_k(-\bal_{(i,j)}) & = e_k(-\alpha_{i+1}, \dotsc, -\alpha_{j-1}) = (-1)^k \sum_{i < i_1 < \cdots < i_k < j} \alpha_{i_1} \cdots \alpha_{i_k},
\\
h_k(\bbe_{[i,j]}) & = h_k(\beta_i, \dotsc, \beta_j) = \sum_{i \leqslant i_1 \leqslant  \cdots \leqslant  i_k \leqslant j} \beta_{i_1} \cdots \beta_{i_k}.
\end{align}
\end{subequations}
For more on symmetric functions, see \textit{e.g.},~\cite{MacdonaldBook,ECII}.

For $k > 0$, define for $i \neq j$
\begin{subequations}
\label{eq:adef}
\begin{align}
\label{eq:aposdef} A_{ij}^k & = (1 - \alpha_j\beta_j) \sum_{r-s = j-i-k} e_r(-\bal_{(i,j)}) h_s(\bbe_{[i,j]}), \\
\label{eq:aposdef_diag} A_{ii}^k & = \beta_i^k, \\
\label{eq:anegdef} A_{ij}^{-k} & = (1 - \alpha_j\beta_j) \sum_{r-s = j-i+k} h_r(\bal_{[j,i]}) e_s(-\bbe_{(j,i)}), \\
\label{eq:anegdef_diag} A_{ii}^{-k} & = \alpha_i^k,
\end{align}
\end{subequations}
and finally for all $i,j$, define $A^0_{i j} = \delta_{ij}$.
Note that~\eqref{eq:aposdef} (resp.~\eqref{eq:anegdef}) is $0$ whenever $j < i$ (resp.\ $j > i$) since $h_r(0) = e_s(0) = 0$ (\textit{i.e.}, with no variables) for all $r,s > 0$.
Furthermore,~\eqref{eq:aposdef_diag} agrees with~\eqref{eq:aposdef} except for the $(1 - \alpha_j\beta_j)$ factor (and similarly for $A_{ij}^{-k}$).
We give a unified contour integral formula for the values $A_{ij}^k$ for all $k$.

\begin{lemma}
For all $i,j,k \in \ZZ$, we have
\begin{equation}
\label{eq:a_int}
A_{ij}^k = \oint_{\gamma} z^{1-k} (1-\alpha_j\beta_j) \frac{(z|\sigma^i\bal)^{j-i-1}}{(z;\sigma^{i-1}\bbe)^{j-i+1}} \frac{dz}{2\pi\ii z},
\end{equation}
where $\gamma$ is a counterclockwise circle centered at $0$ of radius $\abs{\alpha_i} < r < \abs{\beta_j^{-1}}$ for all $i$ and $j$.
\end{lemma}

\begin{proof}
We first show the claim for $k > 0$ by rewriting Equation~\eqref{eq:aposdef} as an integral:
\begin{align*}
A_{ij}^k & = \oint_{\gamma} z^{j-i-k} (1-\alpha_j\beta_j) \prod_{m=i+1}^{j-1} (1 - \alpha_m z^{-1} ) \prod_{n=i}^j (1 - \beta_n z)^{-1}  \frac{dz}{2\pi\ii z}
\\ &= \oint_{\gamma} z^{1-k} (1-\alpha_j\beta_j) \prod_{m=i+1}^{j-1} (z - \alpha_m) \prod_{n=i}^j (1 - \beta_n z)^{-1}   \frac{dz}{2\pi\ii z}
\\ & = \oint_{\gamma} z^{1-k} (1-\alpha_j\beta_j) \frac{(z|\sigma^i\bal)^{j-i-1}}{(z;\sigma^{i-1}\bbe)^{j-i+1}} \frac{dz}{2\pi\ii z}.
\end{align*}
Note that when $j < i$, then there are no zeros outside the contour (including at $z = \infty$).
Hence, we can shrink the contour to the point at $z = \infty$ (\textit{i.e.}, via substituting $z \mapsto z^{-1}$) and so this contour integral is $0$.
This matches Equation~\eqref{eq:aposdef}.

We have a similar computation from Equation~\eqref{eq:anegdef}:
\begin{align*}
A_{ij}^{-k} & = \oint_{\gamma} z^{j-i+k} (1-\alpha_j\beta_j) \prod_{m=j+1}^{i-1} (1 - \beta_m z^{-1}) \prod_{n=j}^i (1 - \alpha_n z)^{-1}  \frac{dz}{2\pi\ii z}
\\ &= \oint_{\gamma} z^{1+k} (1-\alpha_j\beta_j) \prod_{m=j+1}^{i-1} (z - \beta_m) \prod_{n=j}^i (1 - \alpha_n z)^{-1}   \frac{dz}{2\pi\ii z}
\\ & = \oint_{\gamma} z^{1+k} (1-\alpha_j\beta_j) \frac{(z|\sigma^i\bal)^{j-i-1}}{(z;\sigma^{i-1}\bbe)^{j-i+1}} \frac{dz}{2\pi\ii z}.
\end{align*}
Note that when $j > i$, then there are no poles, and so this contour integral is $0$ as in Equation~\eqref{eq:anegdef}.

For the claim at $k = 0$, then Proposition~\ref{prop:orthonormality} yields
\[
A_{ij}^0 = \delta_{ij} = \oint_{\gamma} (1-\alpha_j\beta_j) \frac{(z|\sigma^i\bal)^{j-i-1}}{(z;\sigma^{i-1}\bbe)^{j-i+1}} \frac{dz}{2\pi\ii}.
\]

The validity of~\eqref{eq:aposdef_diag} and~\eqref{eq:anegdef_diag} are shown analogously to the $k = 0$ proof.
\end{proof}

Disregarding the $(1 - \alpha_j \beta_j)$ factor, the formulas in~\eqref{eq:adef} are the plythesm-equse formulas used in~\cite{HJKSS24,HJKSS24II,IMS24} (with~\eqref{eq:a_int} being their contour integral formulas that appeared in~\cite[Sec.~4.8]{IMS24}).
Additionally, note that we do not require our contour to be a circle, but instead it only needs to contain $0$ and all $\alpha_i$ but not contain any $\beta_j^{-1}$.

\begin{proposition}
\label{prop:explicit_current}
We have
\begin{equation}
\label{eq:jkexpansion}
J_k^{(\bal;\bbe)} = \sum_{i,j} A_{ij}^k \normord{\psi_i \psi_j^*}.
\end{equation}
\end{proposition}

\begin{proof}
We can restrict ourselves to the coefficient of $\normord{\psi_i \psi_j^*}$ in $\normord{\psi(z|\bal;\bbe) \psi^*(z|\bal;\bbe)}$ since they are linearly independent in $\mcC$.
Hence, we want to show the resulting coefficient is equal to $A_{ij}^k$, and by linearity, we obtain
\[
\oint_{\gamma} z^{-k} (1 - \alpha_j\beta_j) z \frac{(z|\bbe)^{i-1}}{(z;\bal)^i} \frac{(z; \bal)^{j-1}}{(z|\bbe)^j} \frac{dz}{2\pi\ii z} = \oint_{\gamma} z^{1-k} (1 - \alpha_j\beta_j) \frac{(z;\sigma^i\bal)^{j-i-1}}{(z|\sigma^{i-1}\bbe)^{j-i+1}} \frac{dz}{2\pi\ii z},
\]
which is precisely Equation~\eqref{eq:a_int}.
\end{proof}

We define the \defn{deformed Hamiltonians} as in the classical case except using the deformed current operators:
\[
H_{\pm}(\pp|\bal;\bbe) := \sum_{k=1}^{\infty} \frac{p_k}{k} J^{(\bal;\bbe)}_{\pm k}.
\]
From Proposition~\ref{prop:current_comm} and linearity, we have
\begin{subequations}
\label{eq:H_field_comm}
\begin{align}
[H_{\pm}(\pp|\bal;\bbe), \psi(z|\bal;\bbe)] & = \sum_{k=1}^{\infty} \frac{p_k}{k} z^{\pm k} \psi(z|\bal;\bbe),
\\
[H_{\pm}(\pp|\bal;\bbe), \psi^*(z|\bal;\bbe)] & = -\sum_{k=1}^{\infty} \frac{p_k}{k} z^{\pm k} \psi^*(z|\bal;\bbe).
\end{align}
\end{subequations}
Furthermore, the analog of~\eqref{eq:half_vertex_commutator} also holds for the deformed Hamiltonians:
\begin{equation}
\label{eq:deformed_half_vertex_commutator}
e^{H_{\pm}(\pp|\bal;\bbe)} e^{H_{\mp}(\pp'|\bal;\bbe)} = e^{\pm\xi(\pp; \pp')} e^{H_{\mp}(\pp'|\bal;\bbe)} e^{H_{\pm}(\pp|\bal;\bbe)}.
\end{equation}
By~\cite[Eq.~(2.4)]{AlexandrovZabrodin} (see also, \textit{e.g.},~\cite[Prop.~3.35]{Hall15}), we have that the deformed current operators are eigenfunctions under the adjoint action of the (deformed) \defn{half vertex operators} $e^{H_{\pm}(\pp|\bal;\bbe)}$.
Note that the function $\xi(\pp; \pp')$ defined in~\eqref{eq:xi_function_def} does not change when using the deformed current operators.
For brevity, set $\xi(\pp; z) := \xi(\pp; z, z^2, \cdots)$.

\begin{theorem}
\label{thm:deformed_adjoint}
We have
\begin{align*}
e^{H_{\pm}(\pp|\bal;\bbe)} \psi(z|\bal;\bbe) e^{-H_{\pm}(\pp|\bal;\bbe)} & = e^{\xi(\pp; z^{\pm 1})} \psi(z|\bal;\bbe),
\\
e^{H_{\pm}(\pp|\bal;\bbe)} \psi^*(w|\bal;\bbe) e^{-H_{\pm}(\pp|\bal;\bbe)} & = e^{-\xi(\pp; w^{\pm 1})} \psi^*(w|\bal;\bbe).
\end{align*}
\end{theorem}

\subsection{Deformed shift operators}

Next, we introduce the \defn{deformed shift operator} as the unique operator $\Sigma_{(\bal;\bbe)}$ that sends $\fermionfock_m \to \widehat{\fermionfock}_{m+1}$ (for all $m$) with the property
\begin{equation}
\label{eq:deformed_shift_fermion}
\Sigma_{(\bal;\bbe)} \psi(z|\bal;\bbe) \Sigma_{(\bal;\bbe)}^{-1} = z^{-1} \psi(z|\bal;\bbe),
\qquad\quad
\Sigma_{(\bal;\bbe)} \psi^*(w|\bal;\bbe) \Sigma_{(\bal;\bbe)}^{-1} = w \psi^*(w|\bal;\bbe),
\end{equation}
with an explicit action on the shifted vacuums that we will define later in~\eqref{eq:deform_shift_vacuums}.
This will allow us to evaluate every matrix coefficient ${}_m \bra{\mu} \Sigma_{(\bal;\bbe)} \ket{\lambda}_{\ell}$ (hence, defining the operator as per Remark~\ref{rem:operator_well_defined}) and give a normalization.
We begin by introducing an ``automorphism'' $\phi$ of $\mcC$ given by
\begin{equation}
\phi \colon \psi(z|\bal;\bbe)\mapsto z^{-1} \psi(z|\bal;\bbe),
\qquad
\psi^*(w|\bal;\bbe) \mapsto w \psi^*(w|\bal;\bbe).
\end{equation}
Indeed, this is an isomorphism of $\mcC$ onto its image since the images of the fermion fields still satisfy the commutation relations~\eqref{eq:deformed_CAR}:
\begin{subequations}
\begin{gather}
[z^{-1}\psi(z|\bal;\bbe), w^{-1}\psi(w|\bal;\bbe)]_+ = [z\psi^*(z|\bal;\bbe), w\psi^*(w|\bal;\bbe)]_+ = 0,
\\
[z^{-1}\psi(z|\bal;\bbe), w\psi^*(w|\bal;\bbe)]_+ = \frac{w}{z}\cdot \delta(w/z) = \delta(w/z).
\end{gather}
\end{subequations}
However, by equating coefficients of $z^k$ and $w^k$, we obtain infinite sums of elements in $\mcC$.
Thus, we first analyze the situation under the following assumption.

\begin{provisional}
	\label{assumption:provisional}
We will assume that $\alpha_i = 0$ for $i \ll 0$ and $\beta_i =0$ for $i\gg 0$.
\end{provisional}

Under this assumption, $\phi$ is a genuine automorphism of $\mcC$, which likewise induces an automorphism of its representation on $\fermionfock$.
Next, we show that this automorphism $\phi$ is inner in the sense that
\begin{equation}
	\label{eq:phisigmarel}
\phi(\xi)= \Sigma_{(\bal;\bbe)}\xi\Sigma_{(\bal;\bbe)}^{-1}
\end{equation}
as operators on $\fermionfock$ for $\xi \in \mcC$.
Thus we proceed under Proposition Assumption~\ref{assumption:provisional} for this discussion, and then we describe how this assumption may be removed.

We describe the action of $\phi$ on the Clifford algebra generators explicitly.

\begin{proposition}
\label{prop:deformed_shift_identities}
For all $i \in \ZZ$, we have
\begin{subequations}
\label{eq:shift_fermions}
\begin{align}
\label{eq:shift_psi}
\phi(\psi_i) & = \alpha_i \psi_i + (1 - \alpha_i \beta_i) \sum_{j=0}^{\infty} \left( \prod_{k=1}^j -\beta_{i+k} \right) \psi_{i+j+1},
\\
\label{eq:shiftinv_psi}
\phi^{-1}(\psi_i) & = \beta_i \psi_i + (1 - \alpha_i \beta_i) \sum_{j=0}^{\infty} \left( \prod_{k=1}^j -\alpha_{i-k} \right) \psi_{i-j-1},
\\
\label{eq:shift_psistar}
\phi(\psi_i^*) & = \beta_i \psi_i^* + \sum_{j=0}^{\infty} \left( \prod_{k=1}^j -\alpha_{i+k} \right) (1 - \alpha_{i+j+1} \beta_{i+j+1}) \psi_{i+j+1}^*,
\\
\label{eq:shiftinv_psistar}
\phi^{-1}(\psi_i^*) & = \alpha_i \psi_i^* + \sum_{j=0}^{\infty} \left( \prod_{k=1}^j -\beta_{i-k} \right) (1 - \alpha_{i-j-1} \beta_{i-j-1}) \psi_{i-j-1}^*.
\end{align}
\end{subequations}
\end{proposition}

\begin{proof}
We first prove~\eqref{eq:shift_psi}.
To do so, we first note that by using the definition of the shifted powers, we have
\begin{align*}
\sum_{i \in \ZZ} z \frac{(z|\bbe)^{i-1}}{(z;\bal)^i} \phi(\psi_i) & = \sum_{i \in \ZZ} \frac{(z|\bbe)^{i-1}}{(z;\bal)^i} \psi_i
\\ & = \sum_{i \in \ZZ} \left( \alpha_i  \frac{z(z|\bbe)^{i-1}}{(z;\bal)^i} + \frac{(z|\bbe)^{i-1}}{(z;\bal)^{i-1}} \right) \psi_i
\\ &  = \sum_{i \in \ZZ} \left( \alpha_i \frac{z(z|\bbe)^{i-1}}{(z;\bal)^i} + \frac{z(z|\bbe)^{i-2}}{(z;\bal)^{i-1}} - \beta_{i-1} \frac{(z|\bbe)^{i-2}}{(z;\bal)^{i-1}} \right) \psi_i.
\end{align*}
Now we can repeat this for the $\frac{(z|\bbe)^{i-2}}{(z;\bal)^{i-1}}$ term, yielding a recursive formula for the coefficient of $\frac{z(z|\bbe)^{j-1}}{(z;\bal)^{j}} \psi_i$ for all $j \leqslant i$.
By equating the $\frac{z(z|\bbe)^{i-1}}{(z;\bal)^i}$ terms on both sides, we obtain~\eqref{eq:shift_psi}.

For~\eqref{eq:shiftinv_psi}, we obtain the recursive formula
\begin{align*}
\sum_{i \in \ZZ} z \frac{(z|\bbe)^{i-1}}{(z;\bal)^i} \phi^{-1}(\psi_i) & = \sum_{i \in \ZZ} \frac{z^2 (z|\bbe)^{i-1}}{(z;\bal)^i} \psi_i
\\ & = \sum_{i \in \ZZ} \left( z \frac{(z|\bbe)^i}{(z;\bal)^i} + \beta_i \frac{z (z|\bbe)^{i-1}}{(z;\bal)^i} \right) \psi_i
\\ &  = \sum_{i \in \ZZ} \left( \beta_i \frac{z (z|\bbe)^{i-1}}{(z;\bal)^i} + z \frac{(z|\bbe)^i}{(z;\bal)^{i+1}} - \alpha_{i+1} \frac{z^2 (z|\bbe)^i}{(z;\bal)^{i+1}} \right) \psi_i,
\end{align*}
and the result follows as before.
The proofs of~\eqref{eq:shift_psistar} and~\eqref{eq:shiftinv_psistar} are similar.
\end{proof}

We will provisionally assume that $\Sigma_{(\bal;\bbe)}$ can be defined and
apply it to a vector. We will heuristically deduce a formula that we can then
take to be the definition. Note that since the $\psi_i$ anticommute, they generate an exterior
algebra, and writing 
\begin{equation}
\label{eq:urvacuum}
\ket{\lambda}_m = \psi_{m+\lambda_1}\psi_{m-1+\lambda_2} \cdots \ket{\varnothing}_{-\infty},
\qquad
{}_m \bra{\lambda} = {}_{-\infty}\bra{\varnothing} \cdots \psi_{m-2+\lambda_3}^* \psi_{m-1+\lambda_2}^* \psi_{m+\lambda_1}^*
\end{equation}
is a suggestive way of modeling the fermionic Fock space within the Clifford
algebra. This is a useful heuristic, though to obtain a rigorous result
we will have to analyze some determinantal formulas that arise from it, and
show that the seemingly infinite determinants that arise are essentially finite
and can be evaluated.

Thus we begin by proceeding formally and heuristically under the assumption that $\Sigma_{(\bal;\bbe)}$ exists and that it stabilizes the fictional ur-vacuum $\ket{\varnothing}_{-\infty}$.
Let us begin with the vacuum $\ket{\varnothing}_m=\psi_m \psi_{m-1} \psi_{m-2} \cdots \ket{\varnothing}_{-\infty}$, though the same method of analysis applies to $\ket{\lambda}_m$.
We move $\Sigma_{(\bal;\bbe)}$ past each $\psi_i$, interpreting the resulting conjugation as the application of $\phi$ by~\eqref{eq:phisigmarel}, using the explicit commutation relations from Proposition~\ref{prop:deformed_shift_identities}.

Remembering that the $\psi_i$ generate an exterior algebra and that $\phi$ is an automorphism of it, we obtain an identity of the form
\begin{subequations}
\label{eq:deform_shift_vacuums_inf}
\begin{align}
\label{eq:deform_shift_vacuum_ket_inf}
\Sigma_{(\bal;\bbe)} \ket{\varnothing}_m & = \sum_{\lambda \in \mcP} \det \left[\eta^{(m-i)}_{\lambda_{j+1}+i-j+1}\right]_{i,j=0}^{\infty} \cdot \ket{\lambda}_{m+1},
\\
{}_m \bra{\varnothing} \Sigma_{(\bal;\bbe)} & = \sum_{\lambda \in \mcP} \det \left[\overline{\eta}^{(m-i)}_{\lambda_{j+1}+i-j+1}\right]_{i,j=0}^{\infty} \cdot {}_{m-1} \bra{\lambda},
 \label{eq:deform_shift_vacuum_bra_inf}
\end{align}
\end{subequations}
where $\phi(\psi_m) = \sum_{k=0}^{\infty} \eta_k^{(m)} \psi_{m+k}$ and $\phi^{-1}(\psi_m^*) = \sum_{k=0}^{\infty} \overline{\eta}^{(m)}_k \psi_{m-k}^*$, with the coefficients $\eta$ or $\overline{\eta}$ coming from Proposition~\ref{prop:deformed_shift_identities}.

These determinants are, \textit{a posteriori}, infinite but in fact they are finite determinants as follows.
Starting with the case where $\lambda = \varnothing$, we will show that actually
\[
\det \left[\eta^{(m-i)}_{i-j+1} \right]_{i,j=0}^{\infty} = \det \left[ \overline{\eta}^{(m-i)}_{i-j+1} \right]_{i,j=0}^{\infty} = 1.
\]
Written out in matrix form, the first determinant looks like
\begin{equation}
	\label{eq:vacdeterminant}
\begin{bmatrix}
(1 - \alpha_m \beta_m) & \alpha_m & 0 & \cdots  & \\
-\beta_m (1 - \alpha_{m-1} \beta_{m-1}) & (1 - \alpha_{m-1} \beta_{m-1}) & \alpha_{m-1} & 0 & \cdots \\
\beta_m \beta_{m-1} (1 - \alpha_{m-2} \beta_{m-2}) & -\beta_{m-1} (1 - \alpha_{m-2} \beta_{m-2}) & (1 - \alpha_{m-2} \beta_{m-2}) & \alpha_{m-2} & 0 & \\
\vdots & \vdots & \vdots & \ddots & \ddots
\end{bmatrix}.
\end{equation}
On the Provisional Assumption~\ref{assumption:provisional}, the $\alpha_i$ are eventually zero for $i \ll 0$ so this has the form
\[
\left(\begin{array}{c|c} A & 0\\\hline B & D\end{array}\right)
\]
where $A$ is a finite matrix, and $B,D$ are infinite but $D$ is lower triangular with $1$'s on the diagonal.
Thus the determinant should be interpreted as $\det(A)$, a finite determinant.\footnote{Such matrices form a group with the determinant being a $1$ dimensional representation.}
Now we claim that $\det(A)=1$.
Indeed, if we add $\beta_{m-j}$ times the $(j+1)$-th column to the $j$-th column, we get the  $j$-th standard basis vector for the $j$-th column (that is, the $j$-th column is $0$ except for a $1$ in the $j$-th row).
The second determinant follows similarly.

Now suppose that $\ell(\lambda)=1$.
Thus $\lambda=(k)$ for some $k\geqslant 1$.
Now the determinant is the same as~\eqref{eq:vacdeterminant} except that the first column is multiplied by $(-1)^k\beta_{m+1}\cdots\beta_{m+k-1}$.
Therefore in this case, the determinant is equal to this constant.
If $\ell(\lambda) > 1$, it is easy to see from~\eqref{eq:shift_fermions} that the first column will be a scalar multiple of the second column (more precisely, the factor will be $(-1)^{\lambda_1-\lambda_2} \prod_{k=0}^{\lambda_1 - \lambda_2} \beta_{m+\lambda_1+k}$).
Therefore, the determinant will be $0$. 
Hence we have formally derived the formulas
\begin{subequations}
\label{eq:deform_shift_vacuums}
\begin{align}
\label{eq:deform_shift_vacuum_ket}
\Sigma_{(\bal;\bbe)} \ket{\varnothing}_m & := \sum_{k=0}^{\infty} (-1)^k \beta_{m+1} \beta_{m+2} \cdots \beta_{m+k} \cdot \ket{k}_{m+1},
\\
\label{eq:deform_shift_vacuum_bra}
{}_m \bra{\varnothing} \Sigma_{(\bal;\bbe)} & := \sum_{k=0}^{\infty} \alpha_m \alpha_{m-1} \cdots \alpha_{m-k+1} \cdot {}_{m-1} \bra{1^k},
\end{align}
\end{subequations}
and we take these \emph{as the definition} of the action of $\Sigma_{(\bal;\bbe)}$ on shifted (dual) vacuums.

The same analysis applies to $\Sigma_{(\bal;\bbe)}\ket{\lambda}_m$ and 
${}_m\bra{\lambda} \Sigma_{(\bal;\bbe)}$: these may be reduced to a determinant by the same formal procedure
starting with (\ref{eq:urvacuum}). Applying $\Sigma_{(\bal;\bbe)}$ to such a vector and heuristically moving the shift operator past the $\psi_i$ will produce a determinant that is the same as (\ref{eq:vacdeterminant}) except for finitely many rows, and which is essentially a finite determinant, at least under Provisional Assumption~\ref{assumption:provisional}.

The procedure by which we have defined $\Sigma_{(\bal;\bbe)}\ket{\lambda}_m$ and ${}_m\bra{\lambda}\Sigma_{(\bal;\bbe)}$ by applying the Clifford automorphism $\phi$ to the formal expressions~\eqref{eq:urvacuum} and extending $\phi(\ket{\varnothing}_{-\infty}) = \ket{\varnothing}_{-\infty}$ guarantees that~\eqref{eq:phisigmarel} is satisfied.
Therefore we obtain~\eqref{eq:deformed_shift_fermion}.

Now we would like to remove the Provisional Assumption~\ref{assumption:provisional}.
Note that this assumption implies that the sums~\eqref{eq:deform_shift_vacuums} and corresponding expressions for $\Sigma_{(\bal;\bbe)}\ket{\lambda}_m$ are \textit{finite} sums so that the operator takes $\fermionfock_m$ into $\fermionfock_{m+1}$.
Without this assumption, this is not strictly true: the operator would take $\fermionfock_m$ into a formal completion of $\fermionfock_{m+1}$.
This is harmless in the sense that if $\lambda$ and $\mu$ are fixed, then only finitely many terms in this infinite sum contribute to any ${}_{m+1} \bra{\mu}\Sigma_{(\bal;\bbe)}\ket{\lambda}_m$ (recall that ${}_{\ell} \bra{\mu} \Sigma_{(\bal;\bbe)} \ket{\lambda}_m = 0$ for all $\ell \neq m + 1$).
Moreover only finitely many $\alpha_i$ and $\beta_i$ are involved in this expression.
This type of behavior is typical in Fock space manipulations, so Provisional Assumption~\ref{assumption:provisional} can be dispensed with.

\begin{example}
We directly verify that the action of $\Sigma_{(\bal;\bbe)}$ respects the relation $\ket{\varnothing}_m = \psi_m \ket{\varnothing}_{m-1}$ by first applying our definitions to obtain
\begin{align*}
\Sigma_{(\bal;\bbe)} \psi_m \ket{\varnothing}_{m-1} & = (1 - \alpha_m \beta_m) \sum_{j=0}^{\infty} (-1)^j \prod_{i=1}^j \beta_{m+i} \psi_{m+j+1} \Sigma_{(\bal;\bbe)} \ket{\varnothing}_{m-1}
\\ & \hspace{20pt} + \alpha_m \psi_m \Sigma_{(\bal;\bbe)} \ket{\varnothing}_{m-1}
\\ & = (1 - \alpha_m \beta_m) \sum_{j=0}^{\infty} (-1)^j \prod_{i=1}^j \beta_{m+i} \psi_{m+j+1} \sum_{k=0}^{\infty} (-1)^k \prod_{i=0}^{k-1} \beta_{m+i} \ket{k}_m
\\ & \hspace{20pt} + \alpha_m \psi_m \sum_{k=0}^{\infty} (-1)^k \prod_{i=0}^{k-1} \beta_{m+i} \ket{k}_m.
\end{align*}
Therefore, we can group this together into a sum over partitions $\lambda$ with $\ell(\lambda) \leqslant 2$.
In the double sum, we note that any $\lambda$ with $\ell(\lambda) = 2$ will occur twice with opposite signs coming from $\psi_{m+\lambda_1+1} \ket{\lambda_2}_m = -\psi_{m+\lambda_2+1} \ket{\lambda_1}_m$.
Hence, we then have
\begin{align*}
\Sigma_{(\bal;\bbe)} \psi_m \ket{\varnothing}_{m-1} = (1 - \alpha_m \beta_m) \sum_{j=0}^{\infty} (-1)^j \prod_{i=1}^j \beta_{m+i} \ket{j}_{m+1} + \alpha_m \sum_{k=0}^{\infty} (-1)^k \prod_{i=0}^k \beta_{m+i} \ket{k}_{m+1},
\end{align*}
where the single summation simplifies by noting $\psi_m \ket{k}_m = (\delta_{km} - 1) \ket{k-1}_{m+1}$ and shifting the indices.
Thus it is now easy to see that this equals $\Sigma_{(\bal;\bbe)} \ket{\varnothing}_m$ as given by~\eqref{eq:deform_shift_vacuum_ket}.
\end{example}

Therefore,~\eqref{eq:deformed_shift_fermion} and~\eqref{eq:deform_shift_vacuums} completely determine $\Sigma_{(\bal;\bbe)}$ as we can construct all of its matrix elements ${}_m \bra{\mu} \Sigma_{(\bal;\bbe)} \ket{\lambda}_{\ell}$.
In particular, we have the following.

\begin{corollary}
\label{cor:deformed_shift_vacuum}
We have
\[
{}_{m+1} \bra{\varnothing} \Sigma_{(\bal;\bbe)} \ket{\varnothing}_{\ell} = \delta_{m\ell}.
\]
\end{corollary}

The construction of $\Sigma_{(\bal;\bbe)}^{-1}$ on the (dual) shifted vacuums in done analogously.
As an alternative way to see that $\Sigma_{(\bal;\bbe)}^{-1}$ exists, we note that
\begin{subequations}
\label{eq:deformed_shift_unitriangular}
\begin{align}
\Sigma_{(\bal;\bbe)} \ket{\lambda}_m & =  \ket{\lambda}_{m+1} + \sum_{\substack{\mu \in \mcP \\ \abs{\mu} > \abs{\lambda}}} C_{\lambda}^{\mu} \cdot \ket{\mu}_{m+1},
\\
{}_m \bra{\lambda} \Sigma_{(\bal;\bbe)} & = {}_{m-1} \bra{\lambda} + \sum_{\substack{\mu \in \mcP \\ \abs{\mu} > \abs{\lambda}}} C_{\lambda}^{\mu} \cdot {}_{m-1} \bra{\mu}.
\end{align}
\end{subequations}
Therefore, restricting $\Sigma_{(\bal;\bbe)}$ to a map from $\fermionfock_m \to \fermionfock_{m+1}$ is clearly an upper triangular (infinite) matrix with $1$ along the diagonal, and hence the matrix is invertible.
Since such matrices uniquely determine elements in $\mcC$, we have $\Sigma_{(\bal;\bbe)}^{-1}$ exists.


We show that the deformed shift operator essentially commutes with the deformed current operators, analogous to the classical case.

\begin{lemma}
\label{lemma:shifted_adjoint_normord}
We have
\[
\Sigma_{(\bal;\bbe)} \, \normord{\psi(z|\bal;\bbe) \psi^*(w|\bal;\bbe)} \, \Sigma_{(\bal;\bbe)}^{-1} = \frac{w}{z} \psi(z|\bal;\bbe) \psi^*(w|\bal;\bbe) - \frac{z}{z - w},
\]
where we consider expansions in $\CC\FPS{w/z}$ or equivalently $\abs{w/z} < 1$.
\end{lemma}

\begin{proof}
	This follows from the definition of the normal ordering~\eqref{eq:fermion_normal_ordering}and~\eqref{eq:deformed_shift_fermion}.
\end{proof}


\begin{proposition}
\label{prop:deformed_current_shift_commute}
We have
\[
\Sigma_{(\bal;\bbe)} J(z|\bal;\bbe) \Sigma_{(\bal;\bbe)}^{-1} = J(z|\bal;\bbe) - 1.
\]
Moreover, $[J_k^{(\bal;\bbe)}, \Sigma_{(\bal;\bbe)}] = 0$ for all $k \neq 0$.
\end{proposition}

\begin{proof}
We note that $\frac{w}{z-w} - \frac{z}{z - w} = -1 \in \CC\FPS{w/z}$, and so by Lemma~\ref{lemma:shifted_adjoint_normord} the result follows analogously to the classical case.
Indeed, we compute
\begin{align*}
\Sigma_{(\bal;\bbe)} J(z|\bal;\bbe) \Sigma_{(\bal;\bbe)}^{-1} & = \Sigma_{(\bal;\bbe)} \, \normord{\psi(z|\bal;\bbe) \psi^*(z|\bal;\bbe)} \, \Sigma_{(\bal;\bbe)}^{-1}
\\ & = \lim_{w\to z} \frac{w}{z} \psi(z|\bal;\bbe) \psi^*(w|\bal;\bbe) - \frac{z}{z-w}
\\ & = \lim_{w\to z} \frac{w}{z} \left( \normord{\psi(z|\bal;\bbe) \psi^*(w|\bal;\bbe)} + \frac{z}{z-w} \right) - \frac{z}{z-w}
\\ & = \lim_{w\to z} \frac{w}{z} \normord{\psi(z|\bal;\bbe) \psi^*(w|\bal;\bbe)} - 1
 = J(z|\bal;\bbe) - 1
\end{align*}
by Equation~\eqref{eq:current_series}, Lemma~\ref{lemma:shifted_adjoint_normord}, and Equation~\eqref{eq:fermion_field_normord}.
\end{proof}

Finally, we show that $K^{(\bal;\bbe)} := \ln \Sigma_{(\bal;\bbe)}$ is the canonically conjugate element to $J_0^{(\bal;\bbe)}$, analogous to the classical case.
We remark that our choice of overall scalar for $\Sigma_{(\bal;\bbe)}$ in defining its action on the shifted vacuums was used in Lemma~\ref{lemma:shifted_adjoint_normord}.

\begin{corollary}
The operator $K^{(\bal;\bbe)}$ is the canonically conjugate element to $J_0^{(\bal;\bbe)}$; that is, $[J_k^{(\bal;\bbe)}, K^{(\bal;\bbe)}] = \delta_{k0}$.
\end{corollary}

\begin{proof}
Proposition~\ref{prop:deformed_current_shift_commute} implies
\[
\Sigma_{(\bal;\bbe)} e^{J(z|\bal;\bbe)} = e^{J(z|\bal;\bbe) - 1} \Sigma_{(\bal;\bbe)} = e^{-1} e^{J(z|\bal;\bbe)} \Sigma_{(\bal;\bbe)},
\]
Therefore, if we write $\Sigma_{(\bal;\bbe)} = e^{K^{(\bal;\bbe)}}$, the BCH formula yields the claim.
\end{proof}

\subsection{Deformed vertex operators}

Our next goal is to construct a vertex operator realization of the deformed fermion fields.
In preparation for this, we need to understand how the deformed current operators act on the vacuums and the dual vacuums (of charge $m$).
From Proposition~\ref{prop:explicit_current}, we can easily compute the action of the deformed current operators on the vacuum elements.
For brevity, we define
\[
\Delta_m(k|\bal) := \sum_{0 < j \leqslant m} \alpha_j^k - \sum_{m < j \leqslant 0} \alpha_j^k,
\]
and we note that at most one of the sums is nonzero.
For future use, we also define
\[
\Lambda_m(\pp|\bal) := \sum_{k=1}^{\infty} \frac{p_k}{k} \Delta_m(k|\bal).
\]

\begin{corollary}
\label{cor:deformed_current_vacuum}
For $k > 0$, we have
\[
J_k^{(\bal;\bbe)} \ket{\varnothing}_m = \Delta_m(k|\bbe) \cdot \ket{\varnothing}_m,
\qquad\qquad
{}_{m} \bra{\varnothing} J_{-k}^{(\bal;\bbe)} = \Delta_m(k|\bal) \cdot {}_m \bra{\varnothing},
\]
and $J_0^{(\bal)} = J_0$ with $J_0^{(\bal)} \ket{\lambda}_m = m \ket{\lambda}_m$ for any partition $\lambda$.
Moreover,
\begin{subequations}
\label{eq:exp_vacuum}
\begin{align}
\label{eq:exp_pos_vacuum}
e^{H_+(\pp|\bal;\bbe)} \ket{\varnothing}_m & = e^{\Lambda_m(\pp|\bbe)} \cdot \ket{\varnothing}_m,
\\
\label{eq:exp_neg_dual_vacuum}
{}_{m} \bra{\varnothing} e^{H_-(\pp|\bal;\bbe)} & = e^{\Lambda_m(\pp|\bal)} \cdot {}_{m} \bra{\varnothing}.
\end{align}
\end{subequations}
\end{corollary}

\begin{proof}
Consider $J_k^{(\bal;\bbe)} \ket{\varnothing}_m$.
Note that we must have $j \geqslant i$ for $A_{ij}^k \neq 0$, and hence only the action of $A_{ii}^k \, \normord{\psi_i \psi_i^*}$ can give a nonzero result.
Therefore, the result follows from the definition of the normal order (analogous to how $J_0 \ket{\lambda}_m = m$).
In more detail, if $i > 0$, then we would have $A_{ii}^k \psi_i \psi_i^* \ket{\varnothing}_m$, which is $0$ if $i > m$ and $A_{ii}^k \ket{\varnothing}_m$ otherwise.
On the other hand, for $i < 0$, then this is $-A_{ii}^k \psi_i^* \psi_i \ket{\varnothing}_m$, which is $-A_{ii}^k \ket{\varnothing}_m$ if $i > m$ and $0$ otherwise.

The proof for ${}_m \bra{\varnothing} J_{-k}^{(\bal;\bbe)}$ is similar.
Immediate from the definition of the deformed Hamiltonians, we obtain~\eqref{eq:exp_vacuum}.
\end{proof}

Note that $J_k^{(\bal;\bbe)} \ket{\varnothing}_m \neq 0$ and ${}_m \bra{\varnothing} J_{-k}^{(\bal;\bbe)} \neq 0$ for all $m \neq 0$ and all $k > 0$ in contrast to the classical $\bal = \bbe = 0$ case.

\begin{lemma}
\label{lemma:exp_lambda}
Set $p_k = x^k - (-y)^k$. As a formal power series in $\CC[\bbe]\FPS{x,y}$, the coefficient in~\eqref{eq:exp_vacuum} becomes
\begin{equation}
\label{eqn-exp-lambda}
e^{\Lambda_m(\pp|\bbe)} = \begin{cases}
\prod_{0 < j \leqslant m} \frac{1 + \beta_j y}{1 - \beta_j x} & \text{if $m>0$,}
\\[5pt] \prod_{m < j \leqslant 0}\frac{1 - \beta_j x}{1 + \beta_j y} & \text{if $m\leqslant 0$.}
\end{cases}
\end{equation}
\end{lemma}

\begin{proof}
Setting $p_k = p_k(x / y) = x^k - (-y)^k$ and using $- \log (1 - t) = \sum \frac{1}{k} t^k$ we get
    \[
    \Lambda_m(\pp|\bbe) = \sum_{k=1}^{\infty} \sum_{0 < j \leqslant m} \frac{1}{k}(x^k - (-y)^k) \beta_j^k = \sum_{0 < j \leqslant m} \log \left(\frac{1 + \beta_j y}{1 - \beta_j x} \right)
    \]
    when $m>0$ and
    \[
   \Lambda_m(\pp|\bbe) = -\sum_{k=1}^{\infty}\sum_{m < j \leqslant 0} \frac{1}{k}(x^{- k} - (- y)^{- k}) \beta_j^k = \sum_{m < j \leqslant 0} \log \left(\frac{1 - \beta_j x}{1 + \beta_j y} \right)
    \]
    when $m\leqslant 0$.
\end{proof}

In particular, for $p_k = z^k$, the coefficient in~\eqref{eq:exp_vacuum} becomes
\begin{equation}
\label{eq:half_vertex_vacuum_actions}
e^{\Lambda_m(z|\bal)} = \frac{\prod_{m < j \leqslant 0} (1 - \alpha_j z)}{\prod_{0 < j \leqslant m} (1 - \alpha_j z)} = \frac{1}{(z;\bal)^m}.
\end{equation}

We can now prove our main result for this section, that the deformed fermion fields can be described as a vertex operator using the deformed current operators and deformed shift operator.

\begin{theorem}[Fermion vertex operators]
\label{thm:fermion_vertex_op}
We have
\begin{subequations}
\label{eq:fermion_vertex_ops}
\begin{align}
\psi(z|\bal;\bbe) & = e^{H_-(z|\bal;\bbe)} z^{J_0^{(\bal;\bbe)}} \Sigma_{(\bal;\bbe)} e^{-H_+(z^{-1}|\bal;\bbe)},
\\
\psi^*(w|\bal;\bbe) & = e^{-H_-(w|\bal;\bbe)} \Sigma_{(\bal;\bbe)}^{-1} w^{-J_0^{(\bal;\bbe)}} (1 - \alpha_{J_0^{(\bal;\bbe)}}\beta_{J_0^{(\bal;\bbe)}}) e^{H_+(w^{-1}|\bal;\bbe)}.
\end{align}
\end{subequations}
\end{theorem}

\begin{proof}
Our proof is analogous to~\cite[Eq.~(2.90)]{AlexandrovZabrodin}.
In more detail, we prove our claim by showing they agree for all matrix elements; that is
\begin{align*}
{}_m \bra{\mu} \psi(z|\bal;\bbe) \ket{\lambda}_{\ell} & = {}_m \bra{\mu} e^{H_-(z|\bal;\bbe)} z^{J_0^{(\bal;\bbe)}} \Sigma_{(\bal;\bbe)} e^{-H_+(z^{-1}|\bal;\bbe)} \ket{\lambda}_{\ell},
\\
{}_m \bra{\mu} \psi^*(w|\bal;\bbe) \ket{\lambda}_{\ell} & = {}_m \bra{\mu} e^{-H_-(w|\bal;\bbe)} \Sigma_{(\bal;\bbe)}^{-1} (1 - \alpha_{J_0^{(\bal;\bbe)}}\beta_{J_0^{(\bal;\bbe)}}) w^{-J_0^{(\bal;\bbe)}} e^{H_+(w^{-1}|\bal;\bbe)} \ket{\lambda}_{\ell}.
\end{align*}
We only prove the first identity as the proof for the second is similar.
We claim every such pairing is generated from
\[
{}_m \bra{\varnothing} e^{H_+(\pp|\bal;\bbe)},
\qquad\qquad
e^{H_-(\pp'|\bal;\bbe)} \ket{\varnothing}_{\ell},
\]
which we consider as a multivariate generating series in $\pp$ and $\pp'$.
This claim follows from the classical setting as taking $\bal = \bbe = 0$ results in nonzero linearly independent coefficients for each basis element $\ket{\lambda}_{\ell}$ (which are the Schur functions)~\cite{JM83} (see also, \textit{e.g.},~\cite{AlexandrovZabrodin,KacInfinite,KacRaina,MJD00}).

Based on this, we compute
\begin{align*}
{}_m \bra{\varnothing} e^{H_+(\pp|\bal;\bbe)} & \psi(z|\bal;\bbe) e^{H_-(\pp'|\bal;\bbe)} \ket{\varnothing}_{\ell}
\\ & = e^{\xi(\pp; z)} \cdot {}_m \bra{\varnothing}\psi(z|\bal;\bbe) e^{H_+(\pp|\bal;\bbe)} e^{H_-(\pp'|\bal;\bbe)}\ket{\varnothing}_{\ell}
\\ & = e^{\xi(\pp; z) - \xi(\pp'; z^{-1}) + \xi(\pp; \pp') + \Lambda_m(\pp' | \bal) + \Lambda_{\ell}(\pp|\bal)} \cdot {}_m \bra{\varnothing} \psi(z|\bal;\bbe) \ket{\varnothing}_{\ell}
\\ & = e^{\xi(\pp; z) - \xi(\pp'; z^{-1}) + \xi(\pp; \pp') + \Lambda_m(\pp' | \bal) + \Lambda_{\ell}(\pp|\bal)} \frac{z (z|\bbe)^{m-1}}{(z;\bal)^m} \delta_{m,\ell+1}
\end{align*}
by applying Theorem~\ref{thm:deformed_adjoint} twice, Equation~\eqref{eq:deformed_half_vertex_commutator}, and Equations~\eqref{eq:exp_vacuum}.
On the other hand, we compute
\begin{align*}
{}_m \bra{\varnothing} & e^{H_+(\pp|\bal;\bbe)} e^{H_-(z|\bal;\bbe)} z^{J_0^{(\bal;\bbe)}} \Sigma_{(\bal;\bbe)} e^{-H_+(z^{-1}|\bal;\bbe)} e^{H_-(\pp'|\bal;\bbe)} \ket{\varnothing}_{\ell}
\\ & = e^{\xi(\pp; z) - \xi(\pp'; z^{-1}) + \Lambda_m(z|\bal) - \Lambda_{\ell}(z^{-1}|\bbe)} \cdot {}_m  \bra{\varnothing} e^{H_+(\pp|\bal;\bbe)} z^{J_0^{(\bal;\bbe)}} \Sigma_{(\bal;\bbe)} e^{H_-(\pp'|\bal;\bbe)} \ket{\varnothing}_{\ell}
\\ & = e^{\xi(\pp; z) - \xi(\pp'; z^{-1}) + \Lambda_m(\pp'|\bal) + \Lambda_m(\pp|\bbe)} z^{\ell+1} e^{\Lambda_m(z|\bal)} e^{-\Lambda_{\ell}(z^{-1}|\bbe)} \cdot {}_m \bra{\varnothing} \Sigma_{(\bal;\bbe)} \ket{\varnothing}_{\ell}
\\ & = e^{\xi(\pp; z) - \xi(\pp'; z^{-1}) + \xi(\pp; \pp') + \Lambda_m(\pp' | \bal)} \frac{z (z|\bbe)^{m-1}}{(z; \bal)^m} \delta_{m,\ell+1}
\end{align*}
by using Proposition~\ref{prop:deformed_current_shift_commute}, Equation~\eqref{eq:half_vertex_commutator}, Equations~\eqref{eq:exp_vacuum}, Equation~\eqref{eq:half_vertex_vacuum_actions}, and Corollary~\ref{cor:deformed_shift_vacuum}.
\end{proof}

\subsection{Additional remarks}
\label{sec:fermionic_remarks}

We conclude this section with some additional remarks and identities involving the deformed fermion fields.
The first is that we almost have the analog of~\eqref{eq:duality_fermion_fields}:
\begin{equation}
\label{eq:duality_deformed_fields}
\bigl( \psi(z|\bal;\bbe) \bigr)^* = \eta \psi^*(z^{-1}|\bbe; \bal),
\qquad\qquad
\bigl( \psi^*(w|\bal;\bbe) \bigr)^* = \eta \psi(w^{-1}|\bbe; \bal),
\end{equation}
where we have to additionally apply the rescaling algebra morphism $\eta$ defined by
\[
\psi_i \mapsto (1 - \alpha_i \beta_i)^{-1} \psi_i
\qquad \text{ and } \qquad
\psi_j^* \mapsto (1 - \alpha_j \beta_j) \psi_j^*.
\]
We could make our formulas more symmetric (and thus remove the $\eta$ map) by applying the rescaling $\sqrt{\eta}$.
All of our results in this paper obviously hold using these adjoint operators since this is clearly an automorphism of the Clifford algebra.
However, some of the formulas will change slightly.
In particular, such rescaling will change our definition of the current operators, which we will see in the sequel that we have chosen to match a solvable lattice model given in~\cite{NaprienkoFFS}.
It will also change the definition of the deformed shift operator $\Sigma_{(\bal;\bbe)}$, and in the rescaled setting, we would have $\widetilde{\Sigma}_{(\bal;\bbe)}^* = \widetilde{\Sigma}_{(\bal;\bbe)}^{-1}$ (which is not quite true with our definition).

From~\eqref{eq:duality_deformed_fields}, we also have the analog of the classical $J_k^* = J_{-k}$.

\begin{corollary}
\label{cor:adjoint_current}
For all $k$, we have
\begin{equation}
\label{eq:current_duality}
\left(J_k^{(\bal;\bbe)} \right)^* = \eta J_{-k}^{(\bbe;\bal)}.
\end{equation}
\end{corollary}

\begin{proof}
The claim follows by examining $\normord{\psi_i \psi_j^*}$ in $J(z|\bal;\bbe)$ and $\bigl( J(z|\bbe;\bal) \bigr)^*$.
\end{proof}

\begin{proof}[Alternative proof]
We can see how the coefficients $A_{ij}^k$ compare with $A_{ji}^{-k}$ from~\eqref{eq:adef} under the map that sends $\bal \to \bbe$ and $\bbe \to \bal$.
\end{proof}

Next, we have the deformed version of~\eqref{eq:chiral_boson_identities} using the \defn{deformed chiral boson field}:
\[
\phi(z|\bal;\bbe) := H_-(z|\bal;\bbe) + K^{(\bal;\bbe)} + J_0^{(\bal;\bbe)} \ln z - H_+(z|\bal;\bbe)
\]
(recall $K^{(\bal;\bbe)} = \ln \Sigma_{(\bal;\bbe)}$),
and Theorem~\ref{thm:fermion_vertex_op} implies
\[
\psi(z|\bal;\bbe) = \bnormord{e^{\phi(z|\bal;\bbe)}},
\qquad\;
\psi^*(z|\bal;\bbe) = \bnormord{e^{-\phi(z|\bal;\bbe)}},
\qquad\;
z \frac{\partial}{\partial z} \phi(z|\bal;\bbe) = J(z|\bal;\bbe).
\]

Let us discuss how the classical shift operator $\Sigma$ interacts with our deformed constructions.
We show that conjugation by the shift operator $\Sigma$ essentially acts as in the classical case with additionally shifting the parameters (that is, applying the map $\sigma^{-1}$).

\begin{proposition}
\label{prop:deformed_fermion_classical_shift}
We have
\begin{align*}
\Sigma \psi(z|\bal;\bbe) \Sigma^{-1} & = \frac{(z;\sigma^{-1}_{\bal}\bal)}{(z|\sigma^{-1}_{\bbe}\bbe)} \psi(z|\sigma_{\bal}^{-1}\bal;\sigma_{\bbe}^{-1}\bbe),
\\
\Sigma \psi^*(w|\bal;\bbe) \Sigma^{-1} & = \frac{(w|\sigma^{-1}_{\bbe}\bbe)}{(w;\sigma^{-1}_{\bal}\bal)} \psi^*(w|\sigma_{\bal}^{-1}\bal;\sigma_{\bbe}^{-1}\bbe).
\end{align*}
\end{proposition}

\begin{proof}
We compute using~\eqref{eq:basic_spowers_rels}
\begin{align*}
\Sigma \psi(z|\bal;\bbe) \Sigma^{-1} & = \sum_{i \in \ZZ} \frac{z(z|\bal)^{i-1}}{(z;\bal)^i} \Sigma \psi_i \Sigma^{-1}
= \sum_{i \in \ZZ} \frac{z(z|\bal)^{i-1}}{(z;\bal)^i} \psi_{i+1}
= \sum_{i \in \ZZ} \frac{z(z|\bal)^{i-2}}{(z;\bal)^{i-1}} \psi_i
\\ & = \sum_{i \in \ZZ} \frac{(z|\bbe)^{-1}(z|\sigma_{\bbe}^{-1}\bbe)^{i-1}}{(z;\bal)^{-1} (z;\sigma_{\bal}^{-1} \bal)^i} \psi_i
= \frac{(z;\sigma^{-1}_{\bal}\bal)}{(z|\sigma^{-1}_{\bbe}\bbe)} \psi(z|\sigma_{\bal}^{-1}\bal;\sigma_{\bbe}^{-1}\bbe),
\end{align*}
and the second identity is shown similarly.
\end{proof}

\begin{proposition}
\label{prop:deformed_current_classical_shift}
We have
\[
\Sigma J(z|\bal;\bbe) \Sigma^{-1} = J(z|\sigma_{\bal}^{-1}\bal;\sigma_{\bbe}^{-1}\bbe) - \sum_{k=1}^{\infty} \alpha_0^k z^k - \sum_{k=1}^{\infty} \beta_0^k z^{-k}.
\]
\end{proposition}

\begin{proof}
This is shown similar to the proof of Proposition~\ref{prop:deformed_current_shift_commute} except we have
\[
\frac{(z;\sigma_{\bal}^{-1}\bal)}{(w;\sigma_{\bal}^{-1}\bal)} \frac{(w|\sigma_{\bbe}^{-1}\bbe)}{(z|\sigma_{\bbe}^{-1}\bbe)} \frac{z}{z-w} - \frac{z}{z-w} = \frac{(\alpha_0 \beta_0 - 1) z}{(1 - \alpha_0 w)(z - \beta_0)} = \frac{\alpha_0 \beta_0 - 1}{(1 - \alpha_0 w)(1 - \beta_0 z^{-1})},
\]
which we then expand as a formal power series in $\CC\FPS{w,z^{-1}}$.
Then we take the limit as $w \to z$ and collecting coefficients yields the claim as the series telescopes provided that $\lim_{n\to\infty} (\alpha_0 \beta_0)^n = 0$.

Alternatively, we can prove it directly from $J_k^{(\bal;\bbe)} = \sum_{i,j \in \ZZ} A_{ij}^k \normord{\psi_i \psi_j^*}$ or expanding
\[
J(z|\sigma_{\bal}^{-1}\bal;\sigma_{\bbe}^{-1}\bbe) = \normord{\psi(z|\sigma_{\bal}^{-1}\bal;\sigma_{\bbe}^{-1}\bbe) \psi^*(z|\sigma_{\bal}^{-1}\bal;\sigma_{\bbe}^{-1}\bbe)}
\]
and being careful about the normal ordering.
\end{proof}

Indeed, for $k > 0$ we can see that
\[
\Sigma J_k^{(\bal;\bbe)} \ket{\lambda}_m = (J_k^{(\sigma_{\bal}^{-1}\bal;\sigma_{\bbe}^{-1}\bbe)} - \beta_0^k) \ket{\lambda}_{m+1},
\qquad
\Sigma J_{-k}^{(\bal;\bbe)} \ket{\lambda}_m = (J_{-k}^{(\sigma_{\bal}^{-1}\bal;\sigma_{\bbe}^{-1}\bbe)} - \alpha_0^k) \ket{\lambda}_{m+1}.
\]

\section{Deformed boson-fermion correspondences}
\label{sec:deformed_correspondence}

In this section, we will construct a deformed version of the boson-fermion correspondence using the deformed Hamiltonians and deformed (dual) shifted vacuums.
We will see that the bosonic vertex operators are the same as in the classical case, but we will be considering the expansion into the appropriate shifted modes.

Throughout this section we denote $Q := \svar \frac{\partial}{\partial \svar}$.

\subsection{The deformed correspondence}
\label{sec:basic_BF}

From Theorem~\ref{thm:heisenberg_relations} with the BCH formula, Proposition~\ref{prop:deformed_current_shift_commute}, Corollary~\ref{cor:deformed_current_vacuum}, and Theorem~\ref{thm:fermion_vertex_op}, we have the following deformation of the boson-fermion correspondence.

\begin{theorem}[The deformed boson-fermion correspondence via deformed shifted vacuums]
\label{thm:deformed_boson_fermion}
There is an isomorphism of $\mcH$ algebras $\Phi^{(\bal;\bbe)} \colon \fermionfock \to \bosonfock$ defined by
\[
\ket{\eta} \longmapsto \sum_{m \in \ZZ} {}_{(m)} \bra{\varnothing} e^{H_+(\pp|\bal;\bbe)} \ket{\eta} \cdot \svar^m,
\]
where we define ${}_{(m)} \bra{\varnothing} := \bra{\varnothing} \Sigma_{(\bal;\bbe)}^{-m}$, for any $\ket{\eta} \in \fermionfock$.
In particular, $\ket{\varnothing}_{(m)} := \Sigma_{(\bal;\bbe)}^m \ket{\varnothing} \mapsto \svar^m$.
The map $\Phi^{(\bal;\bbe)}$ induces a mapping of operators given by, for $k > 0 $,
\begin{gather*}
J_k^{(\bal;\bbe)} \mapsto k \frac{\partial}{\partial p_k},
\qquad\qquad
J_{-k}^{(\bal;\bbe)} \mapsto p_k,
\qquad\qquad
J_0^{(\bal;\bbe)} \mapsto Q,
\qquad\qquad
\Sigma_{(\bal;\bbe)} \mapsto \svar,
\\
\begin{aligned}
\psi(z|\bal;\bbe) & \mapsto X(z|\bal;\bbe) := e^{\overline{H}_-(z|\bal;\bbe)} z^Q \svar e^{-\overline{H}_+(z^{-1}|\bal;\bbe)},
\\
\psi^*(w|\bal;\bbe) & \mapsto X^*(w|\bal;\bbe) := e^{-\overline{H}_-(w|\bal;\bbe)} \svar^{-1} w^{-Q} (1 - \alpha_Q \beta_Q) e^{\overline{H}_+(w^{-1}|\bal;\bbe)},
\end{aligned}
\end{gather*}
where $\overline{H}_{\pm}(z|\bal;\bbe)$ is the image of $H_{\pm}(z|\bal;\bbe)$ under $\Phi^{(\bal;\bbe)}$.
\end{theorem}

Before giving the details of the proof of Theorem~\ref{thm:deformed_boson_fermion}, let us make some remarks.
Note that $\Phi^{(\bal;\bbe)}$ preserves the charge.
Our next remark is that
\begin{equation}
\label{eq:explicit_image_H}
\overline{H}_+(z^{-1}|\bal;\bbe) = \xi(\widetilde{\partial}_{\pp};z^{-1}) = \sum_{k=1}^{\infty} \frac{\partial}{\partial p_k}z^{-k},
\qquad\qquad
\overline{H}_-(z|\bal;\bbe) = \xi(\pp;z) = \sum_{k=1}^{\infty} \frac{p_k}{k}z^k,
\end{equation}
where $\widetilde{\partial}_{\pp} = (\partial_{p_1}, 2 \partial_{p_2}, 3 \partial_{p_3}, \cdots)$ using the standard $\partial_x := \frac{\partial}{\partial x}$.
Finally, the vertex operators $X(z|\bal;\bbe) = X(z)$ and $X^*(w|\bal;\bbe) =
X^*(w)$ are the same as those in Section~\ref{sec:boson_fermion} and do not
\emph{appear} to depend on $\bal$ or $\bbe$. 
However, we will be considering their expansions into shifted modes, which
\emph{do} depend on $\bal$ and $\bbe$, rather than normal Fourier modes to
match our deformed fermion fields.

\begin{proof}
We first give the details about the image of the deformed current operators under $\Phi^{(\bal;\bbe)}$.
Let $k > 0$, and we compute
\[
k \frac{\partial}{\partial p_k} \bigl[{}_{(m)} \bra{\varnothing} e^{H_+(\pp|\bal;\bbe)} \ket{\eta} \bigr]
= {}_{(m)} \bra{\varnothing} k \frac{\partial}{\partial p_k} [e^{H_+(\pp|\bal;\bbe)}] \ket{\eta}
= {}_{(m)} \bra{\varnothing} e^{H_+(\pp|\bal;\bbe)} J_k^{(\bal;\bbe)} \ket{\eta}.
\]
This follows from $\frac{d}{dx} e^{\lambda x} = \lambda e^{\lambda x}$, which could also be seen more directly from the expansion $e^{H_+(\pp|\bal;\bbe)} = \sum_{i=0}^{\infty} \frac{H_+(\pp|\bal;\bbe)^i}{i!}$ as we can do one fewer $J_k^{(\bal;\bbe)}$ action.
Next we consider $k = 0$, and since $e^{H_+(\pp|\bal;\bbe)}$ has charge $0$, we have
\begin{align*}
\sum_{m\in\ZZ} {}_{(m)} \bra{\varnothing} e^{H_+(\pp|\bal;\bbe)} J_0^{(\bal;\bbe)} \ket{\eta} \cdot \svar^m & = \sum_{m\in\ZZ} {}_{(m)} \bra{\varnothing} e^{H_+(\pp|\bal;\bbe)} \ket{\eta} \cdot m \svar^m
\\ & = \svar\frac{\partial}{\partial \svar} \sum_{m\in\ZZ} {}_{(m)} \bra{\varnothing} e^{H_+(\pp|\bal;\bbe)} \ket{\eta} \cdot \svar^m
\end{align*}
by Theorem~\ref{thm:heisenberg_relations} and the BCH formula and applying $\Sigma_{(\bal;\bbe)}$ to ${}_m \bra{\lambda}$ results in a homogeneous vector of charge $m-1$ by~\eqref{eq:deform_shift_vacuum_bra}. 
Note that $e^{H_+(\pp|\bal;\bbe)} J_{-k}^{(\bal;\bbe)} = (p_k + J_{-k}^{(\bal;\bbe)}) e^{H_+(\pp|\bal;\bbe)}$ from Theorem~\ref{thm:heisenberg_relations} and the BCH formula (see also, \text{e.g.},~\cite[Eq.~(2.4)]{AlexandrovZabrodin}).
Also, note that
\[
{}_{(m)} \bra{\varnothing} J_{-k}^{(\bal;\bbe)} = \bra{\varnothing} J_{-k}^{(\bal;\bbe)} \Sigma_{(\bal;\bbe)}^m = 0
\]
by Corollary~\ref{cor:deformed_current_vacuum}, and hence
\[
{}_{(m)} \bra{\varnothing} e^{H_+(\pp|\bal;\bbe)} J_{-k}^{(\bal;\bbe)} \ket{\eta} = {}_m \bra{\varnothing} (p_k + J_{-k}^{(\bal;\bbe)}) e^{H_+(\pp|\bal;\bbe)} \ket{\eta}
= p_k \cdot {}_{(m)} \bra{\varnothing} e^{H_+(\pp|\bal;\bbe)} \ket{\eta}.
\]
This completes the proof of the image of the deformed current operators under $\Phi^{(\bal;\bbe)}$.

To finish the proof of Theorem~\ref{thm:deformed_boson_fermion}, we have $\Sigma_{(\bal;\bbe)} \mapsto s$ by computing
\begin{align*}
\sum_{m \in \ZZ} {}_{(m)} \bra{\varnothing} e^{H_+(\pp|\bal;\bbe)} \Sigma_{(\bal;\bbe)} \ket{\eta} \cdot \svar^m
& = \sum_{m \in \ZZ} {}_{(m-1)} \bra{\varnothing} e^{H_+(\pp|\bal;\bbe)} \ket{\eta} \cdot \svar^m
\\ & = \sum_{m \in \ZZ} {}_{(m)} \bra{\varnothing} e^{H_+(\pp|\bal;\bbe)} \ket{\eta} \cdot \svar^{m+1}
\end{align*}
by Proposition~\ref{prop:deformed_current_shift_commute}.
Likewise,~\eqref{eq:deform_shift_vacuum_ket} and $e^{H_+(\pp|\bal;\bbe)} \ket{\varnothing}_{(m)} = \ket{\varnothing}_{(m)}$ implies $\ket{\varnothing}_{(m)} \mapsto \svar^m$.
Hence, the identities~\eqref{eq:fermion_vertex_ops} imply $\psi(z|\bal;\bbe) \mapsto X(z|\bal;\bbe)$ and $\psi^*(w|\bal;\bbe) \mapsto X^*(w|\bal;\bbe)$ under $\Phi^{(\bal;\bbe)}$.
\end{proof}

\begin{remark}
\label{rem:alt_proof_bfc_kac}
We give an alternative way to constructing the boson vertex operators without using Theorem~\ref{thm:fermion_vertex_op} by following Kac~\cite[Ch.~14]{KacInfinite}.
For this, note $\Phi^{(\bal;\bbe)}$ is a canonical map identifying $\fermionfock$ and $\bosonfock$ as representations of $\mcH$, where each irreducible $\mcH$ representation is distinguished by its charge with highest weight vector $\Phi^{(\bal;\bbe)}(\ket{\varnothing}_{(m)}) = \svar^m$.
We remark that the irreducibility of $\fermionfock_m$ follows from the $\bal = \bbe = 0$ case~\cite[Prop.~4.9]{KacInfinite}.
Hence, exactly analogous to~\cite[Thm.~14.10(a)]{KacInfinite} using Proposition~\ref{prop:current_comm} and~\cite[Lemma~14.5]{KacInfinite}, we have that $\Phi^{(\bal;\bbe)}$ maps $\psi(z|\bal;\bbe)$ and $\psi^*(z|\bal;\bbe)$ to the vertex operators
\begin{align*}
X(z|\bal;\bbe) & = C_Q(z) \svar e^{\overline{H}_-(z|\bal;\bbe)} e^{-\overline{H}_+(z^{-1}|\bal;\bbe)},
\\
X^*(w|\bal;\bbe) & = \svar^{-1} C^*_Q(w) e^{-\overline{H}_-(w|\bal;\bbe)} e^{\overline{H}_+(w^{-1}|\bal;\bbe)},
\end{align*}
respectively.
To compute $C_Q(z)$, we bring the vertex operator description to back to operators on $\fermionfock$ and compute the coefficient of $\ket{\varnothing}_m$ in the expansion by
\begin{align*}
{}_m \bra{\varnothing} \psi(z|\bal;\bbe) \ket{\varnothing}_{m-1} & = \frac{z(z|\bbe)^{m-1}}{(z;\bal)^m}
\\ & = C_m(z) z^{1-m} \frac{(z|\bbe)^{m-1}}{(z;\bal)^m} \cdot {}_m \bra{\varnothing} \Sigma_{(\bal;\bbe)} \ket{\varnothing}_{m-1} = C_m(z) z^{-m} \frac{z (z|\bbe)^{m-1}}{(z;\bal)^m}
\end{align*}
by Corollary~\ref{cor:deformed_current_vacuum}, Equation~\eqref{eq:half_vertex_vacuum_actions}, and Corollary~\ref{cor:deformed_shift_vacuum}.
Hence, $C_Q(z) = z^Q$, and the computation for $C^*_Q(w)$ is similar.
\end{remark}

\subsection{Double factorial Schur functions}

We want to look at the image $\Phi^{(\bal;\bbe)}(\ket{\lambda}_m)$ for any partition $\lambda$ and $m \in \ZZ$.
We begin by computing the following generating series.

\begin{lemma}
\label{lemma:fermion_expectation}
We have
\begin{align*}
{}_{\ell} \bra{\varnothing} e^{H_+(\pp|\bal;\bbe)} & \psi(z_1|\bal;\bbe) \cdots \psi(z_n|\bal;\bbe) \ket{\varnothing}_{m-n}
\\ & = \delta_{m\ell} \prod_{i=1}^n z_i e^{\xi(\pp;z_i)} \frac{(z_i|\bbe)^{m-n}}{(z_i;\bal)^m} \prod_{i < j} (z_i - z_j) (1 - \alpha_{1+m-i}\beta_{1+m-j}) e^{\Lambda_{m-n}(\pp|\bbe)}.
\end{align*}
\end{lemma}

\begin{proof}
Recall from Theorem~\ref{thm:heisenberg_relations} that $[J_0^{(\bal;\bbe)}, J_k^{(\bal;\bbe)}] = 0$ for all $k \neq 0$, so $e^{H_+(\pp|\bal;\bbe)}$ preserves the charge.
Hence, we must have $m = \ell$ as otherwise we are taking the pairing of two elements of different charges, which is $0$.

Next, we prove the result when $m = 0$.
We compute
\begin{align*}
\bra{\varnothing} e^{H_+(\pp|\bal;\bbe)} & \psi(z_1|\bal;\bbe) \cdots \psi(z_n|\bal;\bbe) \ket{\varnothing}_{-n}
\\ & = \prod_{i=1}^n e^{\xi(\pp;z_i)} e^{\Lambda_{-n}(\pp|\bbe)} \bra{\varnothing} \psi(z_1|\bal;\bbe) \cdots \psi(z_n|\bal;\bbe) \ket{\varnothing}_{-n}
\\ & = e^{\Lambda_{-n}(\pp|\bbe)} \prod_{i=1}^n e^{\xi(\pp;z_i)} \sum_{\omega \in S_n} (-1)^{\omega} \frac{z_1 (z_1|\bbe)^{-\omega_1} \cdots z_n (z_n|\bbe)^{-\omega_n}}{(z_1;\bal)^{1-\omega_1} \cdots (z_n; \bal)^{1-\omega_n}}
\\ & = e^{\Lambda_{-n}(\pp|\bbe)} \prod_{i=1}^n z_i e^{\xi(\pp;z_i)} (z_i|\bbe)^{-n} \prod_{i < j} (z_i - z_j) (1 - \alpha_{1-i} \beta_{1-j}),
\end{align*}
where the last identity comes from realizing that the sum is the shifted power Vandermonde determinant from~\cite[Prop.~3.15]{NaprienkoFFS} with taking $z_i \mapsto x_i^{-1}$ (see also~\cite[Sec.~2.1]{Molev09} and~\cite[Prop.~1]{Krattenthaler99}), which is also~\cite[Eq.~(2.22)]{MiyauraMukaihiraGeneralized} (stated with no proof).
Now to show the claim for general $m$, by using~\eqref{eq:shifting_powers}, we can see that
\begin{align*}
{}_m \bra{\varnothing}\psi(z_1|\bal) \cdots \psi(z_n|\bal) \ket{\varnothing}_{m-n} & = \prod_{i=1}^n \frac{(z_i|\bbe)^m}{(z_i;\bal)^m} \sigma_{\bal}^m \sigma_{\bbe}^m \bigl( \bra{\varnothing}\psi(z_1|\bal) \cdots \psi(z_n|\bal) \ket{\varnothing}_{-n} \bigr).
\end{align*}
Hence the claim follows from the $m = 0$ case.
\end{proof}


\begin{proof}[Alternative proof]
We could show the claim for $m \neq 0$ from the $m = 0$ case by repeated applications of the shift operator, where repeated applications of Proposition~\ref{prop:deformed_fermion_classical_shift} yields
\begin{align*}
\Sigma^{-m} \psi(z_i|\bal;\bbe) & = \frac{(z|\bbe)}{(z;\bal)} \Sigma^{1-m} \psi(z|\sigma_{\bal}\bal;\sigma_{\bbe}\bbe) \Sigma^{-1}
\\ & = \frac{(z|\bbe)}{(z;\bal)}\frac{(z|\sigma_{\bbe}\bbe)}{(z;\sigma_{\bal}\bal)} \Sigma^{2-m} \psi(z|\sigma_{\bal}^2\bal;\sigma_{\bbe}^2\bbe) \Sigma^{-2}
\\ & = \cdots = \prod_{j=0}^{m-1} \frac{(z|\sigma_{\bbe}^j\bbe)}{(z;\sigma_{\bal}^j\bal)} \psi(z|\sigma_{\bal}^m\bal;\sigma_{\bbe}^m\bbe) \Sigma^{-m}.
\end{align*}
Next, note that $\prod_{j=0}^{m-1} (z_i;\sigma_{\bal}^j \bal) = (z_i;\bal)^m$ and $\prod_{j=0}^{m-1} (z_i|\sigma_{\bbe}^j \bbe) = (z_i|\bbe)^m$.
To finish the proof, we again use for the final step the application of $\sigma_{\bal}^m \sigma_{\bbe}^m$ to the $m = 0$ case.
\end{proof}

Next, we expand the product formula from Lemma~\ref{lemma:fermion_expectation} in terms of the appropriate shifted powers in order to get the image of $\ket{\lambda}_m$.
We will focus on the case $m = 0$ to recover known formulas, but the other cases can be computed similarly.

\begin{definition}[Double factorial Schur functions]
\label{def:dfSchur}
For any $\lambda \in \ZZ^n$, we define the \defn{double factorial Schur functions} by
\[
s_{\lambda}(\pp\dv\bal;\bbe) := \bra{\varnothing} e^{H_+(\pp|\bal;\bbe)} \ket{\lambda} = \Phi^{(\bal;\bbe)}(\ket{\lambda}).
\]
Here $\lambda$ is not necessarily a partition, but we extend the notation $\ket{\lambda}_m$ to elements of $\ZZ^n$ also using~\eqref{eq:ket-from-vacuum}.
However, this is sometimes zero; \textit{e.g.}, when $\lambda_i = \lambda_{i+1} + 1$.

For any $\lambda,\mu \in \ZZ^n$, we define the \defn{skew double factorial Schur function} by
\[
s_{\lambda/\mu}(\pp\dv\bal;\bbe) := \bra{\mu} e^{H_+(\pp|\bal;\bbe)} \ket{\lambda}.
\]
\end{definition}

By using the definition of $\ket{\lambda}$ in terms of the creation operators from $\ket{\varnothing}_{-n}$, the corresponding series over all $\lambda$ becomes
\begin{equation}
\label{eq:MMgen_recovery}
\begin{aligned}
\sum_{\lambda \in \ZZ^n} s_{\lambda}(\pp\dv\bal;\bbe) & \frac{z_1 (z_1|\bbe)^{\lambda_1-1} z_2 (z_2|\bbe)^{\lambda_2-2} \cdots z_n (z_n|\bbe)^{\lambda_n-n}}{(z_1;\bal)^{\lambda_1} (z_2;\bal)^{\lambda_2-1} \cdots (z_n;\bal)^{\lambda_n-n+1}}
\\ & =  \prod_{i=1}^n z_i e^{\xi(\pp;z_i)} (z_i|\bbe)^{-n} \prod_{i < j} (z_i - z_j) (1 - \alpha_{1-i}\beta_{1-j}) e^{\Lambda_{-n}(\pp|\bbe)}
\end{aligned}
\end{equation}
by Lemma~\ref{lemma:fermion_expectation}.

\begin{corollary}
\label{cor:MM_schur}
The double factorial Schur functions $s_{\lambda}(\pp\dv\bal;\bbe)$ are the so-called generalized factorial Schur functions from Miyaura--Mukaihira~\cite{MiyauraMukaihiraGeneralized}.
\end{corollary}

\begin{proof}
After multiplying both sides of~\eqref{eq:MMgen_recovery} by
$
\prod_{j=1}^n z_i^{-1} \frac{(z_j;\bal)^{1-j}}{(z_j|\bbe)^{-j}} 
$
and using~\eqref{eq:shifting_powers}, we obtain~\cite[Eq.~(2.19)]{MiyauraMukaihiraGeneralized} with the parameters in~\cite{MiyauraMukaihiraGeneralized} as $a_i = \beta_i$ and $b_i = \alpha_i$.
\end{proof}

We also note some basic properties of the double factorial Schur functions.

\begin{proposition}[Basis]
\label{prop:basis}
The set of double factorial Schur functions $\{s_{\lambda}(\pp\dv\bal;\bbe) \mid \lambda \in \mcP\}$ is a basis for $\CC\FPS{\pp}$, which is the completion of $\bosonfock_0 = \CC[\pp]$ by degree.
\end{proposition}

\begin{proof}
This is immediate from the fact that the $\ket{\lambda}$ form a basis of $\fermionfock_0$ and Theorem~\ref{thm:deformed_boson_fermion}.
\end{proof}

\begin{proposition}[Containment property]
\label{prop:containment}
If $\mu \not\subseteq \lambda$, then $s_{\lambda/\mu}(\pp\dv\bal;\bbe) = 0$.
\end{proposition}

\begin{proof}
Follows from $A_{ij}^k = 0$ whenever $i > j$.
\end{proof}

\begin{proposition}[Branching rule]
\label{prop:branching}
We have
\[
s_{\lambda/\mu}(\pp+\pp'\dv\bal;\bbe) = \sum_{\mu \subseteq \nu \subseteq \lambda} s_{\lambda/\nu}(\pp\dv\bal;\bbe) s_{\nu/\mu}(\pp'\dv\bal;\bbe).
\]
\end{proposition}

\begin{proof}
This is standard using $e^{H_+(\pp+\pp'|\bal;\bbe)} = e^{H_+(\pp|\bal;\bbe)+H_+(\pp'|\bal;\bbe)} = e^{H_+(\pp|\bal;\bbe)} e^{H_+(\pp'|\bal;\bbe)}$.
\end{proof}

We also compute dual versions of Lemma~\ref{lemma:fermion_expectation} and~\eqref{eq:MMgen_recovery} that will be suitable for describing the double factorial Schur functions in terms of conjugate shapes.

\begin{lemma}
\label{lemma:dual_fermion_expectation}
We have
\begin{align*}
{}_{\ell} \bra{\varnothing} e^{H_+(\pp|\bal;\bbe)} & \psi^*(w_1|\bal;\bbe) \cdots \psi^*(w_n|\bal;\bbe) \ket{\varnothing}_{m+n}
\\ & \hspace{-30pt} = \delta_{m\ell} \prod_{i=1}^n e^{-\xi(\pp;w_i)} \frac{(w_i;\bal)^m}{(w_i|\bbe)^{m+n}} (1 - \alpha_{i+m} \beta_{i+m}) \prod_{i < j} (w_j - w_i) (1 - \alpha_{i+m}\beta_{j+m}) e^{\Lambda_{m+n}(\pp|\bbe)}.
\end{align*}
\end{lemma}

\begin{proof}
The proof is analogous to the proof of Lemma~\ref{lemma:fermion_expectation}.
Indeed, we show the $\ell = 0$ case by computing
\begin{align*}
\bra{\varnothing} e^{H_+(\pp|\bal;\bbe)} & \psi^*(w_1|\bal;\bbe) \cdots \psi^*(w_n|\bal;\bbe) \ket{\varnothing}_n
\\ & = \prod_{i=1}^n e^{-\xi(\pp;w_i)} e^{\Lambda_n(\pp|\bbe)} \bra{\varnothing} \psi^*(w_1|\bal;\bbe) \cdots \psi^*(w_n|\bal;\bbe) \ket{\varnothing}_n
\\ & = e^{\Lambda_n(\pp|\bbe)} \prod_{i=1}^n e^{-\xi(\pp;w_i)} \sum_{\omega \in S_n} (-1)^{\omega} \prod_{i=1}^n (1 - \alpha_{\omega_i} \beta_{\omega_i}) \frac{(w_1;\bal)^{\omega_1-1} \cdots (w_n;\bal)^{\omega_n-1}}{(w_1|\bbe)^{\omega_1} \cdots (w_n|\bbe)^{\omega_n}}
\\ & = e^{\Lambda_n(\pp|\bbe)} \prod_{i=1}^n \frac{e^{-\xi(\pp;w_i)}}{(w_i|\bbe)^n} (1 - \alpha_i \beta_i) \prod_{i < j} (w_i - w_j)(1-\alpha_i\beta_j),
\end{align*}
using the same arguments as Lemma~\ref{lemma:fermion_expectation} except now the determinant is~\cite[Eq.~(2.22)]{MiyauraMukaihiraGeneralized} or is precisely~\cite[Prop.~3.15]{NaprienkoFFS}.
The general case then follows analogously.
\end{proof}

The particle-hole duality of a partition $\lambda$ means that for the $01$ sequence, its conjugate $\lambda'$ is given by taking the reverse sequence of $\lambda$ and replacing $0 \leftrightarrow 1$.
Thus, applying the particle-hole duality of $\lambda$, we obtain the expression
\begin{equation}
\label{eq:conjugate_fermion_vacuum}
(-1)^{\abs{\lambda}} s_{\lambda'}(\pp\dv\bal;\bbe) = \bra{\varnothing} e^{H_+(\pp|\bal;\bbe)} \psi_{1-\lambda_1}^* \psi_{2-\lambda_2}^* \cdots \psi_{\ell-\lambda_{\ell}}^* \ket{\varnothing}_{\ell}
\end{equation}

\begin{example}
\label{ex:partition_particle_hole}
Consider the partition $\lambda = (6,5,2,2,1,1,1,1)$.
Then $\lambda' = (8,4,2,2,2,1)$ and
\begin{gather*}
\begin{aligned}
\ket{\lambda} & = \psi_6 \psi_4 \psi_0 \psi_{-1} \psi_{-3} \psi_{-4} \psi_{-5} \psi_{-6} \ket{\varnothing}_{-8} = -\psi_{-7}^*  \psi_{-2}^* \psi_1^* \psi_2^* \psi_3^* \psi_5^* \ket{\varnothing}_6
\\ & = v_6 \wedge v_4 \wedge v_0 \wedge v_{-1} \wedge v_{-3} \wedge v_{-4} \wedge v_{-5} \wedge v_{-6} \wedge v_{-8} \wedge v_{-9} \wedge \cdots
\end{aligned}
\\
\begin{tikzpicture}[scale=.8]
  \foreach \i in {-9,-8,-6,-5,-4,-3,-1,0,4,6} {
    \draw[fill=black] (-\i,0) circle (.25);
  }
  \foreach \i in {-7,-2,1,2,3,5} {
    \draw[fill=black!20] (-\i,0) circle (.25);
  }
  \foreach \i in {7,8} {
    \draw[fill=white] (-\i,0) circle (.25);
  }
  \foreach \i in {-9,...,8}
    \node at (-\i,-0.6) {\tiny $\i$};
  \node at (10,0) {$\cdots$};
  \node at (-9,0) {$\cdots$};
\end{tikzpicture}
\end{gather*}
with the particles being the solid black circles and the light gray nodes corresponding to the holes created by the $\psi_i^*$ operators.
\end{example}

Therefore, Lemma~\ref{lemma:dual_fermion_expectation} with Equation~\eqref{eq:conjugate_fermion_vacuum} yields the generating series
\begin{equation}
\label{eq:dual_fermion_schur_expansion}
\begin{aligned}
\sum_{\lambda \in \ZZ^n} (-1)^{\abs{\lambda}} s_{\lambda'}(\pp\dv\bal;\bbe) & \prod_{i=1}^n (1-\alpha_{i-\lambda_i} \beta_{i-\lambda_i}) \frac{(w_i;\bal)^{i-\lambda_i-1}}{(w_i|\bbe)^{i-\lambda_i}}
\\ & = \prod_{i=1}^n \frac{e^{-\xi(\pp;w_i)}}{(w_i|\bbe)^n} (1 - \alpha_i \beta_i) \prod_{i < j} (w_i - w_j) (1 - \alpha_i \beta_j) e^{\Lambda_n(\pp|\bbe)}.
\end{aligned}
\end{equation}


One additional way to realize $\ket{\lambda}$ is using the Frobenius coordinates of a partition $\lambda = (a_1, \dotsc, a_m | b_1, \dotsc, b_m)$ (see, \textit{e.g.},~\cite{Andrews84} with the original source~\cite{Frobenius00}).
A well-known interpretation of $\ket{\lambda}$ using Frobenius coordinates means we can rewrite
\begin{equation}
\label{eq:frobenius_fermion_vacuum}
(-1)^{m+b_1+\cdots+b_m} s_{\lambda}(\pp\dv\bal;\bbe) = \bra{\varnothing} e^{H_+(\pp|\bal;\bbe)} \psi_{-b_1}^* \cdots \psi_{-b_m}^* \psi_{a_m+1} \cdots \psi_{a_1+1} \ket{\varnothing}.
\end{equation}

\begin{example}
Continuing from Example~\ref{ex:partition_particle_hole} with $\lambda = (6,5,2,2,1,1,1,1)$, the Frobenius coordinates of $\lambda$ are $(5, 3 \mid 7, 2)$ and $\ket{\lambda} = -\psi_{-7}^* \psi_{-2}^* \psi_4 \psi_6 \ket{\varnothing}$.
\end{example}

Now we want to incorporate~\eqref{eq:frobenius_fermion_vacuum} into a generating series formula like ~\eqref{eq:MMgen_recovery} and~\eqref{eq:dual_fermion_schur_expansion}, but it will more naturally involve a determinant from Wick's theorem.
Let us first consider the case $m \leqslant 1$ by doing a computation analogous to Lemma~\ref{lemma:fermion_expectation}:
\begin{equation}
\label{eq:hook_naive}
\frac{z}{z-w} + \sum_{a, b=0}^{\infty} (-1)^b s_{(a|b)}(\pp\dv\bal;\bbe) (1 - \alpha_{-b} \beta_{-b}) \frac{z (z|\bbe)^{a} (w;\bal)^{-b-1}}{(z;\bal)^{a+1} (w|\bbe)^{-b}}
=  \frac{z}{z-w} e^{\xi(\pp;z)} e^{-\xi(\pp;w)}.
\end{equation}
Note that we have used Proposition~\ref{prop:vacuum_deformed_fields} in obtaining the factors $\frac{z}{z-w}$ on both sides.
Therefore, we can rewrite~\eqref{eq:hook_naive} in the following
\begin{equation}
\label{eq:hook_generating}
1 + (z - w) \sum_{a, b=0}^{\infty} (-1)^b s_{(a|b)}(\pp\dv\bal;\bbe) (1 - \alpha_{-b} \beta_{-b}) \frac{(z|\bbe)^{a} (w;\bal)^{-b-1}}{(z;\bal)^{a+1} (w|\bbe)^{-b}}
=  e^{\xi(\pp;z)} e^{-\xi(\pp;w)}.
\end{equation}

For general $m$, then we will end up with a determinant by Wick's theorem, and we leave the precise generating series formula as an exercise for interested reader.
However, if we apply the multiple contour integrals to extract the $s_{\lambda}(\pp\dv\bal;\bbe)$ factor (Proposition~\ref{prop:orthonormality}), we obtain the following.

\begin{theorem}[Double factorial Giambelli formula]
\label{thm:giambelli}
Write $\lambda = (a_1, \dotsc, a_m|b_1, \dotsc, b_m)$ in Frobenius coordinates.
We have
\[
s_{\lambda}(\pp\dv\bal;\bbe) = \det \bigl[ s_{(a_i|b_j)}(\pp\dv\bal;\bbe) \bigr]_{i,j=1}^m.
\]
\end{theorem}

\begin{proof}
This follows from writing~\eqref{eq:frobenius_fermion_vacuum} as a multiple contour integral, applying Wick's theorem, using the multilinearlity of the determinant with the contour integrals, and finally~\eqref{eq:hook_generating}.
\end{proof}

Next, we will construct the (dual) Jacobi--Trudi identity for general skew shapes.
We define functions
\[
h_k(\pp\dv\bal;\bbe) := s_k(\pp\dv\bal;\bbe),
\qquad\qquad
e_k(\pp\dv\bal;\bbe) := s_{1^k}(\pp\dv\bal;\bbe),
\]
as the single row and column, respectively, versions of $s_{\lambda}(\pp\dv\bal;\bbe)$.
We refer to these as the \defn{homogeneous double factorial symmetric functions} and \defn{elementary double factorial symmetric functions}, respectively, as they are the analogs of the homogeneous and elementary symmetric functions.
Moreover, they have generating series
\begin{subequations}
\label{eq:eh_gen_series}
\begin{align}
\label{eq:h_gen_series}
e^{\xi(\pp;z)} e^{-\xi(\pp;\beta_0)} & = \sum_{k=0}^{\infty} h_k(\pp\dv\bal;\bbe) \frac{(z|\sigma_{\bbe}^{-1} \bbe)^k}{(z;\bal)^k},
\\
\label{eq:e_gen_series}
e^{-\xi(\pp;w)} e^{\xi(\pp;\beta_1)} & = \sum_{k=0}^{\infty} (-1)^k e_k(\pp\dv\bal;\bbe) \frac{1 - \alpha_{1-k}\beta_{1-k}}{1 - \alpha_1 \beta_1} \frac{(w;\bal)^{-k}}{(w|\sigma_{\bbe}^{-1}\bbe)^{-k}},
\end{align}
\end{subequations}
where~\eqref{eq:h_gen_series} (resp.~\eqref{eq:e_gen_series}) series is~\eqref{eq:MMgen_recovery} (resp.~\eqref{eq:dual_fermion_schur_expansion}) with $n = 1$ after multiplying both sides by $(z - \beta_0)$ (resp.\ $\frac{w - \beta_1}{1-\alpha_1\beta_1}$).
We can also obtain~\eqref{eq:h_gen_series} (resp.~\eqref{eq:e_gen_series}) from~\eqref{eq:hook_generating} by substituting $w = \beta_0$ (resp.\ $z = \beta_1$).
We also remark that~\eqref{eq:h_gen_series} is~\cite[Eq.~(2.18)]{MiyauraMukaihiraGeneralized}, and~\eqref{eq:eh_gen_series} yields~\cite[Eq.~(17),~(18)]{MiyauraMukaihiraFactorial} under specialization of $p_k = p_k(\xx/{-\iota_{\bal}\bal})$ by~\eqref{eq:exp_xi_specialization}.

\begin{theorem}[Double factorial Jacobi--Trudi identities]
\label{thm:jacobi_trudi}
We have
\begin{align*}
s_{\lambda/\mu}(\pp\dv\bal;\bbe) & = e^{\Lambda_{-n}(\pp|\bbe)} \det \Bigl[ e^{\xi(\pp;\beta_{\mu_j-j+1})} h_{\lambda_i-\mu_j-i+j}(\pp\dv\sigma_{\bal}^{\mu_j-j+1}\bal;\sigma_{\bbe}^{\mu_j-j+1}\bbe) \Bigr]_{i,j=1}^n
\\ & = e^{\Lambda_n(\pp|\bbe)} \det \Bigl[ e^{-\xi(\pp;\beta_{j-\mu'_j})} e_{\lambda'_i-\mu_j'-i+j}(\pp\dv\sigma_{\bal}^{j-\mu'_j-1}\bal;\sigma_{\bbe}^{j-\mu'_j-1}\bbe) \Bigr]_{i,j=1}^n.
\end{align*}
\end{theorem}

We refer to the determinant using $h_k(\pp\dv\bal;\bbe)$ (resp.\ $e_k(\pp\dv\bal;\bbe)$) as the \defn{Jacobi--Trudi identity} (resp.\ \defn{dual Jacobi--Trudi identity}, also known as the \defn{N\"agelsbach--Kostka formula}).

\begin{proof}
%
We first compute
\begin{align*}
e^{-\Lambda_{-n}(\pp|\bbe)} {}_{-n} \bra{\varnothing} \psi_j^* e^{H_+(\pp|\bal;\bbe)} & \psi(z|\bal;\bbe) \ket{\varnothing}_{-n}
\\ & = e^{\xi(\pp;z)} e^{\Lambda_{-n}(\pp|\bbe)} \cdot {}_{-n} \bra{\varnothing} \psi_j^* \psi(z|\bal;\bbe) \ket{\varnothing}_{-n}
\\ & = e^{\xi(\pp;\beta_m)} \sum_{k=0}^{\infty} h_k(\pp\dv\sigma_{\bal}^m\bal;\sigma_{\bbe}^m\bbe) \frac{(z|\sigma_{\bbe}^{m-1} \bbe)^k}{(z;\sigma_{\bal}^m\bal)^k} \frac{z(z|\bbe)^{j-1}}{(z;\bal)^j},
\end{align*}
where we have applied the map $\sigma_{\bal}^m \sigma_{\bbe}^m$ to both sides of the definition of $h_k(\pp\dv\bal;\bbe)$ and ${}_{-n} \bra{\varnothing} e^{H_+(\pp|\bal;\bbe)} \ket{\varnothing}_{-n} = e^{\Lambda_{-n}(\pp|\bbe)}$.
If we set $m = j$, then $(z;\sigma_{\bal}^j\bal)^k (z;\bal)^j = (z;\bal)^{k+j}$ by~\eqref{eq:shifting_powers}.
Thus, we can easily take the coefficient of $\frac{z(z|\bbe)^{i-1}}{(z;\bal)^i}$ of both sides to obtain
\begin{equation}
\label{eq:JT_matrix_entry}
{}_{-n} \bra{\varnothing} \psi_j^* e^{H_+(\pp|\bal;\bbe)} \psi_i \ket{\varnothing}_{-n} = e^{\xi(\pp;\beta_j)} e^{\Lambda_{-n}(\pp|\bbe)} h_{i-j}(\pp\dv\sigma_{\bal}^j\bal;\sigma_{\bbe}^j\bbe),
\end{equation}
Next, Wick's theorem (see, \textit{e.g.},~\cite[Eq.~(2.47)]{AlexandrovZabrodin}; note that $e^{H_+(z|\bal;\bbe)}$ is group-like in the context of~\cite{AlexandrovZabrodin}) and Equation~\eqref{eq:JT_matrix_entry} imply
\begin{align*}
e^{-\Lambda_{-n}(\pp|\bbe)} \cdot \bra{\mu} e^{H_+(\pp|\bal;\bbe)} \ket{\lambda} & = \det [ e^{-\Lambda_{-n}(\pp|\bbe)} \cdot {}_{-n} \bra{\varnothing} \psi_{\mu_j-j+1}^* e^{H_+(\pp|\bal;\bbe)} \psi_{\lambda_i-i+1} \ket{\varnothing}_{-n} ]_{i,j=1}^n
\\ & = \det [ e^{\xi(\pp;\beta_{\mu_j-j+1})} h_{\lambda_i-\mu_j-i+j}(\pp\dv\sigma_{\bal}^{\mu_j-j+1}\bal;\sigma_{\bbe}^{\mu_j-j+1}\bbe) ]_{i,j=1}^n,
\end{align*}
which shows the first equality.

The second equality (the dual Jacobi--Trudi formula) is proven analogously using~\eqref{eq:conjugate_fermion_vacuum}.
\end{proof}

Note that this recovers~\cite[Eq.~(2.17)]{MiyauraMukaihiraGeneralized} (equivalently~\cite[Eq.~(19)]{MiyauraMukaihiraFactorial}) when $\mu = \varnothing$ by factoring out the $e^{\xi(\pp;\beta_{1-j})}$ from each column and noting $\Lambda_{-n}(\pp|\bbe) = -\sum_{j=1}^n \xi(\pp;\beta_{1-j})$.
However, our skew double factorial Schur functions are different than the so-called skew factorial-type Schur functions in~\cite[Eq.~(24), (25)]{MiyauraMukaihiraFactorial}.\footnote{It is not clear to the authors that~\cite[Eq.~(24)]{MiyauraMukaihiraFactorial} is consistent with~\cite[Eq.~(25)]{MiyauraMukaihiraFactorial}; the latter is claimed without a proof.}

Next, we define an algebra involution of $\bosonfock$ by $\omega p_k = (-1)^{k+1} p_k$.

\begin{corollary}
\label{cor:involution}
Write $\lambda = (a_1, \dotsc, a_m|b_1, \dotsc, b_m)$ in Frobenius coordinates.
Then we have
\begin{equation}
\label{eq:omega_dfSchur}
\omega s_{\lambda}(\pp\dv\bal;\bbe) = \prod_{i=1}^m \frac{1 - \alpha_{1+a_i}\beta_{1+a_i}}{1 - \alpha_{-b_i}\beta_{-b_i}} s_{\lambda'}(\pp\dv{-\iota_{\bal}\bal};{-\iota_{\bbe}\bbe}).
\end{equation}
\end{corollary}

\begin{proof}
We note that
\begin{equation}
\label{eq;omega_xi}
\omega \xi(\pp; z) = -\xi(\pp; -z),
\qquad\qquad
\Lambda_{-n}(\omega \pp|\bbe) = -\Lambda_{-n}(\pp|{-\bbe}) = \Lambda_n(\pp|{-\iota_{\bbe}\bbe}).
\end{equation}
Therefore, applying $\omega$ to both sides of~\eqref{eq:MMgen_recovery} and $z_i \mapsto -w_i$ yields
\begin{equation}
\label{eq:omega_MMgen}
\begin{aligned}
\sum_{\lambda \in \ZZ^n} (-1)^{\abs{\lambda}} \bigl(\omega s_{\lambda}(\pp\dv\bal;\bbe)\bigr) & \frac{(w_1;{-\iota_{\bal}\bal})^{-\lambda_1} \cdots (w_n;{-\iota_{\bal}\bal})^{n-1-\lambda_n}}{(w_1|{-\iota_{\bbe}\bbe})^{1-\lambda_1} \cdots (w_n|{-\iota_{\bbe}\bbe})^{n-\lambda_n}}
\\ & =  \prod_{i=1}^n \frac{e^{-\xi(\pp;w_i)}}{(w_i|{-\iota_{\bbe}\bbe})^n} \prod_{i < j} (w_i - w_j) (1 - \alpha_{1-i}\beta_{1-j}) e^{\Lambda_n(\pp|{-\iota_{\bbe}\bbe})}.
\end{aligned}
\end{equation}
Comparing~\eqref{eq:omega_MMgen} with~\eqref{eq:dual_fermion_schur_expansion} after (applying $(-\iota_{\bal}) (-\iota_{\bbe})$ to both sides), we obtain the reuslt.
%
\end{proof}

A special case of Corollary~\ref{cor:involution} is
\begin{equation}
\label{eq:omega_eh}
\omega h_k(\pp\dv\bal;\bbe) = \frac{1-\alpha_k\beta_k}{1-\alpha_0\beta_0} e_k(\pp\dv{-\iota_{\bal}\bal};{-\iota_{\bbe}\bbe}),
\end{equation}
which we can also see from the generating series~\eqref{eq:eh_gen_series}.
We also note that the factor in~\eqref{eq:omega_dfSchur} is the recipricol of the one that appears in the right hand side of~\cite[Eq.~(22)]{MiyauraMukaihiraFactorial}.
Using this and the Jacobi--Trudi formula, we can give the action $\omega$ for the skew case.

\begin{corollary}
\label{cor:involution_skew}
We have
\begin{equation}
\label{eq:omega_skew_dfSchur}
\omega s_{\lambda/\mu}(\pp\dv\bal;\bbe) = \prod_{i=1}^m \frac{1-\alpha_{\lambda_i-i+1} \beta_{\lambda_i-i+1}}{1-\alpha_{\mu_i-i+1}\beta_{\mu_i-i+1}} s_{\lambda'/\mu'}(\pp\dv{-\iota_{\bal}\bal};{-\iota_{\bbe}\bbe}).
\end{equation}
\end{corollary}


\begin{proof}
By Theorem~\ref{thm:jacobi_trudi} and Corollary~\ref{cor:involution}, we have
\begin{align*}
\omega s_{\lambda/\mu}(\pp\dv\bal;\bbe) & = e^{-\Lambda_{-n}(\pp|{-\bbe})} \det \Bigl[ e^{-\xi(\pp;-\beta_{\mu_j-j+1})} \omega h_{\lambda_i-\mu_j-i+j}(\pp\dv\sigma_{\bal}^{\mu_j-j+1}\bal;\sigma_{\bbe}^{\mu_j-j+1}\bbe) \Bigr]_{i,j=1}^n
\\ & = e^{\Lambda_n(\pp|{-\iota_{\bbe}\bbe})} \det \biggl[ e^{-\xi(\pp;-\iota_{\bbe} \beta_{j-\mu_j})} \frac{1-\alpha_{\lambda_i-i+1} \beta_{\lambda_i-i+1}}{1-\alpha_{\mu_j-j+1}\beta_{\mu_j-j+1}}
\\ & \hspace{100pt} \times e_{\lambda_i-\mu_j-i+j}(\pp\dv {-\sigma_{\bal}^{j-\mu_j-1} \iota_{\bal} \bal}; {-\sigma_{\bbe}^{j-\mu_j-1} \iota_{\bbe} \bbe}) \biggr]_{i,j=1}^n
\\ & = \prod_{i=1}^m \frac{1-\alpha_{\lambda_i-i+1} \beta_{\lambda_i-i+1}}{1-\alpha_{\mu_i-i+1}\beta_{\mu_i-i+1}} s_{\lambda'}(\pp\dv{-\iota_{\bal}\bal};{-\iota_{\bbe}\bbe}),
\end{align*}
where we used~\eqref{eq;omega_xi} and the multilinearity of the determinant for the final equality since the numerator (resp.\ denominator) of each factor only depends on $i$ (resp.~$j$).
\end{proof}

\begin{example}
We verify that Corollary~\ref{cor:involution} and Corollary~\ref{cor:involution_skew} agree for $\lambda = (4,1,1) = (3|2)$ and $\mu = \varnothing$, where we have the factors
\[
\frac{1 - \alpha_4\beta_4}{1 - \alpha_{-2}\beta_{-2}} = \frac{1 - \alpha_4\beta_4}{1 - \alpha_0 \beta_0} \frac{1 - \alpha_0\beta_0}{1 - \alpha_{-1} \beta_{-1}} \frac{1 - \alpha_{-1}\beta_{-1}}{1 - \alpha_{-2} \beta_{-2}}.
\]
\end{example}

Corollary~\ref{cor:involution_skew} also provides a new proof of~\cite[Eq.~(9.6)]{Macdonald92}
in the case of the double factorial Schur functions.
Pulling Corollary~\ref{cor:involution_skew} back across the boson-fermion correspondence yields the following.

\begin{corollary}
$
\displaystyle
\bra{\mu} e^{H_+(\omega \pp|\bal;\bbe)} \ket{\lambda} = \prod_{i=1}^m \frac{1-\alpha_{\lambda_i-i+1} \beta_{\lambda_i-i+1}}{1-\alpha_{\mu_i-i+1}\beta_{\mu_i-i+1}} s_{\lambda'/\mu'}(\pp\dv{-\iota_{\bal}\bal};{-\iota_{\bbe}\bbe}).
$
\end{corollary}

\begin{remark}
We could define a (twisted) involution $\omega_{-\iota}$ of $\bosonfock$ by $\omega_{-\iota} p_k = (-1)^{k-1} p_k$ with also $\omega_{-\iota} \alpha_i = -\alpha_{1-i} = -\iota \bal$ and $\omega_{-\iota} \beta_i = -\beta_{1-i} = -\iota \bbe$.
Then we would not have the additional twisting of the parameters, \textit{e.g.}, $\omega_{-\iota} h_k(\pp\dv\bal;\bbe) = \frac{1-\alpha_{1-k}\beta_{1-k}}{1-\alpha_0\beta_0} e_k(\pp\dv\bal;\bbe)$.
However, for this to be an (Ore) algebra morphism, we would require the $\bal$ and $\bbe$ to be variables and to work over, say, the fraction field $\CC(\bal;\bbe)$.
\end{remark}

We now will give an analog of the Murnaghan--Nakayama rule.
We start with defining some additional notation.
For any partition $\lambda$, let $\pp_{\lambda} = p_{\lambda_1} p_{\lambda_2} \cdots$, where we consider $p_0 = 1$, and $J_{\lambda}^{(\bal;\bbe)} = J_{\lambda_1}^{(\bal;\bbe)} J_{\lambda_2}^{(\bal;\bbe)} \cdots$.
Note that the factors of $J_{\lambda}^{(\bal;\bbe)}$ mutually commute by Theorem~\ref{thm:heisenberg_relations}.
We also need some additional definitions regarding the combinatorics of partitions.
Let $m_i^{(\lambda)}$ denote the number of parts of size $i$ in $\lambda$.
Given a box $\bbb$ in row $R$ and column $C$ of the Young diagram of a skew partition, the \defn{content} of $b$ is $c(\bbb) := C - R$.
The content is constant on diagonals of the Young diagram.
A skew partition $\lambda/\mu$, is called a \defn{ribbon} if it is connected and contains at most one box on each diagonal.
Let $\lambda / \mu$ be a ribbon, and the contents of the boxes of $\lambda/\mu$ form a sequence of consecutive integers $i,i+1,\dotsc,j-1$ for some integers $i$ and $j$, and we will call the sequence $[i,j)$ the \defn{content interval} of $\lambda/\mu$.
The \defn{height} of a ribbon $\lambda/\mu$, denoted $\htf(\lambda/\mu)$, is the number of distinct rows it occupies.

By directly expanding the product, we have
\begin{equation}
e^{H_+(z|\bal;\bbe)} 
 = \sum_{m=0}^{\infty} \frac{1}{m!} \sum_{k_1,\dotsc,k_m=1}^{\infty} \frac{p_{k_1}}{k_1} \cdots \frac{p_{k_m}}{k_m} J_{k_1}^{(\bal;\bbe)} \cdots J_{k_m}^{(\bal;\bbe)}
  = \sum_{\nu\in\mcP} z_{\nu}^{-1} \pp_{\nu} J_{\nu}^{(\bal;\bbe)},
\label{exponential-expansion}
\end{equation}
where $z_{\nu} := \prod_{i \geqslant 1} i^{m_i^{(\nu)}} m_i^{(\nu)}!$.
Thus,
\begin{equation}
s_{\lambda/\mu}(\pp\dv\bal;\bbe) = \sum_{\nu \in \mcP} z_{\nu}^{-1} \pp_{\nu} \bra{\mu} J_{\nu}^{(\bal;\bbe)} \ket{\lambda}. \label{eq:tau-expansion}
\end{equation}

\begin{lemma}[{\cite[Lemma~1]{Lam06}}] \label{Lam-lemma}
Let $k>0$ and $\nu$ be a partition. Then we have
\[
J_{-k}^{(\bal;\bbe)} J_{\nu}^{(\bal;\bbe)} = -k m_k^{(\nu)} J_{\mu}^{(\bal;\bbe)} + J_{\nu}^{(\bal;\bbe)} J_{-k}^{(\bal;\bbe)},
\] 
where $\mu$ be $\nu$ with one part of size $k$ removed.
\end{lemma}

The next lemma follows from the combinatorics of the fermionic Fock space, along with Proposition~\ref{prop:explicit_current} and~\eqref{eq:adef}.

\begin{lemma} \label{lem:ribbon}
Let $\lambda,\mu$ be distinct partitions.
If $\mu \not\leqslant \lambda$ or if $\lambda/\mu$ is not a ribbon, then $\bra{\mu} J_k^{(\bal;\bbe)} \ket{\lambda} = \bra{\lambda} J_{-k}^{(\bal;\bbe)} \ket{\mu} = 0$.
If $\lambda/\mu$ is a ribbon with content interval $[i,j)$, then
\[
\bra{\mu} J_k^{(\bal;\bbe)} \ket{\lambda} = (-1)^{\htf(\lambda/\mu)-1} A_{ij}^k \hspace{40pt} \text{and} \hspace{40pt} \bra{\lambda} J_{-k}^{(\bal;\bbe)} \ket{\mu} = (-1)^{\htf(\lambda/\mu)-1} A_{ji}^{-k}.
\]
\end{lemma}

\begin{remark}
\label{rem:finite_sums_current_beta_zero}
When $\bbe = 0$, then $\bra{\mu} J_k^{(\bal;0)} \ket{\lambda}\ne 0$ precisely when $\lambda/\mu$ is a ribbon of size at least $k$ and $\bra{\lambda} J_{-k}^{(\bal;0)} \ket{\mu} \ne 0$ precisely when $\lambda/\mu$ is a ribbon of size at most $k$.
In particular the action of $J_{-k}^{(\bal;0)}$ on $\ket{\lambda}$ also results in a finite sum.
A similar statement holds if instead $\bal = 0$, where we swap the direction of the inequalities.
\end{remark}

In particular, let us consider the case $\lambda = \mu$.
For $k > 0 $, to compute $\bra{\lambda} J_{-k}^{(\bal;\bbe)} \ket{\lambda}$, we need to consider the effect of the normal ordering.
Indeed, we have
\begin{equation} \label{eq:jkii}
\bra{\lambda} J_{-k}^{(\bal;\bbe)} \ket{\lambda} = \sum_{i \in \ZZ} A_{ii}^k \bra{\lambda} \normord{\psi_i \psi_i^*} \ket{\lambda} =  \sum_{i\in\ZZ} \beta_i^k \bra{\lambda} (\psi_i \psi_i^* - \delta_{i\leqslant 0}) \ket{\lambda} = p_k(\bbe_{S(\lambda)}/(- \bbe_{T(\lambda)}),
\end{equation}
where we use the supersymmetric powersum functions
\begin{gather*}
p_k(\bbe_A / (-\bbe_B)) := \sum_{i\in A} \beta_i^k - \sum_{i\in B} \beta_i^k
\qquad\qquad \text{for any } A,B \subseteq \ZZ,
\\
S(\lambda) = \{i>0 \mid \text{$v_i$ appears in $\ket{\lambda}$}\},
\hspace{20pt}
T(\lambda) = \{i \leqslant 0 \mid \text{$v_i$ does not appear in $\ket{\lambda}$}\}.
\end{gather*}
In terms of $\lambda$ itself, we have
\[
S(\lambda) = \{\lambda_i-i+1 \mid \lambda_i\geqslant i\}
\hspace{10pt} \text{and} \hspace{10pt}
T(\lambda) = \{i-\lambda'_i \mid \lambda_i' \geqslant i\},
\]
where $\lambda'$ is the conjugate partition of $\lambda$ and $\lambda'_i$ is its $i$th part.
Similarly, we have
\begin{equation} \label{eq:jnegkii}
\bra{\lambda} J_{-k}^{(\bal;\bbe)} \ket{\lambda} = p_k(\bal_{S(\lambda)} / (-\bal_{T(\lambda)})),
\end{equation}

\begin{example} \label{ex:jkexample}
Let $\lambda = (3,2,1)$, and $k > 0$.
Then we compute $J_k^{(\bal;\bbe)} \ket{\lambda}$ by having the following motions of the particles
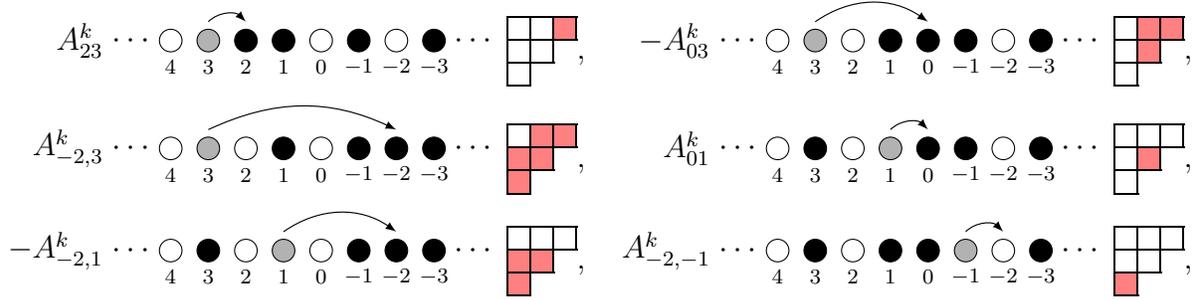
\ytableausetup{boxsize=.7em}%
\begin{align*}
\begin{tikzpicture}[scale=.5,baseline=-7]
  \node[anchor=east] at (-5.5,0) {$A_{23}^k$};
  \draw[fill=black!30] (-3,0) circle (.3);
  \draw[->, >=latex, line width=0.15mm] (-3,0.5) to [out=45,in=135] (-2,0.5);
  \foreach \i in {-3,-1,1,2}
    \draw[fill=black] (-\i,0) circle (.3);
  \foreach \i in {-2,0,4}
    \draw[fill=white] (-\i,0) circle (.3);
  \foreach \i in {-3,...,4}
    \node at (-\i,-0.7) {\tiny $\i$};
  \node at (-5,0) {$\cdots$};
  \node at (4.1,0) {$\cdots$};
\end{tikzpicture}
&
\ydiagram{3,2,1}*[*(red!50)]{2+1},
&
\begin{tikzpicture}[scale=.5,baseline=-7]
  \node[anchor=east] at (-5.5,0) {$-A_{03}^k$};
  \draw[fill=black!30] (-3,0) circle (.3);
  \draw[->, >=latex, line width=0.15mm] (-3,0.5) to [out=35,in=145] (-0,0.5);
  \foreach \i in {-3,-1,1,0}
    \draw[fill=black] (-\i,0) circle (.3);
  \foreach \i in {-2,2,4}
    \draw[fill=white] (-\i,0) circle (.3);
  \foreach \i in {-3,...,4}
    \node at (-\i,-0.7) {\tiny $\i$};
  \node at (-5,0) {$\cdots$};
  \node at (4.1,0) {$\cdots$};
\end{tikzpicture}
&
\ydiagram{3,2,1}*[*(red!50)]{1+2,1+1},
\allowdisplaybreaks\\
\begin{tikzpicture}[scale=.5,baseline=-7]
  \node[anchor=east] at (-5.5,0) {$A_{-2,3}^k$};
  \draw[fill=black!30] (-3,0) circle (.3);
  \draw[->, >=latex, line width=0.15mm] (-3,0.5) to [out=25,in=155] (2,0.5);
  \foreach \i in {-3,-1,-2,1}
    \draw[fill=black] (-\i,0) circle (.3);
  \foreach \i in {0,2,4}
    \draw[fill=white] (-\i,0) circle (.3);
  \foreach \i in {-3,...,4}
    \node at (-\i,-0.7) {\tiny $\i$};
  \node at (-5,0) {$\cdots$};
  \node at (4.1,0) {$\cdots$};
\end{tikzpicture}
&
\ydiagram{3,2,1}*[*(red!50)]{1+2,2,1},
&
\begin{tikzpicture}[scale=.5,baseline=-7]
  \node[anchor=east] at (-5.5,0) {$A_{01}^k$};
  \draw[fill=black!30] (-1,0) circle (.3);
  \draw[->, >=latex, line width=0.15mm] (-1,0.5) to [out=45,in=135] (-0,0.5);
  \foreach \i in {-3,-1,3,0}
    \draw[fill=black] (-\i,0) circle (.3);
  \foreach \i in {-2,2,4}
    \draw[fill=white] (-\i,0) circle (.3);
  \foreach \i in {-3,...,4}
    \node at (-\i,-0.7) {\tiny $\i$};
  \node at (-5,0) {$\cdots$};
  \node at (4.1,0) {$\cdots$};
\end{tikzpicture}
&
\ydiagram{3,2,1}*[*(red!50)]{0,1+1},
\allowdisplaybreaks\\
\begin{tikzpicture}[scale=.5,baseline=-7]
  \node[anchor=east] at (-5.5,0) {$-A_{-2,1}^k$};
  \draw[fill=black!30] (-1,0) circle (.3);
  \draw[->, >=latex, line width=0.15mm] (-1,0.5) to [out=35,in=145] (2,0.5);
  \foreach \i in {-3,-1,3,-2}
    \draw[fill=black] (-\i,0) circle (.3);
  \foreach \i in {0,2,4}
    \draw[fill=white] (-\i,0) circle (.3);
  \foreach \i in {-3,...,4}
    \node at (-\i,-0.7) {\tiny $\i$};
  \node at (-5,0) {$\cdots$};
  \node at (4.1,0) {$\cdots$};
\end{tikzpicture}
&
\ydiagram{3,2,1}*[*(red!50)]{0,2,1},
&
\begin{tikzpicture}[scale=.5,baseline=-7]
  \node[anchor=east] at (-5.5,0) {$A_{-2,-1}^k$};
  \draw[fill=black!30] (1,0) circle (.3);
  \draw[->, >=latex, line width=0.15mm] (1,0.5) to [out=45,in=135] (2,0.5);
  \foreach \i in {-3,1,3,0}
    \draw[fill=black] (-\i,0) circle (.3);
  \foreach \i in {-2,2,4}
    \draw[fill=white] (-\i,0) circle (.3);
  \foreach \i in {-3,...,4}
    \node at (-\i,-0.7) {\tiny $\i$};
  \node at (-5,0) {$\cdots$};
  \node at (4.1,0) {$\cdots$};
\end{tikzpicture}
&
\ydiagram{3,2,1}*[*(red!50)]{0,0,1},
\end{align*}
where the black (resp.\ gray, white) nodes correspond to particles (resp.\ original position of the moved particle, no particle), and the $A_{ii}^k$ contributing $(\beta_3^k + \beta_1^k - \beta_0^k - \beta_{-2}^k) \ket{\lambda}$.
Hence
\begin{align*}
J_k^{(\bal;\bbe)} \ket{3,2,1} & = (\beta_3^k + \beta_1^k - \beta_0^k - \beta_{-2}^k) \ket{3,2,1} + A_{-2,-1}^k \ket{3,2} + A_{01}^k \ket{3,1,1} + A_{23}^k \ket{2,2,1}
\\ & \hspace{20pt} - A_{03}^k \ket{1,1,1} - A_{-2,1}^k \ket{3} + A_{-2,3}^k \ket{1}.
\end{align*}

Next, we consider $J_{-k}^{(\bal;\bbe)} \ket{\lambda}$, where the resulting sum is (for generic $\bbe$) an infinite sum.
First, the coefficient of $\ket{\lambda}$ will be $\alpha_3^k + \alpha_1^k - \alpha_0^k - \alpha_{-2}^k$.
The motion of the first three particles yield
\begin{gather*}
\sum_{i=0}^{\infty} A_{4+i,3}^{-k} \ket{4+i,2,1},
\qquad\qquad
A_{21}^{-k} \ket{3,3,1} - \sum_{i=0}^{\infty} A_{4+i,1}^{-k} \ket{4+i,4,1},
\\
A_{0,-1}^{-k} \ket{3,2,2} - A_{2,-1}^{-k} \ket{3,3,3} + \sum_{i=0}^{\infty} A_{4+i,-1}^{-k} \ket{4+i,4,3},
\end{gather*}
respectively.
Finally, the motion of the particle starting at position $-j$, for $j \geqslant 3$ contributes
\begin{align*}
& (-1)^{j-3} A_{-2,-j}^{-k} \ket{3,2,1,1,1^{j-3}} - (-1)^{j-3} A_{0,-j}^{-k} \ket{3,2,2,2,1^{j-3}} + (-1)^{j-3} A_{2,-j}^{-k} \ket{3,3,3,2,1^{j-3}}
\\ & \hspace{50pt} - \sum_{i=0}^{\infty} (-1)^{j-3} A_{4+i,-j}^{-k} \ket{4+i,4,3,2,1^{j-3}}.
\end{align*}
Note that while this is an infinite sum, each $\ket{\lambda}$ appears at most once, and hence the coefficient of $\ket{\lambda}$ is well-defined.
\end{example}

\begin{theorem}[Double factorial Murnaghan--Nakayama rule] \label{thm:factorialmn}
Let $k>0$.
\begin{subequations}
\begin{equation}
p_k s_\lambda(\pp \dv \bal;\bbe) = p_k\bigl(\bal_{S(\lambda)}/(-\bal_{T(\lambda)})\bigr) s_{\lambda}(\pp \dv \bal;\bbe) + \sum_{\nu \in \mcP} (-1)^{\htf(\nu/\lambda)-1} A_{j,i}^{-k} s_\nu(\pp \dv \bal;\bbe),
\end{equation}
where the sum is over all $\nu$ such that $\nu/\lambda$ is a nonzero ribbon and $[i,j)$ is the content interval of $\nu/\lambda$.
\begin{equation}
k\frac{\partial}{\partial p_k} s_{\lambda}(\pp\dv\bal;\bbe) = p_k\bigl(\bbe_{S(\lambda)}/(-\bbe_{T(\lambda)})\bigr) s_{\lambda}(\pp\dv\bal;\bbe) + \frac{1}{k} \sum_{\mu \in \mcP} (-1)^{\htf(\lambda/\mu)-1} A_{ij}^k  s_\mu(\pp \dv \bal;\bbe),
\end{equation}
where the sum is over all $\mu$ such that $\lambda/\mu$ is a nonzero ribbon and $[i,j)$ is the content interval of $\lambda/\mu$.
\end{subequations}
\end{theorem}

\begin{proof}
We apply Theorem~\ref{thm:deformed_boson_fermion} and Lemma~\ref{lem:ribbon}:
\begin{align*}
p_k s_\lambda(\pp\dv\bal;\bbe) & = \bra{\varnothing} e^{H_+(z|\bal;\bbe)} J_{-k}^{(\bal;\bbe)} \ket{\lambda} = \sum_{\nu \in \mcP} \bra{\varnothing} e^{H_+(z|\bal;\bbe)} \ket{\nu} \bra{\nu} J_{-k}^{(\bal;\bbe)} \ket{\lambda}
\\ & = p_k(\bal_{S(\lambda)}/(-\bal_{T(\lambda)})) s_\lambda(\pp\dv\bal;\bbe) + \sum_{\nu \in \mcP} (-1)^{\htf(\nu/\lambda)-1} A_{j,i}^{-k} s_\nu(\pp\dv\bal;\bbe).
\allowdisplaybreaks \\
k\frac{\partial}{\partial p_k} s_\lambda(\pp\dv\bal;\bbe) & = \bra{\varnothing} e^{H_+(z|\bal;\bbe)} J_k^{(\bal;\bbe)} \ket{\lambda} = \sum_{\mu \in \mcP} \bra{\varnothing} e^{H_+(z|\bal;\bbe)} \ket{\mu} \bra{\mu} J_k^{(\bal;\bbe)} \ket{\lambda}
\\ & = p_k(\bbe_{S(\lambda)}/(-\bbe_{T(\lambda)})) s_{\lambda}(\pp\dv\bal;\bbe) + \sum_{\nu \in \mcP} (-1)^{\htf(\lambda/\mu)-1} A_{ij}^k s_{\mu}(\pp\dv\bal;\bbe),
\end{align*}
as desired.
\end{proof}

\begin{example}
Continuing Example~\ref{ex:jkexample}, applying Theorem~\ref{thm:factorialmn} yields
\begin{align*}
p_3 s_{(3,2,1)} &= (\alpha_3^k + \alpha_1^k - \alpha_0^k - \alpha_{-2}^k) s_{(3,2,1)} - A_{2,-1}^{-k} s_{(3,3,3)}  + A_{21}^{-k} s_{(3,3,1)} + A_{0,-1}^{-k} s_{(3,2,2)}
\\ & \hspace{20pt} + \sum_{i=0}^{\infty} A_{4+i,-1}^{-k} s_{(4+i,4,3)} - A_{4+i,1}^{-k} s_{(4+i,4,1)} + A_{4+i,3}^{-k} s_{(4+i,2,1)}
\\ & \hspace{20pt} + \sum_{j=3}^{\infty} (-1)^{j-3} \bigl( A_{2,-j}^{-k} s_{(3,3,3,2,1^{j-3})} - A_{0,-j}^{-k} s_{(3,2,2,2,1^{j-3})} + A_{-2,-j}^{-k} s_{(3,2,1,1,1^{j-3})} \bigr)
\\ & \hspace{50pt} - \sum_{i=0}^{\infty} (-1)^{j-3} A_{4+i,-j}^{-k} s_{(4+i,4,3,2,1^{j-3})},
\\
k \frac{\partial s_{(3,2,1)}}{\partial p_k} &= (\beta_3^k + \beta_1^k - \beta_0^k - \beta_{-2}^k) s_{(3,2,1)} + A_{-2,-1}^k s_{(3,2)} + A_{01}^k s_{(3,1,1)} + A_{23}^k s_{(2,2,1)}
\\ & \hspace{20pt} - A_{03}^k s_{(1,1,1)} - A_{-2,1}^k s_{(3)} + A_{-2,3}^k s_{(1)}.
\end{align*}
where we have written $s_{\lambda} = s_{\lambda}(\pp\dv\bal;\bbe)$ for simplicity.
\end{example}

\begin{example}
\label{ex:pk_expansion}
Taking Theorem~\ref{thm:factorialmn} with $\lambda = \varnothing$, we obtain
\begin{align*}
p_k & =  \sum_{\lambda \in \mcP} (-1)^{\ell(\lambda)-1} A_{\lambda_1,1-\ell(\lambda)}^k s_{\lambda}(\pp\dv\bal;\bbe)
\\ & = \sum_{a=1}^{\infty} \sum_{i=1}^a (-1)^{a-i} A_{i,a-i}^k s_{i1^{a-i}}(\pp\dv\bal;\bbe),
\end{align*}
where the sum in the first line is over all nonempty hook shapes $\lambda$.
Note that the shapes that we sum over for each $p_k$ are all the same, but the coefficients we obtain depend on $k$ (and are such that the set of which are linearly independent).
On the other hand, as per Remark~\ref{rem:finite_sums_current_beta_zero}, these sums become finite when we set $\bbe = 0$:
\begin{align*}
p_1 & = A_{1,0}^{-1} s_{(1)} = h_0(\alpha_0, \alpha_1) s_{(1)} = s_{(1)}, 
\\
p_2 & = A_{1,0}^{-2} s_{(1)} - A_{1,-1}^{-2} s_{(1,1)} + A_{2,0}^{-2} s_{(2)}
\\ & = h_1(\alpha_0, \alpha_1)  s_{(1)} - h_0(\alpha_{-1}, \alpha_0, \alpha_1) s_{(1,1)} + h_0(\alpha_0, \alpha_1, \alpha_2) s_{(2)}
\\ & = (\alpha_0 + \alpha_1) s_{(1)} - s_{(1,1)} + s_{(2)}.
\end{align*}
\end{example}

Note that Theorem~\ref{thm:factorialmn} reduces to the classical Murnaghan--Nakayama rule when $\bal = \bbe = 0$.
This can be seen from the general definitions, but this also follows by noticing that when $\bbe = 0$ the coefficient $A_{ij}^{\pm k}$ is monic and homogeneous of degree $j - i \mp k \geqslant 0$ (as an element of $\CC[\bal]$).

\subsection{Cauchy identity and duality}

We also have a skew (dual) Cauchy identity by following the classical proof using the boson-fermion correspondence.
However, in order to state it, we first need an analog of Definition~\ref{def:dfSchur}.

\begin{definition}[Dual double factorial Schur functions]
For any $\lambda,\mu \in \ZZ^n$, we define the \defn{skew dual double factorial Schur function} by
\[
\widehat{s}_{\lambda/\mu}(\pp\dv\bal;\bbe) := \bra{\lambda} e^{H_-(\pp|\bal;\bbe)} \ket{\nu}.
\]
The \defn{dual double factorial Schur functions} are when $\mu = \varnothing = (0, \dotsc, 0)$.
\end{definition}

We note that the dual double factorial Schur functions also have the analogous branching rules, containment property, and are a basis for $\CC\FPS{\pp}$.

\begin{theorem}[Skew Cauchy identities]
\label{thm:skew_cauchy}
Let $\mu, \nu \in \mcP$.
We have
\begin{subequations}
\begin{align}
\sum_{\lambda\in\mcP} s_{\lambda/\mu}(\pp\dv\bal;\bbe) \widehat{s}_{\lambda/\nu}(\pp'\dv\bal;\bbe) & = e^{\xi(\pp; \pp')} \sum_{\lambda \in \mcP} \widehat{s}_{\mu/\lambda}(\pp'\dv\bal;\bbe) s_{\nu/\lambda}(\pp\dv\bal;\bbe), \label{eq:cauchy}
\\
\sum_{\lambda \in \mcP} s_{\lambda/\mu}(\pp\dv\bal;\bbe) \widehat{s}_{\lambda/\nu}(-\pp\dv\bal;\bbe) & = e^{-\xi(\pp; \pp')} \sum_{\lambda \in \mcP} \widehat{s}_{\mu/\lambda}(-\pp'\dv\bal;\bbe) s_{\nu/\lambda}(\pp\dv\bal;\bbe). \label{eq:dual_cauchy}
\end{align}
\end{subequations}
\end{theorem}

\begin{proof}
As stated, following the classical proof for Schur functions using the boson-fermion correspondence yields the skew Cauchy identity~\eqref{eq:cauchy}:
\begin{align*}
\sum_{\lambda \in \mcP} s_{\lambda/\mu}(\pp\dv\bal;\bbe) \widehat{s}_{\lambda/\nu}(\pp'\dv\bal;\bbe) & = \sum_{\lambda \in \mcP} \bra{\mu} e^{H_+(\pp|\bal;\bbe)} \ket{\lambda} \cdot \bra{\lambda} e^{H_-(\pp'|\bal;\bbe)} \ket{\nu}
\\ & = \bra{\mu} e^{H_+(\pp|\bal;\bbe)} e^{H_-(\pp'|\bal;\bbe)} \ket{\nu}
\\ & = e^{\xi(\pp;\pp')} \bra{\mu} e^{H_-(\pp'|\bal;\bbe)} e^{H_+(\pp|\bal;\bbe)} \ket{\nu}
\\ & = e^{\xi(\pp;\pp')} \sum_{\lambda \in \mcP} \bra{\mu} e^{H_-(\pp'|\bal;\bbe)} \ket{\lambda} \cdot \bra{\lambda} e^{H_+(\pp|\bal;\bbe)} \ket{\nu}
\\ & = e^{\xi(\pp;\pp')} \sum_{\lambda \in \mcP} \widehat{s}_{\mu/\lambda}(\pp'\dv\bal;\bbe) s_{\nu/\lambda}(\pp\dv\bal;\bbe)
\end{align*}
by using our vacuum expectation definitions and their dual versions.

For the skew dual Cauchy identity~\eqref{eq:dual_cauchy}, we replace $e^{H_-(\pp'|\bal;\bbe)}$ with $e^{-H_-(\pp'|\bal;\bbe)}$.
\end{proof}

By following~\cite[Thm.~4.8]{IMS24} and~\cite[Sec.~7]{Yel19} (which is based off~\cite{Warnaar13}), we can obtain skew-Pieri-type formulas for the double factorial Schur functions.
Here, we give the generating functions that lead to such a rule, but to produce the skew Pieri rules explicitly, we need explicit formulas under the specialization $p_k = p_k(z/(-w)) = z^k - (-w)^k$.
As such defer this until Section~\ref{sec:lattice_models}; in particular, see Corollary~\ref{cor:skew_pieri}.

\begin{theorem}[Skew-Pieri generating formulas]
We have
\begin{subequations}
\label{eq:skew_pieri_vertex}
\begin{align}
\sum_{\lambda,\eta} \widehat{s}_{\nu/\eta}(-\pp\dv\bal;\bbe) s_{\lambda/\eta}(\pp'\dv\bal;\bbe) \widehat{s}_{\lambda/\mu}(\pp\dv\bal;\bbe)  & = e^{\xi(\pp;\pp')} s_{\mu/\nu}(\pp'\dv\bal;\bbe),
\\
\sum_{\lambda,\eta} \widehat{s}_{\nu/\eta}(\pp\dv\bal;\bbe) s_{\lambda/\eta}(\pp'\dv\bal;\bbe) \widehat{s}_{\lambda/\mu}(-\pp\dv\bal;\bbe) & = e^{-\xi(\pp;\pp')} s_{\mu/\nu}(\pp'\dv\bal;\bbe),
\end{align}
\end{subequations}
as well as dual versions replacing $s_{\lambda/\mu} \leftrightarrow \widehat{s}_{\lambda/\mu}$.
\end{theorem}

\begin{proof}
Both formulas are proven similarly using the identities
\[
e^{\mp H_-(\pp|\bal;\bbe)} e^{H_+(\pp'|\bal;\bbe)} e^{\pm H_-(\pp|\bal;\bbe)} = e^{\pm \xi(\pp;\pp')} e^{H_+(\pp'|\bal;\bbe)}
\]
from~\eqref{eq:deformed_half_vertex_commutator} and inserting the identity operator $\sum_{\lambda \in \mcP} \ket{\lambda} \cdot \bra{\lambda}$ as in the proof of Theorem~\ref{thm:skew_cauchy}. (See also the proof of~\cite[Thm.~4.8]{IMS24}.)
\end{proof}

In order to understand how our Cauchy identity relates to the Cauchy-type formula of~\cite[Eq.~(22)]{MiyauraMukaihiraFactorial}, we need to describe the dual double factorial Schur functions in terms of the usual double factorial Schur functions.
We first note that we have the generating series
\begin{equation}
\label{eq:dual_fermion_dual_schur}
\begin{aligned}
\sum_{\lambda \in \ZZ^n} \widehat{s}_{\lambda}(\pp\dv\bal;\bbe) & \prod_{i=1}^n (1-\alpha_{\lambda_i-i+1} \beta_{\lambda_i-i+1}) \frac{(w_i;\bal)^{\lambda_i-i}}{(w_i|\bbe)^{\lambda_i-i+1}}
\\ & \hspace{-20pt} = \prod_{i=1}^n e^{\xi(\pp;w_i^{-1})} (w_i;\bal)^{-n} (1 - \alpha_{1-i} \beta_{1-i}) \prod_{i < j} (w_j - w_i) (1 - \alpha_{1-i} \beta_{1-j}) e^{\Lambda_{-n}(\pp|\bal)}
\end{aligned}
\end{equation}
computed analogously to Lemma~\ref{lemma:fermion_expectation} and~\eqref{eq:MMgen_recovery} by considering the dual version of the matrix coefficient ${}_{-n} \bra{\varnothing} \psi^*(w_n|\bal;\bbe) \cdots \psi^*(w_1|\bal;\bbe) e^{H_-(\pp|\bal;\bbe)} \ket{\varnothing}$.
Similarly, we have the analogs of~\eqref{eq:dual_fermion_schur_expansion} and~\eqref{eq:hook_generating}:
\begin{gather}
\label{eq:fermion_dual_schur}
\sum_{\lambda \in \ZZ^n} (-1)^{\abs{\lambda}} \widehat{s}_{\lambda'}(\pp\dv\bal;\bbe) z_i \frac{(z_i|\bbe)^{i-\lambda_i-1}}{(z_i;\bal)^{i-\lambda_i}}
 = \prod_{i=1}^n \frac{e^{-\xi(\pp;z_i^{-1})}}{(z_i;\bal)^n} z_i \prod_{i < j} (z_j - z_i) (1 - \alpha_i \beta_j) e^{\Lambda_n(\pp|\bal)},
\\
\label{eq:dual_hook_generating}
1 + (w - z) \sum_{a, b=0}^{\infty} (-1)^b \widehat{s}_{(a|b)}(\pp\dv\bal;\bbe) (1 - \alpha_{1+a} \beta_{1+a}) \frac{(w;\bal)^{a} (z|\bbe)^{-b-1}}{(w|\bbe)^{a+1} (z;\bal)^{-b}}
 =  e^{\xi(\pp;w^{-1})} e^{-\xi(\pp;z^{-1})},
\end{gather}
respectively, which are constructed from ${}_n \bra{\varnothing} \psi(z_n|\bal;\bbe) \cdots \psi(z_1|\bal;\bbe) e^{H_-(\pp|\bal;\bbe)} \ket{\varnothing}$ and from $\bra{\varnothing} \psi^*(w|\bal;\bbe) \psi(z|\bal;\bbe) e^{H_-(\pp|\bal;\bbe)} \ket{\varnothing}$, respectively.
This leads to analogous generating functions for the corresponding analogs of the homogeneous and elementary symmetric functions, denoted $\widehat{h}_k(\pp\dv\bal;\bbe) := \widehat{s}_k(\pp\dv\bal;\bbe)$ and $\widehat{e}_k(\pp\dv\bal;\bbe) := \widehat{s}_{1^k}(\pp\dv\bal;\bbe)$ respectively.

These generating series are similar to~\eqref{eq:eh_gen_series} but differ slightly in the factors of the form $\frac{1-\alpha_i \beta_i}{1-\alpha_j \beta_j}$; we leave the precise statement for the interested reader.
More specifically, these generating series and~\eqref{eq:omega_eh} imply that
\begin{subequations}
\begin{align}
\widehat{h}_k(\pp\dv\bal;\bbe) & = \frac{1 - \alpha_k \beta_k}{1 - \alpha_0 \beta_0} h_k(\pp\dv\bbe;\bal) = \omega e_k(\pp\dv{-\iota_{\bbe}\bbe};{-\iota_{\bal}\bal}),
\\
\widehat{e}_k(\pp\dv\bal;\bbe) & =  \frac{1 - \alpha_{1-k} \beta_{1-k}}{1 - \alpha_1 \beta_1} e_k(\pp\dv\bbe;\bal) = \omega h_k(\pp\dv{-\iota_{\bbe}\bbe};{-\iota_{\bal}\bal});
\end{align}
\end{subequations}
alternatively, compare~\eqref{eq:dual_hook_generating} with applying $\omega$ to~\eqref{eq:hook_generating} and swapping $\bal \leftrightarrow \bbe$.
We also have Jacobi--Trudi formulas for $\widehat{s}_{\lambda/\mu}(\pp\dv\bal;\bbe)$ analogous to Theorem~\ref{thm:jacobi_trudi} by Wick's theorem.
By comparing the result with Theorem~\ref{thm:jacobi_trudi} and applying Corollary~\ref{cor:involution_skew}, we obtain the following duality relationship.

\begin{theorem}
\label{thm:dual_schur_identity}
We have
\[
\widehat{s}_{\lambda/\mu}(\pp\dv\bal;\bbe) = \omega s_{\lambda'/\mu'}(\pp\dv{-\iota_{\bbe}\bbe};{-\iota_{\bal}\bal})
= \prod_{i=1}^m \frac{1 - \alpha_{i-\lambda'_i} \beta_{i-\lambda'_i}}{1 - \alpha_{i-\mu'_i} \beta_{i-\mu'_i}} s_{\lambda/\mu}(\pp\dv{\bbe},{\bal}).
\]
\end{theorem}

Thus, the dual double factorial Schur functions $\widehat{s}_{\lambda/\mu}(\pp\dv\bal;\bbe)$ as essentially the double factorial Schur functions up to a simple involution of the parameters $\bal, \bbe$ and a simple overall factor involving $\bal, \bbe$.
In the straight shape case $\mu=\varnothing$, this can be seen by comparing the generating formulas~\eqref{eq:MMgen_recovery} with~\eqref{eq:fermion_dual_schur}.
We remark that by using~\eqref{eq:duality_deformed_fields} (and hence Corollary~\ref{cor:adjoint_current}), we could alternatively show Theorem~\ref{thm:dual_schur_identity}.

If we consider the specialization $p_k = p_k(\xx / \bbe)$ and $p_k' = p_k(\yy / \bal)$, then~\eqref{eq:cauchy} with Theorem~\ref{thm:dual_schur_identity} and~\eqref{eq:half_vertex_vacuum_actions} recovers the projective limit (see, \textit{e.g.},~\cite{MolevFactorialSupersymmetric,Molev09} for a precise description) of~\cite[Eq.~(22)]{MiyauraMukaihiraFactorial}.
Likewise, we could restrict to a finite number of variables.
We can also obtain (the projective limit of)~\cite[Eq.~(23)]{MiyauraMukaihiraFactorial} from~\eqref{eq:dual_cauchy} and $-p_k(\yy/\bal) = p_k((-\bal) / (-\yy)) = \omega p_k((-\yy) / (-\bal))$.
See also~\cite[Thm.~5.7]{NaprienkoFFS}.

\subsection{Bosonic construction}

Our next goal is to describe $s_{\lambda}(\pp\dv\bal;\bbe)$ purely using the bosonic Fock space.
We note that all results in this subsection will hold analogously for $\widehat{s}_{\lambda}(\pp\dv\bal;\bbe)$ by using the dual space (or the adjoint fermion fields/vertex operators/etc.).
To achieve this goal, we expand our bosonic vertex operators in terms of shifted modes:
\[
X(z|\bal;\bbe) = \sum_{i \in \ZZ} X_i \cdot \frac{z_i(z_i|\bbe)^{i-1}}{(z_i;\bal)^i},
\qquad\qquad
X^*(w|\bal;\bbe) = \sum_{j \in \ZZ} X_i^* \cdot (1 - \alpha_j \beta_j) \frac{(w_j;\bal)^{j-1}}{(w_j|\bbe)^j}.
\]
Next, we need to compute
\[
\Phi^{(\bal;\bbe)}(\ket{\varnothing}_{\ell}) = \sum_{m \in \ZZ} {}_{(m)} \bra{\varnothing} e^{H_+(\pp|\bal;\bbe)} \ket{\varnothing}_{\ell} \cdot \svar^m = {}_{(\ell)} \braket{\varnothing}{\varnothing}_{\ell} \cdot e^{\Lambda_{\ell}(\pp|\bbe)} \svar^{\ell}.
\]
Hence, we are left to evaluate ${}_{(\ell)}\braket{\varnothing}{\varnothing}_{\ell}$.

\begin{proposition}
\label{prop:shifted_vacuum_pairing}
We have
\[
{}_{(\ell)}\braket{\varnothing}{\varnothing}_{\ell} = \prod_{1 \leqslant i < j \leqslant -\ell} (1 - \alpha_{1-i}\beta_{1-j}) \prod_{1 \leqslant i < j \leqslant \ell} (1 - \alpha_i \beta_j).
\]
\end{proposition}

\begin{proof}
We prove the claim by redoing the computations in Lemma~\ref{lemma:fermion_expectation} and Lemma~\ref{lemma:dual_fermion_expectation} but using the fermion vertex operator realization from Theorem~\ref{thm:fermion_vertex_op}.
Let $\ell > 0$.
Using Theorem~\ref{thm:fermion_vertex_op}, Proposition~\ref{prop:deformed_current_shift_commute}, Corollary~\ref{cor:deformed_current_vacuum}, and~\eqref{eq:deformed_half_vertex_commutator}, we compute
\begin{align*}
\bra{\varnothing} e^{H_+(\pp|\bal;\bbe)} \psi(z_1|\bal;\bbe) & \cdots \psi(z_{\ell}|\bal;\bbe) \ket{\varnothing}_{-\ell}
\\ & = \prod_{i=1}^{\ell} z_i e^{\xi(\pp;z_i)}(z_i|\bbe)^{-\ell} \prod_{i < j} (z_i - z_j) e^{\Lambda_{-\ell}(\pp|\bbe)} {}_{(-\ell)} \braket{\varnothing}{\varnothing}_{-\ell},
\\
\bra{\varnothing} e^{H_+(\pp|\bal;\bbe)} \psi^*(w_1|\bal;\bbe) & \cdots \psi^*(w_{\ell}|\bal;\bbe) \ket{\varnothing}_{\ell}
\\ & = \prod_{i=1}^{\ell} \frac{e^{-\xi(\pp;w_i)}}{(w_i|\bbe)^{\ell}} (1 - \alpha_i \beta_i) \prod_{i < j} (w_i - w_j) e^{\Lambda_{\ell}(\pp|\bbe)} {}_{(\ell)} \braket{\varnothing}{\varnothing}_{\ell}.
\end{align*}
The claim then follows by comparing the above with the formulas from Lemma~\ref{lemma:fermion_expectation} and Lemma~\ref{lemma:dual_fermion_expectation}, respectively.
\end{proof}

Note that at most one of the products in Proposition~\ref{prop:shifted_vacuum_pairing} is not $1$.
Summarizing, we have the following.

\begin{corollary}
We have
\begin{equation}
\label{eq;shifted_vacuum_image}
\Phi^{(\bal;\bbe)}(\ket{\varnothing}_{\ell}) = \prod_{1 \leqslant i < j \leqslant -\ell} (1 - \alpha_{1-i}\beta_{1-j}) \prod_{1 \leqslant i < j \leqslant \ell} (1 - \alpha_i \beta_j) e^{\Lambda_{\ell}(\pp|\bbe)} \svar^{\ell}.
\end{equation}
\end{corollary}

Therefore, from Theorem~\ref{thm:deformed_boson_fermion} and~\eqref{eq;shifted_vacuum_image}, we can write 
\begin{subequations}
\label{eq:bosonic_schur_expansion}
\begin{align}
s_{\lambda}(\pp\dv\bal;\bbe) & = X_{\lambda_1} X_{\lambda_2-1} \cdots X_{\lambda_{\ell}-\ell+1} \cdot \prod_{1 \leqslant i < j \leqslant \ell} (1 - \alpha_{1-i} \beta_{1-j}) e^{\Lambda_{-\ell}(\pp|\bbe)} \svar^{-\ell},
\\ s_{\lambda'}(\pp\dv\bal;\bbe) & =  (-1)^{\abs{\lambda}} X_{1-\lambda_1}^* X_{2-\lambda_2}^* \cdots X_{\ell-\lambda_{\ell}}^* \cdot \prod_{1 \leqslant i < j \leqslant \ell} (1 - \alpha_i \beta_j) e^{\Lambda_{\ell}(\pp|\bbe)} \svar^{\ell},
	\\\label{eq:bosonic_schur_expansion_c} s_{\lambda}(\pp\dv\bal;\bbe) & = (-1)^{m+b_1+\cdots+b_m} X_{-b_1}^* X_{-b_2}^* \cdots X_{-b_m}^* X_{a_m+1} \cdots X_{a_2+1} X_{a_1+1} \cdot 1,
\end{align}
\end{subequations}
for any $\ell \geqslant \ell(\lambda)$, from the definition,~\eqref{eq:conjugate_fermion_vacuum}, and~\eqref{eq:frobenius_fermion_vacuum}, respectively.
Here in (\ref{eq:bosonic_schur_expansion_c}) we are using Frobenius notation.

We can also translate~\eqref{eq:MMgen_recovery} and~\eqref{eq:dual_fermion_schur_expansion} to the bosonic side by computing
\begin{subequations}
\label{eq:boson_vertex_op_vacuum}
\begin{align}
\label{eq:boson_gen_func}
X(z_1|\bal;\bbe) \cdots X(z_n|\bal;\bbe) \cdot e^{\Lambda_{\ell}(\pp|\bbe)} \svar^{\ell} & = \prod_{i < j} (z_i - z_j) \prod_{i=1}^n z_i (z_i|\bbe)^{\ell} e^{\xi(\pp;z_i)} e^{\Lambda_{\ell}(\pp|\bbe)} \svar^{\ell+n},
\\
X^*(w_1|\bal;\bbe) \cdots X^*(w_n|\bal;\bbe) \cdot e^{\Lambda_{\ell}(\pp|\bbe)} \svar^{\ell} & = \prod_{i < j} (w_i - w_j) \prod_{i=\ell-n+1}^{\ell} (1 - \alpha_i \beta_i)
\\ & \hspace{20pt} \times \prod_{i=1}^{\ell} \frac{e^{-\xi(\pp;w_i)}}{(w_i|\bbe)^{\ell}} e^{\Lambda_{\ell}(\pp|\bbe)} \svar^{\ell-n}, \nonumber
\end{align}
\end{subequations}
from applying the BCH formula with~\eqref{eq:xi_function_def}, noting that $\partial_{p_j} \svar^m = 0$ (\textit{cf}.~the bosonic normal ordering discussed in Section~\ref{sec:fermionic_remarks}), and using~\eqref{eq:half_vertex_vacuum_actions}.
Then combining~\eqref{eq:bosonic_schur_expansion} and~\eqref{eq:boson_vertex_op_vacuum}, we obtain~\eqref{eq:MMgen_recovery} and~\eqref{eq:dual_fermion_schur_expansion}.
Note that this is a different vertex operator construction than~\cite[Eq.~(42)]{MiyauraMukaihiraFactorial}, but can be considered a vertex operator description of~\cite[Eq.~(2.20)]{MiyauraMukaihiraGeneralized}.

Let us consider what happens with the term for $\lambda$, obtained using the bosonic representation (equivalently~\eqref{eq:MMgen_recovery}) by taking a multiple contour integral with Proposition~\ref{prop:orthonormality}, after applying~\eqref{eq:boson_gen_func}:
\begin{align*}
s_{\lambda}(\pp\dv\bal;\bbe)
& = \oint X(z_1|\bal;\bbe) \cdots X(z_{\ell}|\bal;\bbe) \cdot \prod_{1 \leqslant i < j \leqslant \ell} (1 - \alpha_{1-i} \beta_{1-j}) e^{\Lambda_{-\ell}(\pp|\bbe)} \svar^{-\ell}
\\ & \hspace{30pt} \times \prod_{i=1}^{\ell} (1 - \alpha_{\lambda_i-i+1} \beta_{\lambda_i-i+1}) \frac{(z_i;\bal)^{\lambda_i-i}}{(z_i|\bbe)^{\lambda_i-i+1}} \frac{dz_i}{2\pi\ii z_i}
\\ & = \oint \prod_{i < j} (z_i - z_j) \prod_{i=1}^{\ell} z_i (z_i|\bbe)^{-\ell} e^{\xi(\pp;z_i)} e^{\Lambda_{-\ell}(\pp|\bbe)} \prod_{1 \leqslant i < j \leqslant \ell} (1 - \alpha_{1-i} \beta_{1-j})
\\ & \hspace{30pt} \times \prod_{i=1}^{\ell} (1 - \alpha_{\lambda_i-i+1} \beta_{\lambda_i-i+1}) \frac{(z_i;\bal)^{\lambda_i-i}}{(z_i|\bbe)^{\lambda_i-i+1}} \frac{dz_i}{2\pi\ii z_i},
\end{align*}
Therefore we have an integral representation of the double factorial Schur function.
We note that this could also be computed from the definition (Definition~\ref{def:dfSchur}) using the fermion fields and extended to the skew case.
There is also an analogous formula for the dual Schur functions.
Next, let us use the generating series~\eqref{eq:h_gen_series} to rewrite the integral as
\begin{align}
s_{\lambda}(\pp\dv\bal;\bbe)
& = \sum_{k_1,\dotsc,k_{\ell}=0}^{\infty} \oint \prod_{i < j} (z_i - z_j) (1 - \alpha_{1-i} \beta_{1-j}) \prod_{i=1}^{\ell} (1 - \alpha_{\lambda_i-i+1} \beta_{\lambda_i-i+1}) (z_i|\bbe)^{-\ell}  \nonumber
\\ & \hspace{70pt} \times \prod_{i=1}^{\ell} h_{k_i}(\pp\dv\sigma_{\bal}^{s_i}\bal;\sigma_{\bbe}^{1-i}\bbe) \frac{(z_i|\sigma_{\bbe}^{-i}\bbe)^{k_i}}{(z_i;\sigma_{\bal}^{s_i}\bal)^{k_i}} \frac{(z_i;\bal)^{\lambda_i-i}}{(z_i|\bbe)^{\lambda_i-i+1}} \frac{dz_i}{2\pi\ii} \nonumber
\allowdisplaybreaks\\ & = \sum_{k_1,\dotsc,k_{\ell}=0}^{\infty} \oint \prod_{i < j} \left(1 - \frac{z_j}{z_i} \right) (1 - \alpha_{1-i} \beta_{1-j}) \prod_{i=1}^{\ell} \frac{(1 - \alpha_{\lambda_i-i+1} \beta_{\lambda_i-i+1}) (z_i;\bal)^{1-i}}{(z_i^{-1};\sigma_{\bbe}^{-\ell}\bbe)^{\ell-i}} \nonumber
\\ & \hspace{70pt} \times \prod_{i=1}^{\ell} h_{k_i}(\pp\dv\sigma_{\bal}^{s_i}\bal;\sigma_{\bbe}^{1-i}\bbe) \frac{(z_i|\sigma_{\bbe}^{-i}\bbe)^{k_i}}{(z_i;\sigma_{\bal}^{s_i}\bal)^{k_i}} \frac{(z_i;\sigma_{\bal}^{1-i} \bal)^{\lambda_i-1}}{(z_i|\sigma_{\bbe}^{-i} \bbe)^{\lambda_i+1}} \frac{dz_i}{2\pi\ii} \nonumber
\allowdisplaybreaks\\ & = \sum_{k_1,\dotsc,k_{\ell}=0}^{\infty} \oint \prod_{i < j} \left(1 - \frac{(1 - \alpha_{1-i} z_i)(z_j - \beta_{1-j})}{(z_i - \beta_{1-j})(1 - \alpha_{1-i} z_j)} \right) \prod_{i=1}^{\ell} (1 - \alpha_{\lambda_i-i+1} \beta_{\lambda_i-i+1}) \nonumber
\\ & \hspace{70pt} \times \prod_{i=1}^{\ell} h_{k_i}(\pp\dv\sigma_{\bal}^{s_i}\bal;\sigma_{\bbe}^{1-i}\bbe) \frac{(z_i|\sigma_{\bbe}^{-i}\bbe)^{k_i}}{(z_i;\sigma_{\bal}^{s_i}\bal)^{k_i}} \frac{(z_i;\sigma_{\bal}^{1-i} \bal)^{\lambda_i-1}}{(z_i|\sigma_{\bbe}^{-i} \bbe)^{\lambda_i+1}} \frac{dz_i}{2\pi\ii}, \label{eq:raising_operator_integral}
\end{align}
where we have used~\eqref{eq:basic_spowers_rels} and Corollary~\ref{cor:finitegeometric} (with $n = 1$).
Recall that the left hand side of~\eqref{eq:h_gen_series} does \emph{not} depend on $\bal$, so we are free to chose the values $s_i$ as desired (within any sum over all $k_i \in \ZZ_{\geqslant0}$).
We can then expand the product over $i < j$ and evaluate the corresponding contour integrals.
We remark that the sums are in fact finite by a similar argument in the Proposition~\ref{prop:orthonormality} by examining the zeros and poles.

Now we consider another alternating sum formula by using the following formalism.
To do so, we introduce the following operator.

\begin{definition}[Raising operator]
Define the shorthands $h_{k,s} := h_k(\pp\dv\sigma_{\bal}^s\bal;\sigma_{\bbe}^s\bbe)$ and $h_{\lambda,\eta} := h_{\lambda_1,-\eta_1} \dotsm h_{\lambda_{\ell},-\eta_{\ell}}$.
Let $\lambda, \eta \in \ZZ^{\ell}$.
The \defn{raising operator} $R_{ij}$ is defined by
\[
R_{ij} h_{\lambda,\eta} = h_{\overline{\lambda},\overline{\eta}},
\text{ where } \overline{\lambda} = (\lambda_1, \dotsc, \lambda_{i-1}, \lambda_i+1, \lambda_{i+1}, \dotsc, \lambda_{j-1},\lambda_j-1,\lambda_{j+1}, \dotsc, \lambda_{\ell}),
\]
and similarly for $\eta$.
\end{definition}

From~\cite[Sec.~2]{Fun12} (noting a minor typo in the Jacobi--Trudi formula, \textit{cf}.~\cite[9th Variation]{Macdonald92}), we have
\begin{equation}
\label{eq:raising_operator_defn}
s_{\lambda}(\pp\dv\bal;\bbe) = \prod_{1 \leqslant i < j \leqslant \ell} (1 - R_{ij}) h_{\lambda,\delta},
\end{equation}
where $\delta = (0, 1, 2, \dotsc, {n-1})$
On the other hand, one method to prove this formula when $\bal = \bbe = 0$ comes from taking the contour integral formula~\eqref{eq:raising_operator_integral}; see, \textit{e.g.},~\cite[Eq.~(1.15)]{Baker96}.
Therefore, it is natural to conjecture we obtain~\eqref{eq:raising_operator_defn} after expanding the product in~\eqref{eq:raising_operator_integral} and evaluating the contour integrals.
However, that turns out to not be the case (for any choices of $s_i$) as an explicit computation when $\ell(\lambda) \geqslant 3$ can show.
For example, see~\cite[Sec.~6]{BHS0}.

\begin{problem}
The raising operator action corresponds to a natural appearance of factors of the form $\frac{1 - \alpha_j z_i}{z_i - \beta_j}$ (or their inverse) in the contour integral formulas.
From the general framework~\eqref{eq:raising_operator_defn}, these sum to the same product in~\eqref{eq:raising_operator_integral}.
It would be good to have an explicit algebraic proof of this identity.
\end{problem}

\subsection{KP hierarchy and 2D Toda lattice}
\label{sec:KP_Toda}

In this subsection, our goal is to prove that the double factorial Schur functions are solutions to the KP hierarchy and that we can construct deformed versions of tau function solutions for the 2D Toda lattice (which includes the modified KP heirarchy as a special case).
Our approach follows standard references (\text{e.g.},~\cite{AlexandrovZabrodin,JM83,KacInfinite}).
As such, we consider the action of the operator
\[
\Omega = \sum_{k\in\ZZ} \psi_k \otimes \psi_k^*
\]
with any group-like element $G$ of $\text{End}(\fermionfock)$
that satisfies the basic bilinearity condition
\begin{equation}
\label{eq:BBC}
\sum_{k \in \ZZ} \psi_k G \otimes \psi_k^* G = \sum_{k \in \ZZ} G \psi_k \otimes G \psi_k^*,
\end{equation}
or equivalently $G \otimes G$ commutes with $\sum_{k \in \ZZ} \psi_k \otimes \psi_k^*$.

The key observation is that Proposition~\ref{prop:orthonormality} implies that (formally) we have
\[
\oint \psi(z|\bal;\bbe) \otimes \psi^*(z|\bal;\bbe) \frac{dz}{2\pi\ii z} = \res_{z=0} z^{-1} \psi(z|\bal;\bbe) \otimes \psi^*(z|\bal;\bbe) = \sum_{k\in\ZZ} \psi_k \otimes \psi_k^*.
\]
Recall the deformed bra ${}_{(m)} \bra{\varnothing} = \bra{\varnothing} \Sigma_{(\bal;\bbe)}^{-m}$ utilized in the deformed boson-fermion correspondence (Theorem~\ref{thm:deformed_boson_fermion}).
Being fully rigorous with the contour integrals, for all $n \geqslant n'$, we compute
\begin{align*}
\oint_{\gamma} {}_{(n+1)}\bra{\varnothing} & e^{H_+(\bt|\bal;\bbe)} \psi(z|\bal;\bbe) G \ket{\varnothing}_{n} \cdot {}_{(n'-1)}\bra{\varnothing} e^{H_+(\bt'|\bal;\bbe)} \psi^*(z|\bal;\bbe) G \ket{\varnothing}_{n'} \frac{dz}{2\pi\ii}
\\ & = \sum_{k\in\ZZ}  {}_{(n+1)}\bra{\varnothing} e^{H_+(\bt|\bal;\bbe)} \psi_k G \ket{\varnothing}_{n} \cdot {}_{(n'-1)}\bra{\varnothing} e^{H_+(\bt'|\bal;\bbe)} \psi^*_k G \ket{\varnothing}_{n'}
\\ & = \sum_{k\in\ZZ}  {}_{(n+1)}\bra{\varnothing} e^{H_+(\bt|\bal;\bbe)} G \psi_k \ket{\varnothing}_{n} \cdot {}_{(n'-1)}\bra{\varnothing} e^{H_+(\bt'|\bal;\bbe)} G \psi^*_k\ket{\varnothing}_{n'} = 0
\end{align*}
by applying~\eqref{eq:BBC} and noting that either $\psi_k \ket{\varnothing}_n = 0$ or $\psi_k^* \ket{\varnothing}_{n'} = 0$ for all $k$.
Note that the contour $\gamma$ must satisfy the conditions of Proposition~\ref{prop:orthonormality}.
We then rewrite the deformed fermion fields in the contour integral using the vertex operator description (Theorem~\ref{thm:fermion_vertex_op}) and apply the BCH formula to obtain for all $n \geqslant n'$
\begin{equation}
\label{eq:bilinear_condition_field}
0 = \oint_{\eta} z^{n-n'} e^{\xi(\bt-\bt';z)} {}_{(n)}\bra{\varnothing} e^{H_+(\bt-[z^{-1}]|\bal;\bbe)} G \ket{\varnothing}_{n} \cdot {}_{(n')}\bra{\varnothing} e^{H_+(\bt'+[z^{-1}]|\bal;\bbe)} G \ket{\varnothing}_{n'} \frac{dz}{2\pi\ii},
\end{equation}
where $\bt \pm [z^{-1}] := (t_1 \pm z^{-1}, t_2 \pm z^{-2}/2, t_3 \pm z^{-3}/3, \cdots)$

Defining a (deformed) \defn{tau function} as
\[
\tau_n(\bt) := {}_{(n)}\bra{\varnothing} e^{H_+(\bt|\bal;\bbe)} G \ket{\varnothing}_n,
\]
then~\eqref{eq:bilinear_condition_field} is the blinear identity (in integral form) that encodes all of the PDE's for the modified KP hierarchy:
\begin{equation}
\label{eq:KP_bilinear_integral}
0 = \oint_{\eta} z^{n-n'} e^{\xi(\bt-\bt';z)} \tau_n(\bt-[z^{-1}]) \tau_{n'}(\bt'+[z^{-1}]) \frac{dz}{2\pi\ii}.
\end{equation}
If we then take a change of variables $t_i \mapsto t_i + x_i$ and $t'_i \mapsto t_i - x_i$ in~\eqref{eq:KP_bilinear_integral} and expanding as a Taylor series around $\xx = 0$ (if we did not change variables, this is around $\bt' - \bt = 0$), we obtain the modified KP hierarchy PDEs (in Hirota form) as the generating sereis
\[
\sum_{k=0}^{\infty} h_k(-2\xx) h_{k+n-n'+1}(\widetilde{\DD}) \exp\left(\sum_{i=1}^{\infty} x_i D_i\right) \tau_n(\bt) \cdot \tau_{n'}(\bt),
\]
where $\widetilde{\DD} = (D_1, 2D_2, 3D_3, \cdots)$ and $D_i$ is the $i$-th Hirota derivative.
For more details of these computations, see, \textit{e.g.},~\cite[Sec.~3.2]{AlexandrovZabrodin}.
The KP heirarchy is the special case when $n = n'$.

\begin{corollary}
\label{cor:KP_tau_solution}
The double factorial Schur functions $s_{\lambda}(\bt\dv\bal;\bbe)$ are tau function solutions to the KP hierarchy.
\end{corollary}

\begin{proof}
We take the group-like element $G = \Sigma_{(\bal;\bbe)}^n \ket{\lambda} \cdot {}_{n}\bra{\varnothing}$ (which is a product of group-like elements; see, \textit{e.g.},~\cite[Prop.~2.5, Sec.~2.8]{AlexandrovZabrodin}), which yields
\[
\tau_n(\bt) = {}_{(n)} \bra{\varnothing} e^{H_+(\bt|\bal;\bbe)} \Sigma_{(\bal;\bbe)}^n \ket{\lambda} \cdot {}_{n}\braket{\varnothing}{\varnothing}_{n}
= \bra{\varnothing} e^{H_+(\bt|\bal;\bbe)} \ket{\lambda} \cdot 1 = s_{\lambda}(\bt\dv\bal;\bbe)
\]
as desired.
\end{proof}

Note that for the modified KP hierarchy with $n' < n$, the group-like element $G$ above yields $0$, but this realizes an embedding of the KP tau functions into the modified KP hierarchy tau functions.
Furthermore, as in~\cite[Sec.~3.3]{AlexandrovZabrodin}, we can compute the expansion of any solution into the double factorial Schur functions by
\[
\tau_n(\bt) = \sum_{\lambda \in \mcP} {}_{(n)}\bra{\varnothing} e^{H_+(\bt|\bal;\bbe)} \ket{\lambda}_{(n)} \cdot {}_{(n)}\bra{\lambda} G \ket{\varnothing}_n = \sum_{\lambda \in \mcP} {}_{(n)}\bra{\lambda} G \ket{\varnothing}_n \cdot s_{\lambda}(\bt\dv\bal;\bbe),
\]
where $\ket{\lambda}_{(n)} := \Sigma_{(\bal;\bbe)}^n \ket{\lambda}$ and ${}_{(n)} \bra{\lambda} := \bra{\lambda} \Sigma_{(\bal;\bbe)}^{-n}$ are used to express the identity operator.

Analogous to the construction above, we can also give a (deformed) tau function solution to the 2D Toda lattice:
\[
\tau_n(\bt_+, \bt_-):= {}_{(n)} \bra{\varnothing} e^{H_+(\bt_+|\bal;\bbe)} G e^{-H_-(\bt_-|\bal;\bbe)} \ket{\varnothing}_{(n)}.
\]
The computations are analogous to those in, \textit{e.g.},~\cite[Sec.~3.2.2]{AlexandrovZabrodin}.
Similarly, we can realize the double factorial Schur functions as 2D Toda lattice solutions by taking $G = \ket{\lambda}_{(n)} \cdot {}_{(n)} \bra{\mu}$ to give
\[
\tau_{n'}(\bt_+, \bt_-) = \delta_{nn'} s_{\lambda}(\bt_+\dv\bal;\bbe) s_{\mu}(\bt_-\dv\bal;\bbe).
\]

\section{Lattice model row transfer matrices}
\label{sec:lattice_models}

In this section, we make an explicit connection between our (deformed) free fermion construction and row transfer matrices of solvable lattice models.
We will briefly describe the relevant aspects of a solvable lattice model, and refer the reader to, \textit{e.g.},~\cite{Baxter89,BBF11,hkice,ZinnJustinTiling} for more precise details.
Then we will prove our main result: the row transfer matrices of the lattice models described in~\cite{NaprienkoFFS} corresponds to acting by $e^{H_{\pm}(\pp|\bal;\bbe)}$ for $\pp$ specialized to powersum supersymmetric functions.
We show this by showing the motion of a single particle and the group-like structure of our deformed current operators (when considered as matrices).

A \defn{row transfer matrix} is a square grid of one horizontal line and $\ZZ$-indexed vertical lines such that each crossing is given a \defn{Boltzmann weight} $b(x/y|\bal; \bbe)$ according to its local configuration of adjacent edges that are assigned
$\{
\begin{tikzpicture}[baseline=-4]\draw[fill=white] (0,0) circle (.15);\end{tikzpicture},
\begin{tikzpicture}[baseline=-4]\draw[fill=black] (0,0) circle (.15);\end{tikzpicture}
\}$.
We call an assignment a \defn{state} if all of the local configurations have nonzero Boltzmann weight, and the Boltzmann weight of a state is the product of all of the local Boltzmann weights with the $j$-th vertical line having parameters $x, y, \alpha_j, \beta_j$.
The \defn{partititon function} of a row transfer matrix is the sum of the Boltzmann weights of all possible states, possibly with fixed boundaries.
The row transfer matrices are then combined to produce a \defn{lattice model}.

We will use two different row transfer matrices $\RTM_{\pm}$ in this section.
Fix parameters $\xx_1 = (x)$ and $\yy_1 = (y)$.
The Boltzmann weights of $\RTM_+(x/y|\bal;\bbe) $ (resp.~$\RTM_-(x/y|\bal;\bbe)$) are given as normalized versions of the weights in Table~\ref{table:model_weights} (resp.\ Table~\ref{table:neg_model_weights}), where the normalization will depend on the column.
When we fix the top and bottom boundary conditions of $\RTM_{\pm}(x/y|\bal;\bbe)$, we can associate these with basis elements $\ket{\lambda}_m$ and ${}_{\ell}\bra{\mu}$, respectively, by having a $+$ (resp.~$-$) at $i$ if $\psi_i \ket{\lambda}_m \neq 0$ (resp.~${}_{\ell} \bra{\mu} \psi_i \neq 0$).
This is equivalent to having a $-$ at every site there is a particle in $\ZZ$.
We write the partition function as the matrix coefficient ${}_{\ell} \bra{\mu} \RTM_{\pm}(x/y|\bal;\bbe) \ket{\lambda}_m$.
It is easy to see that we must have $\ell = m$, and there is a unique state with finitely many $-$ horizontal edges.
Hence the resulting partition function is well-defined as an element in $\CC[\bal; \bbe](x,y)$, hence we can consider $\RTM_{\pm}(x/y|\bal;\bbe)$ as an operator on $\fermionfock_m$.

\begin{table}
\[\begin{array}{|c|c|c|c|c|c|}
\hline
\tt{a}_1 & \tt{a}_2 &
\tt{b}_1 & \tt{b}_2 &
\tt{c}_1 & \tt{c}_2 \\
\hline
\begin{tikzpicture}[scale=.72] 
\draw (0,1) to (2,1); \draw (1,0) to (1,2);
\draw[fill=white] (0,1) circle (.35); \draw[fill=white] (2,1) circle (.35);
\draw[fill=white] (1,0) circle (.35); \draw[fill=white] (1,2) circle (.35);
\path[fill=white] (1,1) circle (.3); \node[scale=.72] at (1,1) {$i,j$};
\end{tikzpicture}
&
\begin{tikzpicture}[scale=.72] 
\draw (0,1) to (2,1); \draw (1,0) to (1,2);
\draw[fill=black] (0,1) circle (.35); \draw[fill=black] (2,1) circle (.35);
\draw[fill=black] (1,0) circle (.35); \draw[fill=black] (1,2) circle (.35);
\path[fill=white] (1,1) circle (.3); \node[scale=.72] at (1,1) {$i,j$};
\end{tikzpicture}
&
\begin{tikzpicture}[scale=.72] 
\draw (0,1) to (2,1); \draw (1,0) to (1,2);
\draw[fill=white] (0,1) circle (.35); \draw[fill=white] (2,1) circle (.35);
\draw[fill=black] (1,0) circle (.35); \draw[fill=black] (1,2) circle (.35);
\path[fill=white] (1,1) circle (.3); \node[scale=.72] at (1,1) {$i,j$};
\end{tikzpicture}
&
\begin{tikzpicture}[scale=.72] 
\draw (0,1) to (2,1); \draw (1,0) to (1,2);
\draw[fill=black] (0,1) circle (.35); \draw[fill=black] (2,1) circle (.35);
\draw[fill=white] (1,0) circle (.35); \draw[fill=white] (1,2) circle (.35);
\path[fill=white] (1,1) circle (.3); \node[scale=.72] at (1,1) {$i,j$};
\end{tikzpicture}
&
\begin{tikzpicture}[scale=.72] 
\draw (0,1) to (2,1); \draw (1,0) to (1,2);
\draw[fill=black] (0,1) circle (.35); \draw[fill=white] (2,1) circle (.35);
\draw[fill=black] (1,0) circle (.35); \draw[fill=white] (1,2) circle (.35);
\path[fill=white] (1,1) circle (.3); \node[scale=.72] at (1,1) {$i,j$};
\end{tikzpicture}
&
\begin{tikzpicture}[scale=.72] 
\draw (0,1) to (2,1); \draw (1,0) to (1,2);
\draw[fill=white] (0,1) circle (.35); \draw[fill=black] (2,1) circle (.35);
\draw[fill=white] (1,0) circle (.35); \draw[fill=black] (1,2) circle (.35);
\path[fill=white] (1,1) circle (.3); \node[scale=.72] at (1,1) {$i,j$};
\end{tikzpicture}
\\\hline
1-\beta x & y+\alpha & 1 + \beta y & x-\alpha & x+y & 1-\alpha\beta \\
\hline
\end{array}\]
\caption{The unscaled Boltzmann weights $\widetilde{b}_+(x/y|\alpha,\beta)$ of the lattice model.}
\label{table:model_weights}
\end{table}

\begin{table}
\[\begin{array}{|c|c|c|c|c|c|}
\hline
\tt{a}_1 & \tt{a}_2 &
\tt{b}_1 & \tt{b}_2 &
\tt{d}_1 & \tt{d}_2 \\
\hline
\begin{tikzpicture}[scale=.72] 
\draw (0,1) to (2,1); \draw (1,0) to (1,2);
\draw[fill=white] (0,1) circle (.35); \draw[fill=white] (2,1) circle (.35);
\draw[fill=white] (1,0) circle (.35); \draw[fill=white] (1,2) circle (.35);
\path[fill=white] (1,1) circle (.3); \node[scale=.72] at (1,1) {$i,j$};
\end{tikzpicture}
&
\begin{tikzpicture}[scale=.72] 
\draw (0,1) to (2,1); \draw (1,0) to (1,2);
\draw[fill=black] (0,1) circle (.35); \draw[fill=black] (2,1) circle (.35);
\draw[fill=black] (1,0) circle (.35); \draw[fill=black] (1,2) circle (.35);
\path[fill=white] (1,1) circle (.3); \node[scale=.72] at (1,1) {$i,j$};
\end{tikzpicture}
&
\begin{tikzpicture}[scale=.72] 
\draw (0,1) to (2,1); \draw (1,0) to (1,2);
\draw[fill=white] (0,1) circle (.35); \draw[fill=white] (2,1) circle (.35);
\draw[fill=black] (1,0) circle (.35); \draw[fill=black] (1,2) circle (.35);
\path[fill=white] (1,1) circle (.3); \node[scale=.72] at (1,1) {$i,j$};
\end{tikzpicture}
&
\begin{tikzpicture}[scale=.72] 
\draw (0,1) to (2,1); \draw (1,0) to (1,2);
\draw[fill=black] (0,1) circle (.35); \draw[fill=black] (2,1) circle (.35);
\draw[fill=white] (1,0) circle (.35); \draw[fill=white] (1,2) circle (.35);
\path[fill=white] (1,1) circle (.3); \node[scale=.72] at (1,1) {$i,j$};
\end{tikzpicture}
&
\begin{tikzpicture}[scale=.72] 
\draw (0,1) to (2,1); \draw (1,0) to (1,2);
\draw[fill=black] (0,1) circle (.35); \draw[fill=white] (2,1) circle (.35);
\draw[fill=white] (1,0) circle (.35); \draw[fill=black] (1,2) circle (.35);
\path[fill=white] (1,1) circle (.3); \node[scale=.72] at (1,1) {$i,j$};
\end{tikzpicture}
&
\begin{tikzpicture}[scale=.72] 
\draw (0,1) to (2,1); \draw (1,0) to (1,2);
\draw[fill=white] (0,1) circle (.35); \draw[fill=black] (2,1) circle (.35);
\draw[fill=black] (1,0) circle (.35); \draw[fill=white] (1,2) circle (.35);
\path[fill=white] (1,1) circle (.3); \node[scale=.72] at (1,1) {$i,j$};
\end{tikzpicture}
\\\hline
1-\alpha x & y+\beta & 1 + \alpha y & x-\beta & 1-\alpha\beta & x+y \\
\hline
\end{array}\]
\caption{The unscaled Boltzmann weights $\widetilde{b}_-(x/y|\alpha,\beta)$ of the dual lattice model.}
\label{table:neg_model_weights}
\end{table}

Let $\widetilde{b}_{\pm}(x/y|\alpha, \beta)$ be given in Table \ref{table:model_weights} and Table~\ref{table:neg_model_weights}.
Due to the normal ordering, our weights are a scalar multiple of these, depending on the column $j$ and the level $m$.

\begin{definition}
Fix $m \in \ZZ$.
The Boltzmann weights in column $j$ are given by
\begin{gather*}
b_j(x/y | \bal; \bbe) = b^{(m)}_j(x/y | \alpha_j, \beta_j)  = \omega_j ^{(m)}(x/y|\beta_j) \widetilde{b}_+(x/y|\alpha_j, \beta_j),
\\
\widehat{b}_j(x/y|\bal; \bbe) = \widehat{b}^{(m)}_j(x/y| \alpha_j, \beta_j) = \omega_j ^{(m)}(x/y|\alpha_j) \widetilde{b}_-(x/y| \alpha_j, \beta_j),
\\
\text{ where }
\omega_j^{(m)}(x/y|\beta) = \begin{cases} \frac{1}{1-\beta x}, & \text{if } j > m, \\ \frac{1}{1+\beta y}, & \text{if } j \leqslant m.
\end{cases}
\end{gather*}
\end{definition}

\begin{example}
\label{ex:row_transfer_matrix}
Suppose that $m = \ell = 0$, $\lambda = (5)$ and $\mu = (2)$.
Then the unique state of the row transfer matrix coefficient $\bra{\mu} \RTM_+(x/y|\bal;\bbe) \ket{\lambda}$ looks like this:
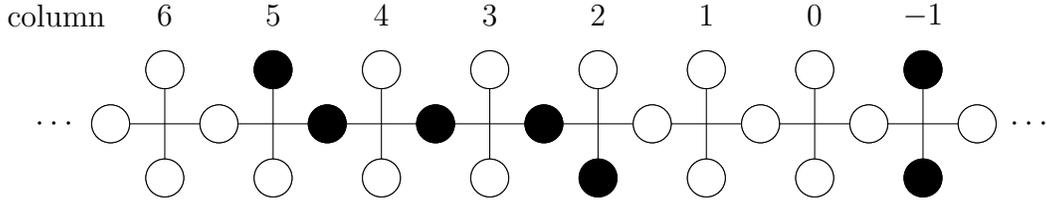
\[\begin{tikzpicture}[scale=.72]
\draw[-] (0,1)--(16,1);
\foreach \i in {1,3,5,...,15} {
  \draw[-] (\i,0)--(\i,2);
  \draw[fill=white] (\i,0) circle (.35);
  \draw[fill=white] (\i,2) circle (.35);
  \draw[fill=white] (\i+1,1) circle (.35);
}
\draw[fill=white] (0,1) circle (.35);
\node at (-1,3) {column};
\foreach \i in {-1,0,1,...,6}
  \node at (13-2*\i,3) {$\i$};

\draw[fill=white] (1,2) circle (.35);
\draw[fill=black] (3,2) circle (.35);
\draw[fill=white] (5,2) circle (.35);
\draw[fill=white] (7,2) circle (.35);
\draw[fill=white] (9,2) circle (.35);
\draw[fill=white] (11,2) circle (.35);
\draw[fill=white] (13,2) circle (.35);
\draw[fill=black] (15,2) circle (.35);
\draw[fill=white] (1,0) circle (.35);
\draw[fill=white] (3,0) circle (.35);
\draw[fill=white] (5,0) circle (.35);
\draw[fill=white] (7,0) circle (.35);
\draw[fill=black] (9,0) circle (.35);
\draw[fill=white] (11,0) circle (.35);
\draw[fill=white] (13,0) circle (.35);
\draw[fill=black] (15,0) circle (.35);
\draw[fill=white] (0,1) circle (.35);
\draw[fill=white] (2,1) circle (.35);
\draw[fill=black] (4,1) circle (.35);
\draw[fill=black] (6,1) circle (.35);
\draw[fill=black] (8,1) circle (.35);
\draw[fill=white] (10,1) circle (.35);
\draw[fill=white] (12,1) circle (.35);
\draw[fill=white] (14,1) circle (.35);
\draw[fill=white] (16,1) circle (.35);

\node at (-1,1) {$\cdots$};
\node at (17,1) {$\cdots$};
\end{tikzpicture}\]
where every column $j > 6$ is a $\tt{a}_1$ vertex (resp.\ $j < -1$ is a $\tt{b}_1$ vertex) with (rescaled) Boltzmann weight $1$.
Thus 
\[
\bra{\mu} \RTM_+(x/y|\bal;\bbe) \ket{\lambda} = \frac{1-\alpha_5\beta_5}{1-\beta_5 x} \frac{x-\alpha_4}{1-\beta_4 x}\frac{x-\alpha_3}{1-\beta_3 x} \frac{x+y}{1-\beta_2 x} \frac{1-\beta_0 x}{1+\beta_0 y}.
\]
\end{example}

Recall the definitions
\[
\Delta_m(k|\bal) = \sum_{0 < j \leqslant m} \alpha_j^k - \sum_{m < j \leqslant 0} \alpha_j^k, \qquad\qquad \Lambda_m(\pp|\bal) = \sum_{k=1}^{\infty} \frac{p_k}{k} \Delta_m(k|\bal),
\]
and so taking the specialization $p_k\mapsto p_k(x/y) = x^k - (-y)^k$ results in
\[
\Lambda_m(x/y|\bal) := \sum_{k=1}^{\infty} \sum_{i=1}^n \frac{x^k - (-y)^k}{k} \Delta_m(k|\bal).
\]

\begin{lemma}
\begin{equation}
{}_m\bra{\varnothing}\RTM_+(x/y|\bal;\bbe)\ket{\varnothing}_m = e^{\Lambda_m(x/y|\bbe)},
\qquad\qquad
{}_m\bra{\varnothing}\RTM_-(x/y|\bal;\bbe)\ket{\varnothing}_m = e^{\Lambda_m(x/y|\bal)}.
\end{equation}
\end{lemma}

\begin{proof}
These partition functions are simply the products of the weights of the vertices in columns between $0$ and $m$.
If $m\geqslant 0$,
\[
{}_m\bra{\varnothing}\RTM_+(x/y|\bal;\bbe)\ket{\varnothing}_m = \prod_{0<j\le m} \frac{1+\beta_j y}{1-\beta_j x}.\] If $m<0$, \[{}_m\bra{\varnothing}\RTM_+(x/y|\bal;\bbe)\ket{\varnothing}_m = \prod_{m<j\leqslant 0} \frac{1-\beta_j x}{1+\beta_j y},
\]
and similarly for $\RTM_-(x/y|\bal;\bbe)$, so the result follows from Lemma~\ref{lemma:exp_lambda}.
\end{proof}



\begin{proposition}
\label{prop:teffect}
If $q\ge p\ge m$ then
\begin{subequations}
\begin{align}\label{eq:pqtform}
\frac{{}_m\bra{p-m} \RTM_+(x/y|\bal;\bbe)\ket{q-m}_m}{e^{\Lambda_{m-1}(x/y|\bbe)}} & = \begin{cases}
\displaystyle \frac{(1-\alpha_q\beta_q)(x+y)}{(1 - \beta_p x)(1 - \beta_q x)} \prod_{p<j<q} \frac{x-\alpha_j}{1-\beta_j x}, & \text{if } q>p, \\
\dfrac{1 + \beta_q y}{1 - \beta_q x}, & \text{if } q=p,
\end{cases}
\\
\label{eq:pqtform_minus}
\frac{{}_m\bra{p-m} \RTM_-(x/y|\bal;\bbe) \ket{q-m}_m}{e^{\Lambda_{m-1}(x/y|\bal)}} & = \begin{cases}
\displaystyle \frac{(1-\alpha_q\beta_q)(x+y)}{(1 - \alpha_p x)(1 - \alpha_q x)} \prod_{q<j<p} \frac{x-\beta_j}{1-\alpha_j x}, & \text{if } p>q, \\
\dfrac{1 + \alpha_q y}{1 - \alpha_q x}, & \text{if } p=q.
\end{cases}
\end{align}
\end{subequations}
\end{proposition}

\begin{proof}
The general case is similar to Example~\ref{ex:row_transfer_matrix} as the six vertex model can be considered as the motion of a particle from $q \to p$.
\end{proof}

Note that we can rewrite the $q > p$ case of~\eqref{eq:pqtform} as $(1 - \alpha_q \beta_q)(x + y) \frac{(x|\sigma_{\bal}^p \bal)^{q-p-1}}{(x|\sigma_{\bbe}^{p-1} \bbe)^{q-p+1}}$, which is almost the integrand in Proposition~\ref{prop:orthonormality}.
We can also rewrite~\eqref{eq:pqtform_minus} into a similar form.

As we will show below, to compute arbitrary matrix coefficients of
$e^{H_+(\bal; \bbe)}$, it is enough to compute them in the case of one
particle. 

\begin{proposition}
\label{prop:heffect}
If $q\geqslant p\geqslant m$, then
\begin{subequations}
\begin{align}
\label{eq:pqhform_plus}
\frac{{}_m \bra{p-m} e^{H_+(x/y|\bal;\bbe)}\ket{q-m}_m}{e^{\Lambda_{m-1}(x/y|\bbe)}} & = \begin{cases}
\displaystyle \frac{(1-\alpha_q\beta_q)(x+y)}{(1 - \beta_p x)(1 - \beta_q x)} \prod_{p<j<q} \frac{x-\alpha_j}{1-\beta_j x}, & \text{if } q>p,
  \\
  \dfrac{1+\beta_p y}{1-\beta_p x}, & \text{if } q=p,
\end{cases}
\\
\label{eq:pqhform_minus}
\frac{{}_m\bra{p-m} e^{H_-(x/y|\bal;\bbe)} \ket{q-m}_m}{e^{\Lambda_{m-1}(x/y|\bal)}} & = \begin{cases}
\displaystyle \frac{(1-\alpha_q\beta_q)(x+y)}{(1 - \alpha_p x)(1 - \alpha_q x)} \prod_{q<j<p} \frac{x-\beta_j}{1-\alpha_j x}, & \text{if } p>q, \\
\dfrac{1 + \alpha_q y}{1 - \alpha_q x}, & \text{if } p=q.
\end{cases}
\end{align}
\end{subequations}
\end{proposition}

\begin{proof}
We only give the proof~\eqref{eq:pqhform_plus} as the proof of~\eqref{eq:pqhform_minus} is analogous.

  Let $\fermionfock_m(q)$ be the finite-dimensional subspace of $\fermionfock_m$ spanned by vectors of the form $\ket{j-m}_m, j\leqslant q$.
  It is clearly invariant under the commutative algebra $\mathfrak{L} = \langle J_k^{(\bal;\bbe)} \mid k>0 \rangle$.
    
  Now
  \[
  J_k^{(\bal;\bbe)} \ket{j-m}_m = J_k^{(\bal;\bbe)} (v_j \wedge \ket{\varnothing}_{m-1})
  = J_k^{(\bal;\bbe)}(v_j) \wedge \ket{\varnothing}_{m-1} + v_j \wedge( J_k^{(\bal;\bbe)} \ket{\varnothing}_{m-1}),
  \]
  and by Corollary \ref{cor:deformed_current_vacuum}, we have
  \[
    v_j \wedge(J_k^{(\bal;\bbe)} \ket{\varnothing}_{m-1}) = \Delta_{m-1}(k|\bbe) \cdot \ket{j-m}_m,
  \]
  Thus with respect to the basis of $\fermionfock_m(q)$, the endomorphism induced by $J_k^{(\bal;\bbe)}$ is
  \[
  \widetilde{J}_k + \Delta_{m-1}(k|\bbe) I_{\fermionfock_m(q)},
  \qquad
    \text{ where }
  \quad
    \widetilde{J}_k = \begin{pmatrix}
       A_{p, p}^k & A^k_{p, p+1} & \cdots & A_{p q}^k\\
       & A_{p+1, p+1}^k & \cdots & A_{p+1, q}^k\\
       &  & \ddots & \vdots\\
       &  &  & A_{q, q}^k
     \end{pmatrix}.
  \]
  Proposition~\ref{prop:aijkmatrix} implies $\widetilde{J}_k \widetilde{J}_{\ell} = \widetilde{J}_{k+\ell}$.
  Thus $H_+(x/y|\bal;\bbe)$ induces on $\fermionfock_m(q)$ the endomorphism
  \[
    \sum_{k=1}^{\infty} \frac{1}{k} (x^k - (-y)^k) \widetilde{J}_1^k
    + \sum_{k=1}^{\infty} \frac{1}{k} (x^k - (-y)^k) \Delta_{m-1}(k|\bbe) I_{\fermionfock_m(q)},
  \] 
  and the second term is simply $\Lambda_{m-1}(x/y|\bbe) I_{\fermionfock_m(q)}$.
  Exponentiating we get
  \[
  {}_m\bra{p-m} e^{H_+ (\bal; \bbe)} \ket{q-m}_m = e^{\Lambda_{m-1}(x/y|\bbe)} \cdot {}_m\bra{p-m} (I - x
     \widetilde{J}_1)^{- 1} (I + y \widetilde{J}_1) \ket{q-m}_m.
  \]
  Now $\bra{p-m} (I - x \widetilde{J}_1)^{- 1} (I + y \widetilde{J}_1) \ket{q-m}_m$ is the upper rightmost entry in the matrix $(I - x \widetilde{J}_1)^{- 1} (I + y \widetilde{J}_1)$, and when $p=q$ this is simply $\frac{1 - \beta_p y}{1 - \beta_p x}$. 
  
For the case $p<q$, we expand out the matrix product:
\[
  (I - x \widetilde{J}_1)^{- 1} (I + y \widetilde{J}_1)
  = (I + y \widetilde{J}_1)\sum_{k=0}^\infty (-1)^k x^k \widetilde{J}_1^k
  = I + (x+y) \sum_{k=1}^\infty x^{k-1} \widetilde{J}_1^k
  = I + (x+y) \sum_{k=1}^\infty x^{k-1} \widetilde{J}_k,
\]
so
\begin{align*}
  \frac{{}_m \bra{p-m} e^{H_+(x/y|\bal;\bbe)}\ket{q-m}_m}{e^{\Lambda_{m-1}(x/y|\bbe)}}
  &= (x+y) \sum_{k=1}^\infty x^{k-1} A^k_{p,q}
  \\&= (x+y)(1 - \alpha_q\beta_q) \sum_{k=1}^\infty x^{k-1} \sum_{r-s = q-p-k} e_r(-\bal_{(p,q)}) h_s(\bbe_{[p,q]})
  \\&= (x+y)(1 - \alpha_q\beta_q) x^{q-p-1} \left(\sum_{r\ge 0} e_r(-\bal_{(p,q)}) x^{-r}\right) \left(\sum_{s\geqslant 0} h_s(\bbe_{[p,q]}) x^s\right)
  \\&= (x+y)(1 - \alpha_q\beta_q) \prod_{p<i<q} (x-\alpha_i) \prod_{p\leqslant j\leqslant q} (1-\beta_jq)^{-1},
\end{align*}
recalling the notation~\eqref{eq:eh_defn}, as desired.
\end{proof}

Let $E_{ij}$ be the elementary endomorphism of $V$ that maps $v_j$ to $v_i$ and annihilates the other basis elements.
From Proposition~\ref{prop:explicit_current}, if we replace $\normord{\psi_i \psi_j^*} \mapsto E_{ij}$
(essentially we are forgetting about the normal ordering), then we can realize the deformed current operators $J_k^{(\bal;\bbe)}$ as linear functions from $\widehat{V} \to \widehat{V}$, where $\widehat{V} := \prod_{i\in\ZZ} \CC v_i$; alternatively as (bi-infinite) matrices.
Under this identification, our goal is to show that the deformed current operators commute as elements in $\End(\widehat{V})$ (\textit{i.e.}, as matrices).

\begin{proposition}
\label{prop:multiplying_linear_currents}
As linear functions in $\End(\widehat{V})$, for all $k, \ell \in \ZZ$, we have
\[
J_k^{(\bal;\bbe)} J_{\ell}^{(\bal;\bbe)} = J_{k+\ell}^{(\bal;\bbe)}.
\]
Moreover, $[J_k^{(\bal;\bbe)}, J_{\ell}^{(\bal;\bbe)}] = 0$ and the inverse of $J_k^{(\bal;\bbe)}$ is $J_{-k}^{(\bal;\bbe)}$.
\end{proposition}

\begin{remark}
\label{rem:mcaveat}
By Proposition~\ref{prop:multiplying_linear_currents} we could say that $J_k^{(\bal;\bbe)} \cdot J_{\ell}^{(\bal;\bbe)} = J_{k+\ell}^{(\bal;\bbe)}$ {\textit{as matrices}}.
However this is misleading in regards to our deformed boson-fermion correspondence.
We have used the natural (associative) algebra structure of the linear operators in $\End(\CC^{\infty})$ in order to define the bracket operation and hence a Lie algebra structure.
Yet it is not true that a representation of $\End(\CC^{\infty})$ as a Lie algebra respects the multiplication of matrices (\textit{i.e.}, that it is an algebra morphism).
In particular, as endomorphisms of the fermionic Fock space $\fermionfock$, the elements $J_k^{(\bal;\bbe)} \circ J_{\ell}^{(\bal;\bbe)}$ and $J_{k + \ell}^{(\bal;\bbe)}$ are not equal.
Indeed, the defining property of being a Lie algebra representation only requires the action of $[J_k^{(\bal;\bbe)}, J_{\ell}^{(\bal;\bbe)}] = J_{k+\ell}^{(\bal;\bbe)} - J_{\ell+k}^{(\bal;\bbe)} = 0$ and $J_k^{(\bal;\bbe)} \circ J_{\ell}^{(\bal;\bbe)} - J_{\ell}^{(\bal;\bbe)} \circ J_k^{(\bal;\bbe)} = 0$ to be equal.
\end{remark}

To prove Proposition~\ref{prop:multiplying_linear_currents}, by matrix multiplication it is equivalent to show the following property of the $A_{ij}^k$ coefficients.
In the proof, we will need another variation of the formal $\delta$ function defined by
\[
\delta(z - w) := z^{-1} \sum_{k \in \ZZ} \left( \frac{z}{w} \right)^n,
\]
We note that $\delta(z - w) = \delta(w - z) = z^{-1} \delta(w/z)$ (see, \textit{e.g.},~\cite[Ch.~2]{KacVertex}) and it sometimes written as $\delta(w,z)$.

\begin{proposition}
\label{prop:aijkmatrix}
Assume that
\[
\frac{(z;\bal)^n (\zeta|\bbe)^n}{(z|\bbe)^n (\zeta;\bal)^n} \xrightarrow{n\to\infty} 0,
\qquad
\text{and} \qquad \frac{(z|\iota_{\bbe}\bbe)^n (\zeta;\iota_{\bal}\bal)^n}{(z;\iota_{\bal}\bal)^n (\zeta|\iota_{\bbe}\bbe)^n} \xrightarrow{n\to\infty} 0.
\]
Then for all $k, \ell, p, q \in \ZZ$, we have
\begin{equation}
\label{eq:aijkmatrix}
\sum_{j \in \ZZ} A_{p, j}^k A_{j, q}^{\ell} = A_{p, q}^{k + \ell}.
\end{equation}
\end{proposition}

\begin{proof}
By using~\eqref{eq:a_int}, we have
\begin{align*}
\sum_j A_{p, j}^k A_{j, q}^{\ell} & = \sum_{j \in \ZZ} \oint_{\gamma} z^{1-k} (1-\alpha_j\beta_j) \frac{(z|\sigma^p\bal)^{j-p-1}}{(z;\sigma^{p-1}\bbe)^{j-p+1}} \frac{dz}{2\pi\ii z} \oint_{\eta} \zeta^{1-\ell} (1-\alpha_q\beta_q) \frac{(\zeta|\sigma^j\bal)^{q-j-1}}{(\zeta;\sigma^{j-1}\bbe)^{q-j+1}} \frac{d\zeta}{2\pi\ii \zeta}
\\ & = \oint_{\gamma} \oint_{\eta} z^{1-k} \zeta^{1-\ell} (1-\alpha_q\beta_q) \frac{(\zeta|\bal)^{q-1} (z; \bbe)^{p-1}}{(\zeta; \bbe)^q (z|\bal)^p} \sum_{j \in \ZZ} (1-\alpha_j\beta_j) \frac{(z|\bal)^{j-1} (\zeta; \bbe)^{j-1} }{ (z; \bbe)^j (\zeta|\bal)^j} \frac{d\zeta}{2\pi\ii \zeta} \frac{dz}{2\pi\ii z}
\end{align*} 
As in the proof of Lemma~\ref{lem:finitegeometric} and Proposition~\ref{prop:vacuum_deformed_fields},
\[
(z-\zeta) (1-\alpha_j\beta_j)\frac{(z;\bal)^{j-1} (\zeta|\bbe)^{j-1}}{(z|\bbe)^j (\zeta;\bal)^j} = \frac{(z;\bal)^{j-1} (\zeta|\bbe)^{j-1}}{(z | \bbe)^{j-1} (\zeta;\bal)^{j-1}} - \frac{(z;\bal)^j (\zeta|\bbe)^j}{(z | \bbe)^j (\zeta;\bal)^j},
\]
so
\[
(z-\zeta)\sum_{0<j\le n}(1-\alpha_j\beta_j)\frac{(z;\bal)^{j-1} (\zeta|\bbe)^{j-1}}{(z|\bbe)^j (\zeta;\bal)^j} = 1 - \frac{(z;\bal)^n (\zeta|\bbe)^n}{(z | \bbe)^n (\zeta;\bal)^n}.
\]
Rearranging and taking the limit, we have
\begin{align*}
\sum_{0<j}(1-\alpha_j\beta_j)\frac{(z;\bal)^{j-1} (\zeta|\bbe)^{j-1}}{(z|\bbe)^j (\zeta;\bal)^j} & = \frac{z^{-1}}{1-\zeta/z}\lim_{n\to\infty}\left(1 - \frac{(z;\bal)^n (\zeta|\bbe)^n}{(z | \bbe)^n (\zeta;\bal)^n)}\right)
\\ & = \frac{z^{-1}}{1-\zeta/z} = z^{-1} + \zeta/z^2 + \cdots,
\end{align*}
using the relevant assumption.
Similarly,
\[(\zeta-z)\sum_{-n<j\le 0}(1-\alpha_j\beta_j)\frac{(z;\bal)^{j-1} (\zeta|\bbe)^{j-1}}{(z|\bbe)^j (\zeta;\bal)^j} = 1 - \frac{(z | \iota_{\bbe}\bbe)^n (\zeta;\iota_{\bal}\bal)^n}{(z;\iota_{\bal}\bal)^n (\zeta|\iota_{\bbe}\bbe)^n},
\]
and using the relevant assumption,
\[
\sum_{j\le 0}(1-\alpha_j\beta_j)\frac{(z;\bal)^{j-1} (\zeta|\bbe)^{j-1}}{(z|\bbe)^j (\zeta;\bal)^j} = \frac{\zeta^{-1}}{1-z/\zeta} = \zeta^{-1} + z/\zeta^2 + \cdots.
\]
Summing these, we get
\[
\sum_{j\in\ZZ}(1-\alpha_j\beta_j)\frac{(z;\bal)^{j-1} (\zeta|\bbe)^{j-1}}{(z|\bbe)^j (\zeta;\bal)^j} 
= \delta(\zeta - z).
\]
Plugging this back in,
\begin{align*}
\sum_j A_{p, j}^k A_{j, q}^{\ell} &= \oint_{\gamma} \oint_{\eta} z^{-k} \zeta^{1-\ell} (1-\alpha_q\beta_q) \frac{(\zeta|\bal)^{q-1} (z; \bbe)^{p-1}}{(\zeta; \bbe)^q (z|\bal)^p} \delta(\zeta - z) \frac{d\zeta}{2\pi\ii} \frac{dz}{2\pi\ii z}
\\&= \oint_{\gamma} z^{-k} z^{1-\ell} (1-\alpha_q\beta_q) \frac{(z|\bal)^{q-1} (z; \bbe)^{p-1}}{(z; \bbe)^q (z|\bal)^p} \frac{dz}{2\pi\ii z}
\\&= \oint_{\gamma} z^{1-k-\ell} (1-\alpha_q\beta_q) \frac{(z|\sigma^p\bal)^{q-p-1}}{(z; \sigma^{p-1}\bbe)^{q-p+1}} \frac{dz}{2\pi\ii z}
\\&= A_{p,q}^{k+\ell},
\end{align*}
where we have used the standard fact of formal residues of the $\delta$ distribution $\operatorname{Res}_{z=0} \delta(\zeta - z) f(\zeta) = f(z)$ (see, \textit{e.g.},~\cite[Prop.~3.8]{BN06}).
\end{proof}

As a consequence of Proposition~\ref{prop:aijkmatrix}, we have the following symmetric function identities.

\begin{corollary}
Let $p,q,k,\ell\in\ZZ$, and let $d = q-p-k-\ell$.

\begin{subequations}
If $k,\ell > 0$,
\begin{align}
\sum_{r-s+t-u = d} (1 - \alpha_j\beta_j) e_r(-\bal_{(p,j)}) h_s(\bbe_{[p,j]}) e_t(-\bal_{(j,q)}) h_u(\bbe_{[j,q]}) = \sum_{r-s = d} e_r(-\bal_{(p,q)}) h_s(\bbe_{[p,q]});
\end{align} 
if $k,\ell < 0$,
\begin{align}
\sum_{r-s+t-u = d} (1 - \alpha_j\beta_j) e_r(-\bbe_{(j,p)}) h_s(\bal_{[j,p]}) e_t(-\bbe_{(q,j)}) h_u(\bal_{[q,j]}) = \sum_{r-s = d} e_r(-\bbe_{(q,p)}) h_s(\bal_{[q,p]});
\end{align}
if $k>0>\ell$,
\begin{align}
\sum_{r-s+t-u = d} (1 - \alpha_j\beta_j)  & e_r(-\bal_{(p,j)}) h_s(\bbe_{[p,j]}) e_t(-\bbe_{(q,j)})  h_u(\bal_{[q,j]})
\\ &= \begin{cases} \notag
  \sum_{r-s = d} e_r(-\bal_{(p,q)}) h_s(\bbe_{[p,q]}) & \text{if } k+\ell > 0, \\
  \sum_{r-s = d} e_r(-\bbe_{(q,p)}) h_s(\bal_{[q,p]}) & \text{if }  k+\ell < 0, \\ 
  \frac{\delta_{pq}}{1-\alpha_q\beta_q} & \text{if } k+\ell = 0;
\end{cases}
\end{align}
\end{subequations}
where on the left side of every equation $j = r-s+p+k = q-\ell-t+u$.
\end{corollary}

\begin{theorem}
\label{thm:texph}
As operators on $\fermionfock$, we have
\[
\RTM_{\pm}(x/y|\bal;\bbe) = e^{H_{\pm}(x/y|\bal;\bbe)}.
\]
\end{theorem}

\begin{proof}
We show the claim by proving
\begin{equation}
\bra{\mu} e^{H_{\pm}(x/y|\bal;\bbe)} \ket{\lambda}_m = \bra{\mu} \RTM_{\pm}(x/y|\bal;\bbe) \ket{\lambda}_m
\end{equation}
for all partitions $\lambda$ and $\mu$ as the matrix coefficients uniquely determine the operator (see, \textit{e.g.},~\cite[Prop.~2.2]{AlexandrovZabrodin}).
Given $\lambda$ and $\mu$, we choose an integer $k$ such that $\lambda_{k+1}=\mu_{k+1}=0$.
Recall that
\begin{align*}
\ket{\lambda}_m & = \psi_{\lambda_1+m} \psi_{\lambda_2+m-1} \dotsm \psi_{\lambda_k+m-k+1} \ket{\varnothing}_{m-k}
\\ & = v_{\lambda_1+m}\wedge v_{\lambda_2+m-1} \wedge\cdots.
\end{align*}

By~\cite[Prop.~2.4, (2.47)]{AlexandrovZabrodin}, $e^{H_{\pm}(x/y|\bal;\bbe)}$ satisfies Wick's theorem: 
\[
\bra{\mu} e^{H_{\pm}(x/y|\bal;\bbe)} \ket{\lambda}_{m-k} = \det_{i,j\leqslant k} \bra{\mu_i+k-i+1} e^{H_{\pm}(x/y|\bal;\bbe)} \ket{\lambda_j+k-j+1}_{m-k}.
\] 
On the other hand, by the free fermionic LGV Lemma~\cite[Prop.~A.2]{NaprienkoFFS}, we have
\begin{equation}
\label{eq:ttwopoint}
\bra{\mu} \RTM_{\pm}(x/y|\bal;\bbe) \ket{\lambda}_{m-k} = \det_{i,j\leqslant k} \bra{\mu_i+k-i+1} \RTM_{\pm}(x/y|\bal;\bbe) \ket{\lambda_j+k-j+1}_{m-k}.
\end{equation}
See also Molev~\cite[Thm.~3.1]{MolevFactorialSupersymmetric}, where the LGV Lemma is used to prove a result equivalent to~\eqref{eq:ttwopoint} in the case where $\bbe = 0$.
It is thus sufficient to prove that if $q\geqslant p\geqslant m$, then
\[
\bra{q-m} \RTM_{\pm}(x/y|\bal;\bbe) \ket{p-m}_m= \bra{q-m} e^{H_{\pm}(x/y|\bal;\bbe)} \ket{p-m}_m,
\]
which follows by comparing Propositions~\ref{prop:teffect} and~\ref{prop:heffect}.
\end{proof}

Now we consider the case of multiple variables.
Fix some positive integer $n$.
We define
\begin{equation}
	\label{eq:multirow_transfer_defn}
	\RTM_{\pm}(\xx_n/\yy_n | \bal; \bbe) := \RTM_{\pm}(x_n/y_n | \bal; \bbe) \RTM_{\pm}(x_{n-1}/y_{n-1} | \bal; \bbe) \cdots \RTM_{\pm}(x_1/y_1 | \bal; \bbe).
\end{equation}
Let $f_{\lambda/\mu}(\xx/\yy\dv\bal;\bbe)$ and $\widehat{f}_{\lambda/\mu}(\xx/\yy\dv\bal;\bbe)$ be the free fermionic Schur functions and their duals, respectively, from~\cite{NaprienkoFFS}.

\begin{corollary}
	\label{cor:free_fermion_schur_recovery}
	We have 
	\begin{align*}
		\bra{\mu} \RTM_+(\xx_n/\yy_n | \bal; \bbe) \ket{\lambda} & = \widehat{f}_{\lambda/\mu}(\xx_n/\yy_n\dv\bbe;\bal) = s_{\lambda/\mu}(\xx_n/\yy_n\dv\bal;\bbe),
		\\
		\bra{\lambda} \RTM_-(\xx_n/\yy_n | \bal; \bbe) \ket{\mu} & = f_{\lambda/\mu}(\xx_n/\yy_n\dv\bbe;\bal) = \widehat{s}_{\lambda/\mu}(\xx_n/\yy_n\dv\bal;\bbe),
	\end{align*} 
\end{corollary}

\begin{proof}
As in the branching rule (Proposition~\ref{prop:branching}), we have
\[
e^{H_{\pm}(\pp+\pp'|\bal;\bbe)} = e^{H_{\pm}(\pp|\bal;\bbe) + H_{\pm}(\pp'|\bal;\bbe)} = e^{H_{\pm}(\pp|\bal;\bbe)} e^{H_{\pm}(\pp'|\bal;\bbe)},
\]
and hence the claim follows from Theorem~\ref{thm:texph} and~\eqref{eq:multirow_transfer_defn}, along with comparing our lattice model with that of~\cite{NaprienkoFFS}.
\end{proof}

As a consequence of Corollary~\ref{cor:MM_schur} and Corollary~\ref{cor:free_fermion_schur_recovery}, we see that the dual free fermionic Schur functions from~\cite{NaprienkoFFS} are the generalized Schur functions from~\cite{MiyauraMukaihiraGeneralized,MiyauraMukaihiraFactorial} under the specialization $p_k = p_k(\xx/\yy)$.
Moreover, Corollary~\ref{cor:free_fermion_schur_recovery} means we have provided alternative proofs for many of the results in~\cite{NaprienkoFFS}:
\begin{itemize}
\item Theorem~\ref{thm:dual_schur_identity} recovers~\cite[Thm.~5.14]{NaprienkoFFS} since $p_k(\yy / \xx) = (-1)^{k-1} p_k(\xx/\yy) = \omega p_k(\xx/\yy)$ and Corollary~\ref{cor:involution}.
Additionally, the duality in Theorem~\ref{thm:dual_schur_identity} corresponds to the duality between the lattice models described in~\cite[Prop.~4.10]{NaprienkoFFS}.
\item Equations~\eqref{eq:hook_generating} and~\eqref{eq:dual_hook_generating} are~\cite[Cor.~5.18]{NaprienkoFFS} with mapping $w \mapsto -w$ and $z \mapsto -z$, respectively, together with using~\eqref{eq:exp_xi_specialization} and
\[
\frac{(-w;\bal)^{-b-1}}{(-w|\bbe)^{-b}} = \frac{(-w|\iota_{\bbe}\bbe)^b}{(-w;\iota_{\bal}\bal)^{b+1}} = (-1)^b \frac{(w|{-\iota_{\bbe}\bbe})^b}{(w;{-\iota_{\bal}\bal})^{b+1}}.
\]
\item Theorem~\ref{thm:jacobi_trudi} and Theorem~\ref{thm:giambelli} is~\cite[Thm.~5.9]{NaprienkoFFS}.
\item Theorem~\ref{thm:skew_cauchy} is~\cite[Prop.~5.15]{NaprienkoFFS}.
\item Proposition~\ref{prop:branching} is~\cite[Lemma~5.5]{NaprienkoFFS}.
\item Equation~\eqref{eq:eh_gen_series} with the equivalent specializations of~\eqref{eq:dual_fermion_dual_schur} and~\eqref{eq:fermion_dual_schur} yield~\cite[Cor.~5.20]{NaprienkoFFS} (noting minor misprints in the formulas~\cite[Eq.~(5.10), (5.11)]{NaprienkoFFS}).
\end{itemize}

\begin{remark}
	By Wick's theorem, the values $\bra{i-m} e^H \ket{j-m}$ uniquely determine all other matrix coefficients.
	Then we could pose the following problem: Given a sequence of values $h_{i,j}$, find the Hamiltonian $H$ such that $\bra{i-m} e^H \ket{j-m}$, where the Hamiltonian has the form $H = \sum_{i,j}t_{i,j}E_{i,j}$ for some coefficients $t_{i,j}$.
	In the case when $h_{i,j}$ are the complete factorial supersymmetric functions, we obtain~\eqref{eq:jkexpansion} as the unique solution. 
\end{remark}

From the lattice model description (Proposition~\ref{prop:teffect}), we have an explicit description for the specialization $p_k = p_k(\xx_1 / \yy_1)$.
Recall from the Murnaghan--Nakayama rule (Theorem~\ref{thm:factorialmn}) that a ribbon has an associated content interval.

\begin{corollary}[{\cite[Prop.~4.8, Ex.~5.2]{NaprienkoFFS}}]
\label{cor:single_variable}
If $\lambda / \mu$ contains a $2 \times 2$ block, then
\[
s_{\lambda/\mu}(\xx_1/\yy_1 \dv \bal; \bbe) = \widehat{s}_{\lambda/\mu}(\xx_1/\yy_1 \dv \bal; \bbe) = 0;
\]
that is, if $\lambda / \mu$ is not a disjoint union of ribbons.
Otherwise
\begin{subequations} \label{eq:schur_ribbon_formulas}
\begin{align}
s_{\lambda/\mu}(\xx_1/\yy_1\dv\bal;\bbe) & = \prod_{k=1}^{\ell} \frac{1 - \beta_{k-\lambda'_k} x}{1 + \beta_{1-k} y} \frac{1 + \beta_{\lambda_k-k+1} y}{1 - \beta_k x} \prod_R \wt_{(\bal;\bbe)}(R),
\\
\widehat{s}_{\lambda/\mu}(\xx_1/\yy_1\dv\bal;\bbe) & = \prod_{k=1}^{\ell} \frac{1 - \alpha_{k-\lambda'_k} x}{1 + \alpha_{1-k} y} \frac{1 + \alpha_{\lambda_k-k+1} y}{1 - \alpha_k x} \prod_R \widehat{\wt}_{(\bal;\bbe)}(R),
\end{align}
\end{subequations}
where we take the products over all ribbons $R$ in $\lambda / \mu$ and
\begin{align*}
\wt_{(\bal;\bbe)}(R) & = \frac{(1 - \alpha_j \beta_j) (x + y)}{(1 - \beta_i x)(1 + \beta_j y)} \prod_{\bbb \in R} \begin{cases} \frac{x - \alpha_{c(\bbb)}}{1 - \beta_{c(\bbb)} x} & \text{if $\bbb$ has a left neighbor}, \\ \frac{y + \alpha_{c(\bbb)}}{1 + \beta_{c(\bbb)} y} & \text{if $\bbb$ has a bottom neighbor}, \end{cases}
\\
\widehat{\wt}_{(\bal;\bbe)}(R) & = \frac{(1 - \alpha_i \beta_i) (x + y)}{(1 - \alpha_i x)(1 + \alpha_j y)} \prod_{\bbb \in R} \begin{cases} \frac{x - \beta_{c(\bbb)}}{1 - \alpha_{c(\bbb)} x} & \text{if $\bbb$ has a left neighbor}, \\ \frac{y + \beta_{c(\bbb)}}{1 + \alpha_{c(\bbb)} y} & \text{if $\bbb$ has a bottom neighbor}, \end{cases}
\end{align*}
where $[i, j)$ is the content interval of $R$.
\end{corollary}

A consequence of Corollary~\ref{cor:single_variable} and the branching rule (Proposition~\ref{prop:branching}), we obtain a tableaux formula for the (skew) double factorial Schur functions.
In~\cite[Eq.~(26)]{MiyauraMukaihiraFactorial}, a similar formula using reverse column-strict tableaux was given.

\begin{corollary}
The double factorial Schur function $s_{\lambda/\mu}(\xx_n/\yy_n\dv\bal;\bbe)$ and the dual double factorial Schur functions $\widehat{s}_{\lambda/\mu}(\xx_n/\yy_n\dv\bal;\bbe)$ are given as a sum over reverse plane partitions\footnote{A reverse plane partition is a filling of the Young diagram such that the entries weakly increase along rows and columns. See, \textit{e.g.},~\cite{ECII} for plane partitions (but with the inequalities reversed).} with entries $1, \dotsc, n$ such that the restriction to $i$, for each $1 \leqslant i \leqslant n$, is a disjoint union of ribbons and has weight given by \eqref{eq:schur_ribbon_formulas}.
\end{corollary}

Now we will give the skew-Pieri formulas as we can give explicit contour integral formulas under the specialization $p_k = p_k(x/y)$.
This is given explicitly by Corollary~\ref{cor:single_variable} taken together with Theorem~\ref{thm:dual_schur_identity} (\textit{cf}.~\cite[Prop.~4.9]{NaprienkoFFS}).

\begin{corollary}[Skew-Pieri formulas]
\label{cor:skew_pieri}
We have
\begin{subequations} \label{eq:skew-pieri}
\begin{align}
h_k(\pp'\dv\bal;\bbe)s_{\mu/\nu}(\pp'\dv\bal;\bbe) &= \sum_{\lambda,\eta} c_{k,\mu/\nu}^{\lambda/\eta} s_{\lambda/\eta}(\pp'\dv\bal;\bbe),
\\
e_k(\pp'\dv\bal;\bbe) s_{\mu/\nu}(\pp'\dv\bal;\bbe) &= (-1)^k \frac{1 - \alpha_1 \beta_1}{1 - \alpha_{1-k}\beta_{1-k}} \sum_{\lambda,\eta} \overline{c}_{k,\mu/\nu}^{\lambda/\eta} s_{\lambda/\eta}(\pp'\dv\bal;\bbe),
\end{align}
\end{subequations}
where the sums are over all partitions $\lambda,\eta$ such that $\lambda/\mu$ and $\nu/\eta$ are ribbons, and
\begin{subequations} \label{eq:skew-pieri-coeffs}
\begin{align}
c_{k,\mu/\nu}^{\lambda/\eta} = (1 - \alpha_k \beta_k)\oint \widehat{s}_{\nu/\eta}(\beta_0/(-z)\dv\bal;\bbe)  \widehat{s}_{\lambda/\mu}(z/(-\beta_0)\dv\bal;\bbe) \frac{(z; \bal)^{k-1}}{(z|\sigma_{\bbe}^{-1}\bbe)^{k+1}} \frac{dz}{2\pi\ii},
\\
\overline{c}_{k,\mu/\nu}^{\lambda/\eta} = (1 - \alpha_k \beta_k)\oint \widehat{s}_{\nu/\eta}(z/(-\beta_1)\dv\bal;\bbe) \widehat{s}_{\lambda/\mu}(\beta_1/(-z)\dv\bal;\bbe) \frac{(z|\sigma_{\bbe}\bbe)^{-k-1}}{(z; \bal)^{1-k}} \frac{dz}{2\pi\ii}.
\end{align}
\end{subequations}
\end{corollary}

\begin{proof}
Starting from~\eqref{eq:skew_pieri_vertex}, we specialize $p_k = p_k(z/w) = z^k - (-w)^k$ with $w = -\beta_0$ or $w = -\beta_1$ to obtain
\begin{subequations}
\begin{align}
e^{\xi(\pp; z)} e^{-\xi(\pp; \beta_0)} s_{\mu/\nu}(\pp'\dv\bal;\bbe) &= \sum_{\lambda,\eta} \widehat{s}_{\nu/\eta}(\beta_0/(-z)\dv\bal;\bbe) s_{\lambda/\eta}(\pp'\dv\bal;\bbe) \widehat{s}_{\lambda/\mu}(z/(-\beta_0)\dv\bal;\bbe),
\\
e^{-\xi(\pp; z)} e^{\xi(\pp; \beta_1)}  s_{\mu/\nu}(\pp'\dv\bal;\bbe) &= \sum_{\lambda,\eta} \widehat{s}_{\nu/\eta}(z/(-\beta_1)\dv\bal;\bbe) s_{\lambda/\eta}(\pp'\dv\bal;\bbe) \widehat{s}_{\lambda/\mu}(\beta_1/(-z)\dv\bal;\bbe),
\end{align}
\end{subequations}
respectively.
Note that we have used the fact $-p_k(z/w) = p_k((-w)/(-z))$.
Next, apply~\eqref{eq:eh_gen_series}, and in the second equation, replace $z$ with $-z$, to obtain
\begin{subequations}
\begin{align}
\begin{aligned}
\sum_{k=0}^{\infty} h_k(\pp'\dv\bal;\bbe) & s_{\mu/\nu}(\pp'\dv\bal;\bbe) \frac{(z|\sigma_{\bbe}^{-1} \bbe)^k}{(z;\bal)^k}
\\ & =
\sum_{\lambda,\eta} \widehat{s}_{\nu/\eta}(\beta_0/(-z)\dv\bal;\bbe) \widehat{s}_{\lambda/\mu}(z/(-\beta_0)\dv\bal;\bbe) s_{\lambda/\eta}(\pp'\dv\bal;\bbe),
\end{aligned}
\allowdisplaybreaks \\
\begin{aligned}
\sum_{k=0}^{\infty} (-1)^k e_k(\pp'\dv\bal;\bbe) & s_{\mu/\nu}(\pp'\dv\bal;\bbe) \frac{1 - \alpha_{1-k}\beta_{1-k}}{1 - \alpha_1 \beta_1} \frac{(z;\bal)^{-k}}{(z|\sigma_{\bbe}^{-1}\bbe)^{-k}}
\\ & =
\sum_{\lambda,\eta} \widehat{s}_{\nu/\eta}(z/(-\beta_1)\dv\bal;\bbe) \widehat{s}_{\lambda/\mu}(\beta_1/(-z)\dv\bal;\bbe) s_{\lambda/\eta}(\pp'\dv\bal;\bbe).
\end{aligned}
\end{align}
\end{subequations}
Hence, $c_{k,\mu/\nu}^{\lambda/\eta}$ (resp.\ $\overline{c}_{k,\mu/\nu}^{\lambda/\eta}$) in~\eqref{eq:skew-pieri} is the coefficient of $\frac{(z|\sigma_{\bbe}^{-1}\bbe)^k}{(z;\bal)^k}$ (resp.\ of $\frac{(z;\bal)^{-k}}{(z|\sigma_{\bbe}\bbe)^{-k}}$) in
\[
\widehat{s}_{\lambda/\mu}(\beta_0/(-z)\dv\bal;\bbe)  \widehat{s}_{\nu/\eta}(z/(-\beta_0)\dv\bal;\bbe)
\qquad
\text{(resp. } \widehat{s}_{\lambda/\mu}(z/(-\beta_1)\dv\bal;\bbe) \widehat{s}_{\nu/\eta}(\beta_1/(-z)\dv\bal;\bbe) 
\text{)}.
\]
By Corollary~\ref{cor:single_variable}, we must have that $\lambda/\mu$ and $\nu/\eta$ are ribbons else the corresponding term is $0$. 
Recall from Proposition \ref{prop:orthonormality} that
\[
(1 - \alpha_k \beta_k)\oint_{\eta} \frac{(z|\bbe)^{k-1}}{(z;\bal)^k} \frac{(z; \bal)^{n-1}}{(z|\bbe)^n} \frac{dz}{2\pi\ii}
= \delta_{nk}.
\]
Applying this and noting $(z|\sigma_{\bbe}^{-1}\bbe)^k = (z-\beta_0)(z|\bbe)^{k-1}$ and $(z|\sigma_{\bbe}\bbe)^{-k} = (z-\beta_1)^{-1} (z|\bbe)^{1-k}$ (\textit{cf}.~\cite[Eq.~(2.12)]{MiyauraMukaihiraGeneralized}), we obtain $c_{k,\mu/\nu}^{\lambda/\eta}$ and $\overline{c}_{k,\mu/\nu}^{\lambda/\eta}$ as claimed.
\end{proof}

Similar formulas hold for the dual double factorial Schur functions; we leave the details to the interested reader.

\begin{problem}
Determine an explicit combinatorial formula for the coefficients~\eqref{eq:skew-pieri-coeffs} and for the dual double factorial Schur function Pieri rule.
\end{problem}

These integrals appear to be difficult to evaluate, even in the case $\nu = \varnothing$.
Indeed, in the case $\nu = \varnothing$, then we force $\eta = \varnothing$, reducing the skew Pieri rule to a (straight shape) Pieri rule, but the desired coefficients do not explicitly appear in the formulas.
It is possible that~\eqref{eq:skew-pieri-coeffs} can be computed using computations similar to those in the proof of Proposition~\ref{prop:deformed_shift_identities}.

We are also able to extend Corollary~\ref{cor:skew_pieri} to give a skew Littlewood--Richardson rule for $s_{\kappa}(\pp'\dv\bal;\bbe) s_{\mu/\nu}(\pp'\dv\bal;\bbe)$ by instead substituting $p_k = p_k(\zz_{\ell} /(-\bbe_{[1-\ell,0]}))$, where $\ell \geqslant \ell(\kappa)$ by using the generating~\eqref{eq:MMgen_recovery}.

\begin{corollary}[Skew-Littlewood--Richardson formula]
\label{cor:skew_LR}
We have
\begin{equation*}
s_{\kappa}(\pp'\dv\bal;\bbe)s_{\mu/\nu}(\pp'\dv\bal;\bbe) = \sum_{\substack{\lambda \supseteq \mu \\ \eta \subseteq \nu}} c_{k,\mu/\nu}^{\lambda/\eta} s_{\lambda/\eta}(\pp'\dv\bal;\bbe),
\end{equation*}
where
\begin{equation}
\begin{aligned}
c_{\kappa,\mu/\nu}^{\lambda/\eta} = \prod_{i=1}^{\ell} (1 - \alpha_{\kappa_i} \beta_{\kappa_i}) & \oint \widehat{s}_{\nu/\eta}(\bbe_{[1-\ell,0]}/(-\zz_{\ell}) \dv\bal;\bbe)  \widehat{s}_{\lambda/\mu}(\zz_{\ell}/(-\bbe_{[1-\ell,0]})\dv\bal;\bbe)
\\ & \hspace{10pt} \times \prod_{i<j} (z_i - z_j)(1 - \alpha_{1-i}\beta_{1-j}) \prod_{i=1}^{\ell} \frac{(z_i; \sigma_{\bal}^{1-i} \bal)^{k-1}}{(z_i|\sigma_{\bbe}^{-i}\bbe)^{k+1}} \frac{dz_i}{2\pi\ii}.
\end{aligned}
\end{equation}
A similar formula holds for the dual double factorial Schur functions by using the double factorial Schur functions in the contour integral.
\end{corollary}

\section{Deformed Miwa parameters}
\label{sec:deformed_powersums}

In this section, we consider an alternative deformed boson-fermion correspondence that produces a different (but isomorphic) representation of $\mcH$ on $\bosonfock$.
In Section~\ref{sec:deformed_correspondence}, we used the deformed shift operators to construct deformed (dual) shifted vacuums, but here we will instead use the normal shift operators.
This will lead to a deformation of the $p_i$ variables (known as Miwa variables/parameters), which are typically set to the powersum (super)symmetric functions for applications to symmetric functions and lattice models \textit{a la} Section~\ref{sec:lattice_models}.
Recall that $Q = s \partial_s$.

\begin{theorem}[The deformed boson-fermion correspondence via normal shifted vacuums]
\label{thm:deformed_boson_fermion_II}
There is an isomorphism of $\mcH$ algebras $\Theta^{(\bal;\bbe)} \colon \fermionfock \to \bosonfock$ defined by
\[
\ket{\eta} \longmapsto \sum_{m \in \ZZ} {}_m \bra{\varnothing} e^{H_+(\pp|\bal;\bbe)} \ket{\eta} \cdot \svar^m,
\]
for any $\ket{\eta} \in \fermionfock$.
The map $\Theta^{(\bal;\bbe)}$ induces a mapping of actions given by, for $k > 0 $,
\begin{gather*}
J_k^{(\bal;\bbe)} \mapsto k \frac{\partial}{\partial p_k},
\quad
J_{-k}^{(\bal;\bbe)} \mapsto p_k + \Delta_Q(k|\bal),
\quad
J_0^{(\bal;\bbe)} \mapsto Q,
\quad
\Sigma_{(\bal;\bbe)} \mapsto e^{\xi(\partial_{\pp}; \alpha_Q)} s e^{-\Lambda_Q(\alpha_{Q+1}|\bbe)},
\\
\begin{aligned}
\psi(z|\bal;\bbe) & \mapsto Y(z|\bal;\bbe) := e^{\widehat{H}_-(z|\bal;\bbe)} z^Q s e^{-\widehat{H}_+(z^{-1}|\bal;\bbe)+\xi(\widetilde{\partial}_{\pp}; \alpha_{Q+1})-\Lambda_Q(\alpha_{Q+1}|\bbe)},
\\
\psi^*(w|\bal;\bbe) & \mapsto Y^*(w|\bal;\bbe) := e^{-\widehat{H}_-(w|\bal;\bbe)} \svar^{-1} w^{-Q} (1-\alpha_Q\beta_Q) e^{\widehat{H}_+(w^{-1}|\bal;\bbe)-\xi(\widetilde{\partial}_{\pp}; \alpha_Q)+\Lambda_{Q-1}(\alpha_Q|\bbe)},
\end{aligned}
\end{gather*}
where $\widehat{H}_{\pm}(z|\bal;\bbe)$ is the image of $H_{\pm}(z|\bal;\bbe)$ under $\Theta^{(\bal;\bbe)}$.
\end{theorem}

\begin{proof}
The proof of the image of the deformed current operators is analogous to the proof of Theorem~\ref{thm:deformed_boson_fermion}.
For example, let $k > 0 $, and we compute
\[
{}_m \bra{\varnothing} e^{H_+(\pp|\bal;\bbe)} J_{-k}^{(\bal;\bbe)} \ket{\eta}
= {}_m \bra{\varnothing} (p_k + J_{-k}^{(\bal;\bbe)}) e^{H_+(\pp|\bal;\bbe)} \ket{\eta}
= \bigl( p_k + \Delta_m(k|\bal) \bigr) \cdot {}_m \bra{\varnothing} e^{H_+(\pp|\bal;\bbe)} \ket{\eta}.
\]
We also provide an alternative proof for $J_k^{(\bal;\bbe)}$ for $k > 0$ using generating series:
\begin{align*}
{}_m \bra{\varnothing} e^{H_+(\pp|\bal;\bbe)} J_k^{(\bal;\bbe)} e^{H_-(\pp'|\bal;\bbe)} \ket{\varnothing}_m & = \bigl( p'_k + \Delta_m(k|\bbe) \bigr) \cdot {}_m \bra{\varnothing} e^{H_+(\pp|\bal;\bbe)} e^{H_-(\pp'|\bal;\bbe)} \ket{\varnothing}_m
\\ & = \bigl( p'_k + \Delta_m(k|\bbe) \bigr) e^{\Lambda_m(\pp'|\bal)} e^{\xi(\pp;\pp')} e^{\Lambda_m(\pp|\bbe)}
\\ & = k \frac{\partial}{\partial p_k} e^{\Lambda_m(\pp'|\bal)} e^{\xi(\pp;\pp')} e^{\Lambda_m(\pp|\bbe)},
\end{align*}
and comparing this to ${}_m \bra{\varnothing} e^{H_+(\pp|\bal;\bbe)}e^{H_-(\pp'|\bal;\bbe)} \ket{\varnothing}_m = e^{\Lambda_m(\pp'|\bal)} e^{\xi(\pp;\pp')} e^{\Lambda_m(\pp|\bbe)}$ yields the image of $J_k^{(\bal;\bbe)}$.

Next, we compute the image of the deformed shift operator.
Using generating series, we compute using Proposition~\ref{prop:deformed_current_shift_commute} and Corollay~\ref{cor:deformed_shift_vacuum}:
\begin{equation}
\label{eq:corr_deformed_shift}
{}_m \bra{\varnothing} e^{H_+(\pp|\bal;\bbe)} \Sigma_{(\bal;\bbe)} e^{H_-(\pp'|\bal;\bbe)} \ket{\varnothing}_{m-1} \cdot \svar^m = e^{\Lambda_m(\pp'|\bal)} e^{\xi(\pp;\pp')} e^{\Lambda_{m-1}(\pp|\bbe)} \cdot \svar^m.
\end{equation}
Let $S_{(\bal;\bbe)}$ denote the image of $\Sigma_{(\bal;\bbe)}$.
Note that the BCH formula yields
\begin{equation}
\label{eq:BCH_ket_image}
e^{\pm\xi(\widetilde{\partial}_{\pp}; z^{-1})} e^{\Lambda_m(\pp|\bbe)} \svar^m = e^{\pm\Lambda_m(z^{-1}|\bbe)} e^{\Lambda_m(\pp|\bbe)} e^{\pm\xi(\widetilde{\partial}_{\pp}; z^{-1})}.
\end{equation}
Thus, we compare~\eqref{eq:corr_deformed_shift} with
\[
S_{(\bal;\bbe)} \cdot {}_{m-1} \bra{\varnothing} e^{H_+(\pp|\bal;\bbe)} e^{H_-(\pp'|\bal;\bbe)} \ket{\varnothing}_{m-1} \cdot \svar^{m-1} = S_{(\bal;\bbe)} \cdot e^{\Lambda_{m-1}(\pp'|\bal)} e^{\xi(\pp;\pp')} e^{\Lambda_{m-1}(\pp|\bbe)} \cdot \svar^{m-1}
\]
to conclude
\[
S_{(\bal;\bbe)} = e^{\xi(\widetilde{\partial}_{\pp}; \alpha_Q)} s e^{-\Lambda_Q(\alpha_{Q+1}|\bbe)}
\]
by using the fact as a differential operator $e^{\xi(\widetilde{\partial}_{\pp}; \alpha_Q)} f(p_1, p_2, \cdots) = f(p_1+\alpha_Q, p_2+\alpha_Q^2, \cdots)$ for any function $f$ by its Taylor expansion (see, \textit{e.g.},~\cite[\S14.5]{KacInfinite}).

The image of the deformed fermion fields follows from taking the image~\eqref{eq:fermion_vertex_ops} under $\Theta^{(\bal;\bbe)}$.
\end{proof}

We make a few remarks regarding the deformed version of the boson-fermion correspondence in Theorem~\ref{thm:deformed_boson_fermion_II}.
First is that we can define $\bal$-deformed versions of the powersums by
\[
p_k^{(m|\bal)} := p_k + \Delta_m(k|\bal) = p_k + \sum_{0 < j \leqslant m} \alpha_j^k - \sum_{m < j \leqslant 0} \alpha_j^k
\]
for each $m \in \ZZ$.
Note that $p_k^{(0|\bal)} = p_k^{(m|0)} = p_k$.
However, there is an inherent asymmetry in our boson-fermion correspondence, which is reflected in the fact that we only deformed the powersums using $\bal$.
Although there is an analogous deformation that can be done on the differential side, it does not come from the deformed current operators.

This can also be seen as a manifestation of the image of $\Sigma_{(\bal;\bbe)}$ no longer simply being $s$, which agrees with the fact that $s$ does not commute with the image of $J_{-k}^{(\bal;\bbe)}$.
Indeed, we note that
\begin{equation}
\label{eq:deformed_powersum_s_comm}
[p_k + \Delta_Q(k|\bal), \svar] = \alpha_{Q+1}^k,
\qquad\qquad
[p_k + \Delta_Q(k|\bal), \svar^{-1}] = -\alpha_Q^k,
\end{equation}
On the other hand, we can also follow the alternative proof following Remark~\ref{rem:alt_proof_bfc_kac} to compute the image of the deformed fermion fields.
In more detail, we first perform a linear change of variables $u_k = p_k + \Delta_Q(k|\bal)$ and note that
\begin{subequations}
\begin{gather}
\label{eq:deformed_Heisenberg_relations}
\partial_{u_k} = \partial_{p_k},
\qquad\qquad
[u_k, \partial_{u_j}] = [u_k, \partial_{p_j}] = [p_k, \partial_{p_j}] = \delta_{kj},
\\ \label{eq:uY_commutator}
[u_k, Y(z|\bal;\bbe)] = z^{-k} Y(z|\bal;\bbe),
\end{gather}
\end{subequations}
where~\eqref{eq:uY_commutator} is the image of Proposition~\ref{prop:current_comm} under $\Theta^{(\bal;\bbe)}$.
In particular, the images of the deformed current operators still satisfy the Heisenberg algebra relations.
Therefore, we can now follow~\cite[Thm.~14.10(a)]{KacInfinite} and use~\cite[Lemma~14.5]{KacInfinite} to see that under $\Theta^{(\bal;\bbe)}$
\begin{align*}
Y(z|\bal;\bbe) & = 
e^{\xi(\pp;z)} e^{\Lambda_Q(z|\bal)} C_Q(z) e^{\xi(\widetilde{\partial}_{\pp}; \alpha_Q)} \svar e^{-\xi(\widetilde{\partial}_{\pp}; z^{-1})},
\\
Y^*(z|\bal;\bbe) & = 
e^{-\xi(\pp;w)} e^{-\Lambda_Q(w|\bal)} \svar^{-1} e^{-\xi(\widetilde{\partial}_{\pp}; \alpha_Q)} C_Q^*(w) e^{\xi(\widetilde{\partial}_{\pp}; w^{-1})},
\end{align*}
where 
$C_Q(z)$ and $C^*_Q(z)$ are constants (with respect to $\uu$) to be determined.
This is obtained by substituted back, but the extra factors $e^{\pm\xi(\widetilde{\partial}_{\pp}; \alpha_Q)}$ comes from the fact the variables used change based on the charge (\textit{i.e.}, the result of $Q$ changes) or to account for~\eqref{eq:deformed_powersum_s_comm}.
Next, from Lemma~\ref{lemma:fermion_expectation} and Lemma~\ref{lemma:dual_fermion_expectation}, respectively, we have
\begin{subequations}
\label{eq:fermion_field_deformed_image}
\begin{align}
\Theta^{(\bal;\bbe)}(\psi(z|\bal;\bbe) \ket{\varnothing}_m) & = e^{\xi(\pp;z)} \frac{z (z|\bbe)^m}{(z;\bal)^{m+1}} e^{\Lambda_m(\pp|\bbe)} \svar^{m+1},
\\
\Theta^{(\bal;\bbe)}(\psi^*(w|\bal;\bbe) \ket{\varnothing}_m) & = e^{-\xi(\pp;w)} \frac{(w;\bal)^{m-1}}{(w|\bbe)^m} (1 - \alpha_m \beta_m) e^{\Lambda_m(\pp|\bbe)} \svar^{m-1},
\end{align}
\end{subequations}
and note that $\Theta^{(\bal;\bbe)}(\ket{\varnothing}_m) = e^{\Lambda_m(\pp|\bbe)} \svar^m$.
Next, we apply~\eqref{eq:half_vertex_vacuum_actions} and~\eqref{eq:BCH_ket_image} in order to compare~\eqref{eq:fermion_field_deformed_image} with $Y(z|\bal;\bbe) \cdot e^{\Lambda_m(\pp|\bbe)} \svar^m$ and $Y^*(w|\bal;\bbe) \cdot e^{\Lambda_m(\pp|\bbe)} \svar^m$, from which we see
\begin{align*}
C_Q(z) & = z^Q e^{-\Lambda_{Q-1}(\alpha_Q|\bbe)},
&
C_Q^*(w) & = w^{-Q} e^{\Lambda_Q(\alpha_Q|\bbe)} = (1 - \alpha_Q \beta_Q) w^{-Q} e^{\Lambda_{Q-1}(\alpha_Q|\bbe)},
\end{align*}
where the last equality for $C_Q^*(w)$ used~\eqref{eq:half_vertex_vacuum_actions} and $e^{-\Lambda_Q(\alpha_Q|\bbe)} = (1 - \alpha_Q \beta_Q) e^{-\lambda_{Q-1}(\alpha_Q|\bbe)}$.
Hence, we could also have obtained the image of the deformed shift operator by comparing the image of Theorem~\ref{thm:fermion_vertex_op} under $\Theta^{(\bal;\bbe)}$ with $Y(z|\bal;\bbe)$ computed as above.

Next, if we want to consider the image of $\Sigma$, this will map to $s$ that also has a twisted multiplication that applies the shifts $\sigma_{\bal} \sigma_{\bbe}$ by Proposition~\ref{prop:deformed_current_classical_shift}.
If we tried to do the computations in Section~\ref{sec:KP_Toda} with the deformed boson-fermion correspondence $\Theta^{(\bal;\bbe)}$, a straightforward computation shows that each PDE in the corresponding hierarchy would comprise of infinite sums of differential operators.
Lastly, we sketch how one could arrive at the vertex operators~\cite[Eq.~(39)]{MiyauraMukaihiraFactorial} with specializing their $\overline{\mathbf{t}} = 0$.
The idea is to consider $Y(z|\bal;\bbe)$ from Theorem~\ref{thm:deformed_boson_fermion_II} but as a charge $0$ operator, but the operator would need to account for the factors coming from the change in the charge.

\bibliographystyle{habbrv}
\bibliography{focktorial}

\end{document}